\documentclass[twoside,11pt]{article}

\usepackage{lastpage}
\usepackage{jmlr2e}
\usepackage{amssymb, amsmath, mathtools} 
\usepackage{graphicx, float} 
\usepackage{chngcntr} 
\usepackage{hyperref} 
\usepackage{bbm} 
\usepackage{indentfirst} 
\usepackage{microtype} 
\usepackage{tikz}
\usetikzlibrary{matrix}
\usepackage{graphicx}
\usepackage{booktabs} 
\usepackage{xparse}   
\NewDocumentCommand{\rot}{O{45} O{1em} m}{\makebox[#2][l]{\rotatebox{#1}{#3}}}%
\usepackage{tabularx}

\usepackage{xcolor} 
\usepackage{subcaption}
\usepackage{enumerate}
\usepackage{enumitem}
\usepackage{comment}
\usepackage{thmtools} 
\usepackage{algorithm, algorithmicx, algpseudocode}
\usepackage{cuted}
\usepackage{arydshln}
\usepackage{dsfont}
\usepackage{ulem} 
\usepackage{subcaption}
\graphicspath{ {../figures/} }

\renewcommand{\P}{\mathbb{P}} 
\newcommand{\Prob}{\P}
\newcommand{\E}{\mathbb{E}} 
\newcommand{\Var}{\mathrm{Var}} 
\newcommand{\Cov}{\mathrm{Cov}} 
\newcommand{\Cor}{\mathrm{Cor}} 
\newcommand{\vect}{\mathrm{vec}}
\newcommand\indep{\protect\mathpalette{\protect\independenT}{\perp}} \def\independenT#1#2{\mathrel{\rlap{$#1#2$}\mkern2mu{#1#2}}} 
\newcommand{\1}{\mathbbm{1}} 
\newcommand{\R}{\mathbb{R}} 
\newcommand{\N}{\mathbb{N}} 
\renewcommand{\O}{\mathcal{O}} 
\renewcommand{\o}{o}
\newcommand{\codite}{CoDiTE}
\DeclarePairedDelimiter\norm{\lVert}{\rVert} 
\newcommand\myeq{\mkern2.5mu{=}\mkern2.5mu} 
\renewcommand\mid{\mkern4mu{|}\mkern4mu} 

\newcommand{\F}{\ensuremath{{\mathcal F}}}
\renewcommand{\H}{\ensuremath{{\mathcal H}}}
\newcommand{\A}{\ensuremath{{\mathcal A}}}
\newcommand{\I}{\ensuremath{{\mathcal I}}}

\newcommand{\HS}{\ensuremath{{\mathcal S}}}
\newcommand{\HSz}{\ensuremath{{\mathcal S}_0}}
\newcommand{\HSo}{\ensuremath{{\mathcal S}_1}}

\newcommand{\one}{\mathds{1}}

\newcommand{\forestass}[1]{\textbf{(F#1)}}
\newcommand{\dataass}[1]{\textbf{(D#1)}}
\newcommand{\kernelass}[1]{\textbf{(K#1)}}

\newcommand{\Ybf}{\mathbf{Y}}
\newcommand{\Yibf}{\mathbf{Y}_i}
\newcommand{\ybf}{\mathbf{y}}

\newcommand{\Xbf}{\mathbf{X}}
\newcommand{\xbf}{\mathbf{x}}

\newcommand{\Xibf}{\mathbf{X}_i}

\newcommand{\Zbf}{\mathbf{Z}}

\newcommand{\Zcal}{\mathcal{Z}}
\newcommand{\PYgX}{\P_{\Ybf \mid \Xbf \myeq \xbf}}
\newcommand{\PYgXz}{\P_{\Ybf \mid \Xbf \myeq \xbf}^0}
\newcommand{\PYgXo}{\P_{\Ybf \mid \Xbf \myeq \xbf}^1}
\newcommand{\hPYgX}{\hat\P_{\Ybf \mid \Xbf \myeq \xbf}}

\newcommand{\wix}{w_i(\xbf)}
\newcommand{\hwix}{\hat w_i(\xbf)}
\newcommand{\kYi}{k(\Yibf, \cdot)}

\newcommand{\hmun}{\hat\mu_n(\xbf)}
\newcommand{\hmunz}{\hat\mu_{n_0,0}(\xbf)}
\newcommand{\hmuno}{\hat\mu_{n_1,1}(\xbf)}
\newcommand{\hmunHS}{\hat{\mu}_n^{\HS}(\xbf)}
\newcommand{\hmunHSz}{\hat{\mu}_{n_0,0}^{\HS_0}(\xbf)}
\newcommand{\hmunHSo}{\hat{\mu}_{n_1,1}^{\HS_1}(\xbf)}
\newcommand{\Kbf}{\mathbf{K}}
\newcommand{\kbf}{\mathbf{k}}

\newcommand{\diracYi}{\delta_{\Yibf}}

\newcommand{\diracY}{\delta_{\Ybf}}

\newcommand{\Lcal}{\mathcal{L}}
\newcommand{\wbf}{\mathbf{w}}
\newcommand{\hw}{\hat w}
\newcommand{\hwbf}{\hat\wbf}
\newcommand{\hwbfz}{\hat\wbf_0}
\newcommand{\hwbfo}{\hat\wbf_1}

\newcommand{\bandybf}{\mathcal{B}(\ybf)}

\newtheorem{corollarytwo}[theorem]{Corollary}

\newenvironment{claim}[1]{\par\noindent\underline{Claim:}\space#1}{}
\newenvironment{claimproof}[1]{\par\noindent\underline{Proof:}\space#1}{\hfill $\square$} 

\ShortHeadings{
Uncertainty Assessment for DRF}{Näf, Emmenegger, Bühlmann, and Meinshausen}

\firstpageno{1}

\begin{document}

\title{
Confidence and Uncertainty Assessment for Distributional Random Forests
}

\author{\name Jeffrey N\"{a}f \email jeffrey.naf@inria.fr\\
\addr Inria, PreMeDICaL Team, University of Montpellier\\
       34000 Montpellier,\\
       France
       \AND
       \name Corinne Emmenegger \email emmenegger@stat.math.ethz.ch 
       \AND
       \name Peter B\"{u}hlmann \email buhlmann@stat.math.ethz.ch 
       \AND 
       \name Nicolai Meinshausen \email meinshausen@stat.math.ethz.ch\\
       \addr Seminar for Statistics ETH Zurich\\
       8092 Zurich,\\
       Switzerland}

\editor{Boaz Nadler}

\maketitle

\begin{abstract}%
The Distributional Random Forest (DRF) is a 
recently introduced Random Forest algorithm to 
estimate 
multivariate conditional distributions. 
Due to its general estimation procedure, it can be employed to estimate a wide range of targets such as conditional average treatment effects, conditional quantiles, and conditional correlations. 
However, only results about the consistency and convergence rate of the DRF prediction are available so far. 
We characterize the asymptotic distribution of DRF and develop a bootstrap approximation of it. This allows  
us to derive inferential tools for quantifying standard errors and the construction of confidence regions that have asymptotic coverage guarantees.
In simulation studies, we empirically validate the developed theory for inference of low-dimensional targets and for testing 
distributional differences between two populations. 
\end{abstract}


\section{Introduction}

Building on Random Forests~\citep{breiman2001random}, Distributional Random Forests (DRF)~\citep{DRF-paper} provide nonparametric estimates of the distribution 
of a multivariate response, conditional on potentially  many covariates. 
DRF estimates a locally adaptive Hilbert space embedding $\hmun$ of a multivariate conditional distribution $\PYgX$ of a variable of interest $\Ybf=(Y_1, Y_2, \ldots, Y_d)^T \in \mathbb{R}^{d}$ given covariates $\Xbf=(X_1, X_2, \ldots, X_p)^T \in \mathbb{R}^{p}$.
More precisely, in a reproducing kernel Hilbert space (RKHS) with reproducing kernel $k$ and associated Hilbert space $\H$,  
DRF computes the estimator 
\begin{align} \label{eq:muhat-intro}
    \hmun = \sum_{i=1}^n \hwix \kYi
\end{align}
of the conditional mean embedding (CME)
$\mu(\xbf)=\E[k(\Ybf, \cdot) \mid \Xbf=\xbf]$ of $\PYgX$.
The weights $\hwix$ quantify the relevance of each training data point $\xbf_i$ to predict $\mu(\xbf)$, which makes DRF locally adaptive. 
\citet{DRF-paper} established consistency of $\hmun$ at a fixed test point $\xbf$. 
A natural, but more challenging, question is whether an asymptotic normality result can be formulated for $\hmun$. 
Providing such a result is the aim of the present paper. 

We present two main results.
First, we show that the appropriately centered and scaled embedding $\hmun$, for a fixed test point $\xbf$, weakly converges to a limiting Gaussian process. Second, we present a resampling-based approach to infer properties of the distribution of $\hmun$. 
In practice, this resampling-based approach allows us to simultaneously and computationally efficiently compute the DRF prediction and a bootstrap approximation of its distribution.

In addition to our theoretical developments,
we present two lines of applications. 
First, we use the estimated Hilbert space embedding to formally test if two conditional distributions coincide or not, and we provide confidence bands for the so-called (conditional) witness function that can be used to assess where the two distributions differ.
Second, we make inference for targets $\theta(\xbf)=G(\PYgX)$ that can be represented by some smooth function $G$ of the underlying distribution $\PYgX$ by replacing $\PYgX$ by its DRF estimate.
A wide range of conditional (multivariate) estimators like the conditional average treatment effect (CATE), conditional quantiles, or conditional correlations can be obtained in this way.
These estimators are mutually consistent. For example, estimated conditional covariance matrices are guaranteed to be positive semi-definite for $d < n$. In general, this might not be guaranteed if we estimated the conditional variances and covariances individually.

\subsection{Contributions}\label{sec:contributions}

We develop asymptotic results for uncertainty quantification for the DRF and apply them in two use cases: testing two conditional distributions for equality and making inference for target parameters like  conditional expectations, the CATE, conditional quantiles, or conditional correlations.

We present a rigorous analysis of the DRF in an RKHS that does not depend on a specific target parameter. 
Consequently, the same DRF can be used to estimate different targets. Furthermore, the targets may be $\R^q$-valued for $q \geq 2$, and confidence ellipsoids in $\R^q$ can be constructed. Generalizing the arguments in~\citet{wager2018estimation} to RKHS's allows us to develop a U-statistics approximation of the DRF prediction $\hmun$ in the RKHS. Particularly, we show that $\hmun$ for a fixed test point $\xbf$ is asymptotically equivalent to a sum of independent, but not necessarily identically distributed, random elements in the Hilbert space $\H$. The former requires a considerable extension of the arguments in~\citet{DRF-paper}. We then extend and refine the arguments in~\citet{wager2018estimation} for Random Forest to obtain an improved characterization of the asymptotic variance of this U-statistics approximation.
These results allow us to establish that $\hmun$, appropriately scaled, converges weakly to a limiting Gaussian process in the RKHS. This result holds under rather natural assumptions and does not depend on the estimation target we have in mind. For instance, the primary assumption is on the Lipschitz continuity of the map $\xbf \mapsto \mu(\xbf)$, which was already used in \citet{DRF-paper}. Using the expression of the Maximum Mean Discrepancy (MMD) between two multivariate Gaussian distributions derived in \citet{ourMMDresult}, we give a specific example where one can verify this assumption.

To cope with the theoretical complexity of our Hilbert space-valued Random Forest, we use and extend techniques to analyze Generalized Random Forests (GRF)~\citep{wager2018estimation, athey2019generalized}, theory for random elements in Hilbert spaces~\citep{hilbertspacebook, HilbertspaceCLTs}, and bootstrap arguments~\citep{B1,B2,B3}. Our RKHS-valued bootstrap result builds on arguments from the bootstrap and empirical process literature and those of~\citet{athey2019generalized}. We show that an adaptation of half-sampling can be used 
to obtain a random element $\hmunHS$ in $\H$, by sampling from the data, that converges to the same limiting distribution as the original estimate $\hmun$, conditional on the data. Consequently, a resampling-based approach can be used to infer properties of the distribution of the random element $\hmun$ of $\H$. 
In practice, we propose to adapt the DRF algorithm of~\citet{DRF-paper} to be fitted in ``little bags'' as motivated in~\citet{athey2019generalized}. This allows us to simultaneously and computationally efficiently compute the DRF prediction and a bootstrap approximation of its distribution in the form of $\hmunHS$.

Finally, we use our bootstrap results for the DRF to formally test for distributional differences between two groups.~\citet{CATEGeneralization} introduced the idea to test equality of the distributions of the control and treatment groups of an experiment, given some covariates. In contrast to estimating the CATE, which compares the two groups based on their mean, comparing whole distributions allows us to identify differences that may not be captured by the mean alone. 
Our developments allow us to formally test for conditional distributional differences between the control and the treatment group at a test point $\xbf$. Although it may be possible to derive an asymptotic normality result for the usual kernel-based CME estimator as used in~\citet{CATEGeneralization}, 
we are not aware of
a formal test for fixed $\xbf$. 
Finally, our confidence bands for the conditional witness function can be interpreted as the Hilbert space-valued generalization of the work in~\citet{wager2018estimation, athey2019generalized}, which derived confidence intervals for the CATE at a fixed $\xbf$.

\subsection{Previous Work}

There is a growing literature 
on nonparametric estimation of 
multivariate conditional distributions. These include Conditional Generative Adversarial Neural Networks~\citep{aggarwal2019benchmarking}, Conditional Variational Auto-Encoders~\citep{sohn2015learning}, Masked Autoregressive Flows~\citep{papamakarios2017masked}, and Conditional Mean Embeddings~\citep{CMEinDynamicalSystems, kernelmeanembeddingreview, OurapproachtoCME}. To the best of our knowledge, none of these methods provide mathematical guarantees of uncertainty.
Our methodology might be most closely related 
to the GRF, which builds on the theory of Causal Forests~\citep{wager2018estimation}. GRF is a locally adaptive method to estimate univariate real-valued targets defined by local moment conditions using forest-based weights. It uses a splitting criterion for growing trees that depends on the specific estimation target, and the resulting estimator is proven to be consistent and asymptotically normal at a test point $\xbf$. In contrast to DRF, a new splitting criterion needs to be constructed for each new target and the theory presented in~\citet{athey2019generalized} only provides results for univariate targets. However, from a theoretical perspective, GRF has exact asymptotic normality guarantees for  more univariate functionals than what the current paper is able to derive with DRF because some functionals mapping $\P_{\Ybf\mid \Xbf=\xbf}$ to the desired targets might not be sufficiently smooth. This is discussed in more detail in Remark~\ref{problematicremark}. 
\citet{kunzel2019metalearners} introduce the X-learner to estimate the CATE, which is a meta algorithm that initially estimates the unobserved potential outcomes, and confidence intervals are obtained via the Bootstrap.

\textit{Outline:} 
In the subsequent Section~\ref{sec:background},
we recall relevant definitions and results concerning RKHS's, the Landau notation, and we introduce basic concepts and summarize core ideas of the DRF. 
Afterwards, Section~\ref{sec:theory} presents our formal assumptions and main results. 
Section~\ref{sec: Experimental} and~\ref{sec:real-valued-targets} discuss our two applications: 
inference for the conditional distributional treatment effect and general multivariate real-valued parameters. 
Finally, Section~\ref{sec:empirical} 
demonstrates empirical validation of our theoretical developments, and Section~\ref{sec:conclusion} concludes with a brief discussion of our results.

\section{Background}\label{sec:background}

In this section, we introduce notation and present key results from~\citet{DRF-paper} that serve as a basis for our subsequent developments. Throughout, we assume an underlying probability space $(\Omega, \mathcal{A}, \P)$ and denote by $\mathcal{M}_{b}(\R^d)$ the space of all bounded signed measures on $\R^d$.

\subsection{Reproducing Kernel Hilbert Spaces and Landau Notation}\label{sec:RKHS} 

Let $\left(\H, \langle\cdot,\cdot\rangle\right)$ be the reproducing kernel Hilbert space induced by the positive definite, bounded, and continuous kernel $k\colon\R^d\times\R^d \to \R$; see for instance~\citet[Chapter 2.7]{hilbertspacebook} for an exposition of the topic. Crucially, continuity of $k$ ensures that $\H$ is separable~\citep[Theorem 2.7.5]{hilbertspacebook}. For a random element $\xi$ taking values in the (separable) Hilbert space $\H$ with $\E[\|\xi \|_{\H}] < \infty$, we define its expected value in $\H$ by
\[
\E[\xi]= \int_{\Omega} \xi d \P \in \H,
\]
where the integral is to be understood in a Bochner sense~\citep[Chapter 3]{hilbertspacebook}. Because $\H$ is separable, this integral is well defined and there are no measurability issues. If $\E[\| \xi \|_{\H}^2] < \infty$, we define the variance of $\xi\in\H$ by
\[
\Var(\xi)=\E[\|\xi \|_{\H}^2] - \| \E[\xi] \|_{\H}^2.
\]

For a sequence of random elements $ \xi_n$ in $\H$, we denote by $\xi_n \stackrel{D}{\to} \xi$ convergence in distribution. That is, for all bounded and continuous functions $F\colon \H \to \R$, we have $\E[F(\xi_n)] \to \E[F(\xi)]$ as $n \to \infty$. 
By separability, every random element $\xi$ with values in $\H$ is tight~\citep[Chapter 7.1]{dudley}. That is, for all $\varepsilon > 0$, there is a compact $K_{\varepsilon} \subset \H$ such that $\P(\xi \in K_{\varepsilon}) \geq 1-\varepsilon$. More generally, uniform tightness of a sequence $\xi_n$, $n \in \mathbb{N}$ means that for all $\varepsilon > 0$, there is a compact $K_{\varepsilon} \subset \H$ such that
\begin{align*}
     \inf_n \P(\xi_{n} \in K_{\varepsilon}) \geq 1-\varepsilon.
\end{align*}
If for all $f \in \H$ the distribution of $\langle \xi,f \rangle$ on $\R$ is $N(0, \sigma_f^2)$ for some $\sigma_f > 0$, we write $\xi \sim N(0, \boldsymbol{\Sigma})$ with $\boldsymbol{\Sigma}$ a self-adjoint Hilbert-Schmidt (HS) operator satisfying $\langle \boldsymbol{\Sigma}f , f\rangle=\sigma_f^2$ . 
In this case, we also write $\xi_n \stackrel{D}{\to} N(0, \boldsymbol{\Sigma})$, if $\xi_n \stackrel{D}{\to} \xi$.

The kernel embedding function $\Phi\colon\mathcal{M}_{b}(\R^d) \to \H$
maps any bounded signed Borel measure $Q$ on $\R^d$ to an element $\Phi(Q) \in \H$ defined by
$$\Phi(Q) = \int_{\R^d} k(\mathbf{y}, \cdot)dQ(\mathbf{y})=\int_{\R^d} k(\mathbf{y}, \cdot)dQ^{+}(\mathbf{y}) - \int_{\R^d} k(\mathbf{y}, \cdot)dQ^{-}(\mathbf{y}),$$
where the integrals are Bochner integrals. Boundedness of $k$ ensures that $\Phi$ is indeed defined on all of $\mathcal{M}_{b}(\R^d)$. If $k$ is the Gaussian kernel, $\| \Phi(Q_1) - \Phi(Q_2) \|_{\H}=0$ implies $Q_1=Q_2$ for all $Q_1, Q_2 \in \mathcal{M}_{b}(\R^d)$; see for example~\citet[p.2]{simon2020metrizing} and~\citet[Example 3.2]{optimalestimationofprobabilitymeasures}. Thus, $\Phi$ is injective, and the inverse $\Phi^{-1}\colon\Phi(\mathcal{M}_{b}(\R^d)) \to \mathcal{M}_{b}(\R^d)$ is well defined. In particular, for $Q=\diracY$, it holds that $\Phi(\diracY)=\kYi$, and thus
\begin{align*}
    \Phi(\hPYgX )= \sum_{i=1}^n \wix \Phi(\diracYi) = \sum_{i=1}^n \wix \kYi =  \hmun
\end{align*}
because $\Phi$ is linear, where $\hPYgX$ is the estimator of $\PYgX$ defined in~\eqref{eq:phat-intro} below.

For two functions $f$ and $g$ from the real numbers into the real numbers with $\liminf_{s \to \infty} g(s) > 0$, we write $f(s)= \mathcal{O}(g(s))$ if 
\[
\limsup_{s \to \infty} \frac{|f(s)|}{g(s)} \leq C
\]
holds
for some $0<C < \infty$. If $C=1$, we write $f(s) \precsim g(s)$. For a sequence of random variables $X_n\colon \Omega \to \R$ and $a_n \in (0,+\infty)$, $n \in \N$, we write $X_n=\O_p(a_n)$ if 
\[
\lim_{M \to \infty} \sup_{n} \P(a_n^{-1} |X_n| > M) = 0.
\]
We write $X_n=\o_p(a_n)$ if $a_n^{-1} X_n$ converges in probability to zero. Similarly, for $(S,d)$ a separable metric space, $\Xbf_n\colon (\Omega, \mathcal{A}) \to (S, \mathcal{B}(S))$, $n \in \N$, and $\Xbf\colon (\Omega, \mathcal{A}) \to (S, \mathcal{B}(S))$ measurable, we write $\Xbf_n \stackrel{p}{\to} \Xbf$ if $d(\Xbf_n, \Xbf)=o_{p}(1)$.

\subsection{Distributional Random Forests}\label{sec:DRF-recall}

Given an i.i.d.\ data sample of size $n$, DRF can be used to estimate a representation $\hPYgX$ of the conditional distribution $\PYgX$ of  $\Ybf=(Y_1, Y_2, \ldots, Y_d)^T \in \mathbb{R}^{d}$ given a realization $\xbf$ of covariates $\mathbf{X}=(X_1, X_2, \ldots, X_p)^T \in \mathbb{R}^{p}$ by the weighted sum 
\begin{align} \label{eq:phat-intro}
    \hPYgX = \sum_{i=1}^n \hwix \diracYi
\end{align}
of Dirac measures $\diracYi$.  
The weights $\hwix$ quantify the relevance of a data point $\xbf_i$ in predicting the target distribution  $\PYgX$. 

To compute the weights $\hwix$, DRF applies a Random Forest algorithm in the RKHS  $(\H, k)$. That is, $N$ trees are built, and each tree splits the data repeatedly with respect to the covariates into sets of the form $\mathcal{I}_L=\{i: X_{ij} < S \}$ and $\mathcal{I}_R=\{i: X_{ij} \geq S \}$, whereby a number $S \in \R$ and a candidate feature  $j$ are chosen according to a splitting criterion depending on $\Ybf$. In DRF,  
each split is chosen to maximize
the Maximum Mean Discrepancy (MMD) statistic~\citep{gretton2007kernel} across the child nodes 
such that the induced distributions in the child nodes are as different as possible.
For example, to split the root node of a tree, two sets of indices $\mathcal{I}_L$ and $\mathcal{I}_R$ are searched for which 
\begin{equation} \label{eq: splitting}
    \left\lVert
    \Phi \left(\frac{1}{|\mathcal{I}_L|}\sum_{i\in\mathcal{I}_L} \diracYi \right)
    -
    \Phi \left(\frac{1}{|\mathcal{I}_R|}\sum_{i\in\mathcal{I}_R}\diracYi \right)
    \right\rVert_{\H}^2 = 
    \left\lVert
    \frac{1}{|\mathcal{I}_L|}\sum_{i\in\mathcal{I}_L}  \kYi
    -
  \frac{1}{|\mathcal{I}_R|}\sum_{i\in\mathcal{I}_R} \kYi
    \right\rVert_{\H}^2
\end{equation}
is maximal. This is essentially the traditional CART splitting criterion~\citep{breiman2001random}, but now in the RKHS. Indeed, for $d=1$ and the kernel $k(x,y)=xy$, the MMD statistic~\eqref{eq: splitting} simplifies to the CART criterion~\citep[Section 2.3.1]{DRF-paper}. Thus, the trees are built such that the distribution of the response variable 
in the child nodes are as different as possible in the MMD metric. Intuitively, this should lead to leaves that are as homogeneous as possible such that the leaf containing $\xbf$ of the $k$th tree, denoted by $\Lcal_k(\xbf)$, approximately contains a sample from the distribution $\PYgX$. For $k$ being the Gaussian kernel, the embedding $\Phi$ is injective, which allows the MMD statistic to detect any distributional differences for large enough sample sizes. Crucially, this splitting criterion does not depend on the estimation target like for instance the CATE. 
\citet{DRF-paper} employed efficient computation methods of this MMD statistic to obtain a forest construction with comparable computational complexity as the original Random Forest algorithm. This is achieved by using a well-known approximation of the MMD statistic with a specified number of random features~\citep[Section 2.3]{DRF-paper}. 

Once the trees are grown and leaf nodes determined, the weights $\wix$ can be computed. 
For each tree $k=1,\ldots,N$, the leaf node $\Lcal_k(\xbf)$ of the $k$th tree is the leaf in which $\xbf$ falls. 
Then, the prediction of $\mu(\xbf)=\Phi(\PYgX) =
\E[k(\Ybf,\cdot)\mid\Xbf=\xbf]$ from each tree is given by averaging the elements $k(\Ybf_j,\cdot)$ that belong to $\Lcal_k(\xbf)$, namely $1 / |\Lcal_k(\xbf)|\sum_{j\in\Lcal_k(\xbf)}k(\Ybf_j,\cdot)$; that is, the $k(\Ybf_j,\cdot)$'s belonging to the leaf $\Lcal_k(\xbf)$ of $\xbf$ each get assigned the weight $1/|\Lcal_k(\xbf)|$. 
These per-tree predictions 
are subsequently averaged to form the forest predictor 
\begin{displaymath}
    \hmun = \frac{1}{N}\sum_{k=1}^N \left(\frac{1}{ |\Lcal_k(\xbf)|}\sum_{j\in\Lcal_k(\xbf)}k(\Ybf_j,\cdot)\right). 
\end{displaymath}
Rearranging this double sum such that each Hilbert element is present only once yields
\begin{displaymath}
    \hmun 
    = \sum_{i=1}^n\hwix k(\Ybf_i,\cdot)
\end{displaymath}
for suitable weights $\hwix$. 
From this last expression, we can read off our weights $\hwix$ that quantify the importance of the $i$th data point in predicting $\mu(\xbf)$. Consequently, this approach allows us to characterize data-adaptive neighborhoods of data points $\xbf$ whose corresponding conditional distribution is similar to $\PYgX$. Algorithm~\ref{pseudocode} in Appendix~\ref{appendix: pseudocode} provides pseudocode for this procedure.

\section{Theoretical Development}\label{sec:theory}

DRF estimates the embedding $\mu(\xbf)=\Phi(\PYgX) = \E[k(\Ybf,\cdot)\mid\Xbf=\xbf]$ of the conditional distribution $\PYgX$ in an RKHS with reproducing kernel $k$. In this section, we first state the assumptions on the forest construction and the data generating process and recall that it consistently estimates $\mu(\xbf)$ at a certain rate~\citep{DRF-paper}. Subsequently, we establish convergence in distribution of the standardized estimator to a limiting Gaussian process. Lastly, we develop a consistent variance estimation procedure that enables efficient empirical computation.

\subsection{Forest Construction and Consistency in the RKHS}

We require our forest construction to satisfy the following properties that are similar to~\citet{wager2018estimation}. 
First, we require that the data used to build a tree is independent from the data used to populate its leaves for prediction. 
To ensure this, we split the subsample used to build a particular tree into two halves. The first half is used to construct the tree. Then, the data from the second half gets assigned to the leaves of the tree according to the covariate splits that were fitted on the first half. Subsequently, the responses from the second half of the data, which are now distributed across the leaves, are used to form the DRF predictions. Second, when a parent node is split into two child nodes, every feature may be chosen with at least a certain non-zero probability. Third, the prediction of a tree is not allowed to depend on the order of the training samples. Fourth, when a parent node of a tree is split into two child nodes, this split may not be arbitrarily imbalanced. Each child node needs to contain a certain fraction $\alpha$ of its parent's data points. 
Finally, to grow a tree, the traditional Random Forest algorithm samples training data points with replacement from the $n$ training points; that is, a bootstrap approach is pursued. In contrast, we sample a subset without replacement as done by~\citet{wager2018estimation, athey2019generalized}. 
These assumptions on the forest construction are summarized as follows:
\begin{enumerate}[label=(\textbf{F\arabic*})]
    \item\label{forestass1} (\textit{Honesty}) The data used for constructing each tree is split into two halves; the first is used for determining the splits and the second for populating the leaves and thus for estimating the response. The covariates in the second sample may be used for the splits, to enforce the subsequent assumptions, but not the response.
    \item\label{forestass2} (\textit{Random-split}) At every split point and for all feature dimensions $j=1,\ldots,p$, the probability that the split occurs along the feature $X_j$ is bounded from below by $\pi/p$ for some $\pi > 0$.
    \item\label{forestass3} (\textit{Symmetry}) The (randomized) output of a tree does not depend on the ordering of the training samples.
    \item\label{forestass4} (\textit{$\alpha$-regularity}) 
   After splitting a parent node, each child node contains at least a fraction $\alpha \leq 0.2$ of the parent's training samples. Moreover, the trees are grown until every leaf contains between $\kappa$ and $2\kappa - 1$ many observations for some fixed tuning parameter $\kappa \in \N$. 
    \item\label{forestass5}(\textit{Data sampling}) 
    To grow a tree, a subsample of size $s_n$ out of the $n$ training data points is sampled. 
    We consider $s_n=n^{\beta}$ with 
    \[
    1 > \beta >  \left( 1 +  \frac{\log( (1-\alpha)^{-1})}{\log(\alpha^{-1})} \frac{\pi}{p} \right)^{-1},
    \]
   where $\alpha$ is chosen in~\ref{forestass4}.
\end{enumerate}
The validity of the above properties are ensured by the forest construction.
As outlined above, the prediction of DRF for a given test point $\xbf$ is an element of $\H$. If we denote the $i$th training observation by $\Zbf_i=(\Xbf_i, k(\Ybf_i, \cdot)) \in \R^p \times \H$, then DRF estimates the embedding of the conditional distribution $\Phi(\PYgX)$ by averaging the corresponding estimates across the $N$ trees, namely  
\begin{equation*}\label{rewriting}
    \Phi(\hPYgX) = \frac{1}{N} \sum_{k=1}^N T(\mathbf{x}; \varepsilon_k, \Zcal_k),
\end{equation*}
where $\Zcal_k = \{\Zbf_{k_1}, \ldots, \Zbf_{k_{s_n}}\}$ is a random subset of $\{\Zbf_i\}_{i=1}^n$ of size $s_n$ (see~\ref{forestass5}) chosen for constructing the $k$th tree, and $\varepsilon_k$ is a random variable capturing the randomness in growing the $k$th tree such as the choice of the splitting candidates, and $T(\mathbf{x}; \varepsilon_k, \Zcal_k)$ denotes the output of a single tree.  The output of a single tree is given by the average of the terms $k(\Ybf_i, \cdot)$ over all data points $\mathbf{X}_i$ contained in the leaf $\Lcal_k(\mathbf{x})$ of the tree constructed from $\varepsilon_k$ and $\Zcal_k$: 
\begin{align}\label{tree} 
        T(\xbf; \varepsilon_k, \{\Zbf_{k_1}, \ldots, \Zbf_{k_{s_n}}\} )=\sum_{j \in \I_{k}} \frac{\1(\Xbf_{j} \in \mathcal{L}_{k}(\xbf))}{|\mathcal{L}_{k}(\xbf)|} k(\Ybf_{j}, \cdot), 
\end{align}
where $\I_{k}$ is the set of indices of size $s_n/2$ that is used for populating the leaves (see \forestass{1}).

To develop our theory, we do not consider forests that consist of a user-specified number $N$ of trees. Instead, we consider $N \to \infty$, such that the forest estimator $\hmun$ is obtained by averaging all possible $\binom{n}{s_n}$ many trees, which equals the number of possible subsets of $\{\Zbf_{i}\}_{i=1}^n$ of size $s_n$.
This idealized version of our DRF predictor, which we will denote by $\hmun$ from now onwards, is given by 
\begin{equation}\label{finalestimator}
    \hmun = \binom{n}{s_n}^{-1}  \sum_{i_1 < i_2 < \cdots < i_{s_n}} \E_{\varepsilon} \left[T(\xbf; \varepsilon, \{\mathbf{Z}_{i_1}, \ldots, \mathbf{Z}_{i_{s_n}}\})\right].
\end{equation}
This is a standard simplification also employed by~\citet{wager2018estimation, athey2019generalized}. 
\citet{DRF-paper} established that 
$\hmun$ in~\eqref{finalestimator} consistently estimates $\mu(\xbf)$ with respect to the RKHS norm at a certain rate.

\begin{restatable}[Theorem 1 in~\citet{DRF-paper}]{theorem}{consistency}\label{thm: consistency}
Assume that the forest construction satisfies the properties~\ref{forestass1}--\ref{forestass5}. Additionally, assume  that $k$ is a bounded and continuous kernel (this corresponds to Assumption~\ref{kernelass1} and~\ref{kernelass2} below) and that we have a random design with $\mathbf{X}_1,\ldots, \mathbf{X}_n$ independent and identically distributed on $[0,1]^p$ with a density bounded away from $0$ and infinity (this corresponds to~\ref{dataass1} below). If the subsample size $s_n$ is of order $n^\beta$ 
for some $0  < \beta < 1$, the mapping $\xbf \mapsto \mu(\xbf) \in \H$ is Lipschitz (this corresponds to~\ref{dataass2} below).
Then, we have consistency of $\hmun$ in~\eqref{finalestimator} with respect to the RKHS norm, namely
\begin{equation}\label{eq: rate} 
\norm{\hmun- \mu(\xbf)}_\H = \O_{p}\left(n^{-\gamma} \right)
\end{equation}
for any $\gamma \leq \frac{1}{2} \min\left( 1- \beta, \frac{\log((1-\alpha)^{-1})}{\log(\alpha^{-1})} \frac{\pi}{p} \cdot \beta \right)$.
\end{restatable}

\begin{remark}
    \citet{DRF-paper} also assume 
    \begin{align}\label{supbound}
        \sup_{\xbf \in [0,1]^p} \E[ \|k(\Ybf, \cdot )\|_\H^2\mid \Xbf \myeq \xbf] < \infty
    \end{align}
    in their Theorem 1. However, this assumption is redundant as it is implied by the boundedness of $k$. Indeed, for all $\xbf$,
    \begin{align*}
        \E[ \|k(\Ybf, \cdot )\|_\H^2\mid \Xbf \myeq \xbf] = \E[ k(\Ybf, \Ybf )\mid \Xbf \myeq \xbf] \leq \sup_{\ybf_1, \ybf_2} k(\ybf_1, \ybf_2) < \infty,
    \end{align*}
    implying \eqref{supbound}.
\end{remark}

Although this result shows consistency of $\hmun$ at a certain rate, it does not establish distributional convergence of the scaled difference. Subsequently, we  establish this result.

\subsection{Asymptotic Normality in the RKHS}\label{sec:asympt-norm-RKHS}

To establish an asymptotic Gaussian process behavior of $\hmun$ in the Hilbert space, we first show asymptotic linearity in Theorem~\ref{thm: asymptoticlinearity}. More precisely, we show that
\begin{align*}
    \hmun -  \mu(\xbf) = \frac{s_n}{n} \sum_{i=1}^n  T_{n}(\Zbf_i) + \o_{p}(\sigma_n)
\end{align*}
holds, where $\Zbf_i=(\Xbf_i,k(\Ybf_i,\cdot))\in\R^d\times\H$ concatenates the $i$th covariates and the embedding of the $i$th response in the Hilbert space, $T_n$ is some function depending on $n$, and $\sigma_n$ is some standard deviation converging to zero. Denote by
\begin{align}\label{xindef}
    \xi_n = \frac{1}{\sigma_n} \left( \hmun - \mu(\xbf) \right)
\end{align}
the shifted and scaled DRF estimator whose asymptotic distribution we subsequently investigate. 
To establish that $ \xi_n $  asymptotically converges to a Gaussian process, two ingredients are required~\citep[Chapter 7]{hilbertspacebook}.  First, we require weak convergence to a limiting Gaussian distribution in $\R$ of the univariate marginals 
$\langle\frac{s_n}{n\sigma_n} \sum_{i=1}^n  T_{n}(\Zbf_i), f \rangle$ for all $f\in\H$.  Second, we require uniform tightness of the sequence $\frac{s_n}{n\sigma_n} \sum_{i=1}^n  T_{n}(\Zbf_i)$ for $n\ge 1$.

We make the following assumptions on the data generating process. Throughout, we assume all involved expectations exist and are finite.

\begin{enumerate}[label=(\textbf{D\arabic*})]
    \item\label{dataass1} The covariates $\mathbf{X}_1,\ldots, \mathbf{X}_n$ are independent and identically distributed on $[0,1]^p$ with a density bounded away from $0$ and infinity.
        \item\label{dataass2} The mapping $\xbf \mapsto \mu(\xbf)=\E[ k(\Ybf,\cdot)  \mid \Xbf \myeq \xbf] \in \H $ is Lipschitz.
          \item\label{dataass3} The mapping $\xbf \mapsto \E[ \|k(\Ybf,\cdot) \|_{\H}^2 \mid \Xbf \myeq \xbf] $ is Lipschitz.
                \item\label{dataass4} $\Var(k(\Ybf,\cdot)\mid \Xbf=\xbf) = \E[\| k(\Ybf,\cdot) \|_{\H}^2 \mid \Xbf=\xbf] - \|\E[ k(\Ybf,\cdot)  \mid \Xbf=\xbf]\|_{\H}^2 > 0$.
                   \item\label{dataass5} $\E[\left\| k(\Ybf, \cdot) - \mu(\xbf)\right \|_{\H}^{2+\delta}\mid \Xbf=\xbf] \leq M$, for some constants $\delta, M$ uniformly over $\xbf\in [0,1]^d$. 
                        \item\label{dataass6} For all $f \in \H \setminus \{0\}$, $\Var( \langle k(\Ybf,\cdot), f \rangle\mid \Xbf=\xbf)= \Var( f(\Ybf)\mid \Xbf=\xbf) > 0$.
                               \item\label{dataass7} For all $f \in \H\setminus \{0\}$, $\xbf \mapsto \E[\left| f(\Ybf)\right|^{2}\mid \Xbf=\xbf]$ is Lipschitz. 
\end{enumerate}

As outlined below and in Appendix \ref{appendix: assumptions}, \dataass{3}--\dataass{5} are automatically satisfied when using the Gaussian kernel
\begin{align*}
    k(\ybf, \cdot) = \exp \left( - \frac{\norm{\ybf - \cdot}}{2\sigma^2} \right),
\end{align*}
and if $\Ybf \mid \Xbf=\xbf$ has nonzero variance.
Assumption~\ref{dataass1} is a standard assumption when analyzing Random Forests~\citep{meinshausen2006quantile,wager2018estimation}, 
and~\ref{dataass2}--\ref{dataass5} correspond to natural generalizations of the assumptions in~\citet{wager2018estimation} to the RKHS setting.
Particularly, Assumption~\ref{dataass2} implies that we have  
\[
\|\mu(\xbf_1) - \mu(\xbf_2) \|_{\H} \leq L \| \xbf_1 - \xbf_2  \|
\]
 for some $L > 0$. This means that if $\| \xbf_1 - \xbf_2  \|$ is small, the two respective conditional distributions have to be close in MMD distance. Because $\|\cdot \|_{\H}$ metrizes weak convergence for the Gaussian kernel, 
the distributions $\P_{\Ybf\mid\Xbf=\xbf_1}$ and $\P_{\Ybf\mid\Xbf=\xbf_2}$ are consequently close in the weak topology if   $\xbf_1$ and $\xbf_2$ are close enough in $\R^p$. 
Moreover,~\ref{dataass2} implies that for all $f \in \H$ and all $\xbf_1,\xbf_2 \in [0,1]^d$, we have
\begin{align}\label{eq:Lipschitzcontinuityf}
   | \E[f(\Ybf) \mid \Xbf=\xbf_1 ] - \E[f(\Ybf) \mid \Xbf=\xbf_2 ] | &= | \langle f, \mu(\xbf_1) - \mu(\xbf_2) \rangle | \nonumber \\
   &\leq \| f\|_{\H} \|\mu(\xbf_1) - \mu(\xbf_2) \|_{\H} \nonumber \\
   & \leq  \| f\|_{\H} L \| \xbf_1 - \xbf_2  \|.
\end{align}
Consequently,~\ref{dataass2} implies Lipschitz continuity of $\mathbf{x} \mapsto \E[f(\Ybf) \mid \Xbf=\xbf ]$ for all $f \in \H$. Appendix~\ref{appendix: assumptions} discusses an example for which~\dataass{2} is met. Assumption~\dataass{3} is for example trivially met for shift-invariant kernels, that is, $k$ such that $k(\ybf_1,\ybf_2)=k_0(\ybf_1-\ybf_2)$ for an appropriate function $k_0: \R^d \to \R$. Indeed, it holds that,
\begin{align*}
   \E[ \|k(\Ybf,\cdot) \|_{\H}^2 \mid \Xbf \myeq \xbf] = \E[k(\Ybf,\Ybf)  \mid \Xbf \myeq \xbf] = k_0(0). 
\end{align*}
Thus, $ \E[ \|k(\Ybf,\cdot) \|_{\H}^2 \mid \Xbf \myeq \xbf]$ is constant in $\xbf$ and therefore also Lipschitz. In particular, we have $k_0(0)=1$ for the Gaussian kernel. Similarly, Assumption \dataass{4} on the variance of $k(\Ybf, \cdot)$ is also met under a Gaussian kernel if in addition $\PYgX$ is not concentrated on a constant. That is, for some $c \in \R^d$,
\[
\ybf=c, \text{ for all } \ybf \in A, \text{ with } \PYgX(A)=1,
\]
must not be true. This is shown formally in Appendix \ref{appendix: assumptions}. Assumption~\ref{dataass5} is trivially met if $k$ is a bounded kernel, that is, if \kernelass{1} below holds. \dataass{5} implies that for all $f \in \H$,
\begin{align}\label{eq:delta-plus-2-moment-bound}
    \E[\left| f(\Ybf) - \E[f(\Ybf) \mid \Xbf=\xbf]\right|^{2+\delta}\mid \Xbf=\xbf] &=\E[\left| \langle f, k(\Ybf, \cdot) - \mu(\xbf) \rangle\right|^{2+\delta}\mid \Xbf=\xbf]\nonumber\\
    & \leq \E[\left| \| f \|_{\H} \cdot \| k(\Ybf, \cdot) - \mu(\xbf)\|_{\H} \right|^{2+\delta}\mid \Xbf=\xbf]\nonumber\\
    & \leq \| f \|_{\H}^{2+\delta} M
\end{align}
 holds uniformly over $\xbf\in [0,1]^d$. These two conclusions together with \dataass{6} and \dataass{7} will allow us to apply results of~\cite{wager2018estimation} for the univariate marginal 
$$
\left\langle\frac{s_n}{n\sigma_n} \sum_{i=1}^n  T_{n}(\Zbf_i), f \right\rangle
= \frac{s_n}{n\sigma_n} \sum_{i=1}^n  \left\langle T_{n}(\Zbf_i), f \right\rangle 
$$ 
to establish the asymptotic normality of these marginals. 

We also make the following assumptions on the kernel $k$:
\begin{enumerate}[label=(\textbf{K\arabic*})]
    \item\label{kernelass1} $k$ is bounded, i.e., $\sup_{\ybf_1, \ybf_2} k(\ybf_1, \ybf_2) < \infty$.
   \item\label{kernelass2} $(\xbf,\ybf) \mapsto k(\xbf,\ybf)$ is (jointly) continuous.
   \item\label{kernelass3} $k$ is integrally strictly positive definite (denoted by $\int$spd), that is
   \begin{align*}
      \| \Phi(Q_1) - \Phi(Q_2) \|_{\H} = 0 \implies Q_1=Q_2  ,\text{ for all } Q_1, Q_2 \in \mathcal{M}_{b}(\R^d);
   \end{align*}
   see for instance~\citet{optimalestimationofprobabilitymeasures, simon2020metrizing}.
\end{enumerate}

The Gaussian kernel satisfies the conditions in~\ref{kernelass1}--\ref{kernelass3}, for instance.

As outlined above, our first main result shows that  
$\xi_n$ in~\eqref{xindef} is  asymptotically linear, that is, indistinguishable from a sum of independent elements in $\H$ as $n\to\infty$.

\begin{restatable}{theorem}{asymptoticlinearity}\label{thm: asymptoticlinearity}
Assume conditions~\ref{forestass1}--\ref{forestass5},~\ref{dataass1}--\ref{dataass7},~\ref{kernelass1}, and~\ref{kernelass2} hold. Denote by $\Zbf_i=(\Xbf_i, k(\Ybf_i, \cdot))$, $i=1,\ldots,n$. Then, there exists a map $T_n\colon[0,1]^p \times \H \to \H$ such that, with 
\begin{align}\label{eq: sigmandef}
   \sigma_n^2=\frac{s_n^2}{n} \Var(T_{n}(\Zbf_1)),
\end{align}
we have $\sigma_n \to 0$, $\| \hmun -  \mu(\xbf)\| =\mathcal{O}_p(\sigma_n)$, and
\begin{align} \label{asymptoticlin}
   \hmun -  \mu(\xbf) = \frac{s_n}{n} \sum_{i=1}^n  T_{n}(\Zbf_i) + \o_{p}(\sigma_n).
\end{align}
Moreover, $T_n$ is given by
\begin{align}\label{Tn}
    T_{n}(\Zbf_i)=\E[T(\Zcal_{s_n})\mid\Zbf_i] - \E[T(\Zcal_{s_n})].
\end{align}
\end{restatable}

\begin{remark}
To decrease the bias of the individual trees, the subsample size $s_n$ must not be of too small order compared to $n$. However, this causes the variance $\sigma_n^2$ to go to $0$ at a slower rate than $\sqrt{n}$, and the precise rate is given by
\begin{align*}
      C_1 \frac{\sqrt{s_n}}{\log(s_n)^{p/2} \sqrt{n}} \precsim \sigma_n \precsim     C_2 \frac{\sqrt{s_n}}{\sqrt{n}}
\end{align*}
similarly to~\citet{wager2018estimation}.
If $s_n=n^{\beta}$ with $0< \beta < 1$, this translates to 
\begin{align*}
      C_1  \frac{1}{ \beta^{p/2} \log(n)^{p/2} n^{(1-\beta)/2} } \precsim \sigma_n \precsim     C_2 \frac{1}{n^{(1-\beta)/2}}.
\end{align*}
\end{remark}

Due to Theorem~\ref{thm: asymptoticlinearity}, it is enough to show that 
\[
\frac{s_n}{n \sigma_n} \sum_{i=1}^n  T_{n}(\Zbf_i) \stackrel{D}{\to} N(0, \boldsymbol{\Sigma}_{\xbf})
\]
to establish asymptotic normality of $\xi_n$. To achieve this, we need to establish univariate convergence and asymptotic tightness.

For $f \in \H$, consider the univariate marginal $\frac{s_n}{n } \sum_{i=1}^n  \langle T_{n}(\Zbf_i) ,f \rangle$. Due to Assumption~\ref{forestass1}--\ref{forestass5},  Lipschitz continuity of $\xbf \mapsto \langle \mu(\xbf),f \rangle$ implied by~\ref{dataass2}, and~\eqref{eq:delta-plus-2-moment-bound}, Assumption~\ref{dataass1}--\ref{dataass7} verify all assumptions of Theorem 3.1 of~\citet{wager2018estimation}. Consequently, there exists a $\sigma_{n}(f) > 0$ converging to zero with $n$ such that
\begin{align}\label{univariatconvergence0}
 \left\langle \frac{1}{\sigma_n(f)} \left( \frac{s_n}{n } \sum_{i=1}^n  T_{n}(\Zbf_i) \right), f  \right\rangle 
 \stackrel{D}{\to} N(0,1).
\end{align}
Unfortunately, the scaling factor $\sigma_n(f)$ obtained from~\citet{wager2018estimation} depends on $f$. The challenge is to show that the convergence in~\eqref{univariatconvergence0} holds for any $f \in \H$ if $\sigma_n(f)$ is replaced by $\sigma_n$ given in~\eqref{eq: sigmandef}. To establish this, we need to refine the characterization of the asymptotic behavior of the variance of $T_{n}$. 
The following result achieves this.

\begin{restatable}{theorem}{varianceprop}\label{amazingvarianceproposition}
Assume conditions~\ref{forestass1}--\ref{forestass5},~\ref{dataass1}--\ref{dataass7},~\ref{kernelass1}, and~\ref{kernelass2} hold. Then, for all $f \in \H\setminus\{0\}$, we have
\begin{align}\label{varianceassumption}
    \lim_{n \to \infty }\frac{\Var(\langle T_{n}(\Zbf_1), f \rangle)}{\Var(T_{n}(\Zbf_1))} = \frac{\Var(\langle k(\Ybf, \cdot) ,f  \rangle|\Xbf=\xbf)}{\Var(k(\Ybf, \cdot)|\Xbf=\xbf)}  =\sigma^2(f) > 0.
\end{align}
\end{restatable}

Thus, the variance of the first order approximation of the univariate forest prediction is of the same order as that of the forest prediction in the Hilbert space. That the resulting ratio $\sigma^2(f)$ is strictly larger than zero is a consequence of assumption~\ref{dataass6}.

The convergence in~\eqref{univariatconvergence0} together with Theorem~\ref{amazingvarianceproposition} establishes
\begin{align*}
    \left\langle \frac{s_n}{n \sigma_n} \sum_{i=1}^n  T_{n}(\Zbf_i), f \right\rangle \stackrel{D}{\to} N(0, \sigma^2(f)),
\end{align*}
that is, weak convergence of the univariate marginals $\langle  \frac{s_n}{n \sigma_n} \sum_{i=1}^n  T_{n}(\Zbf_i), f \rangle$ for all $f \in \H$. Establishing additionally uniform tightness~\citep[Chapter 7]{hilbertspacebook} yields our second main result, namely the asymptotic Gaussian process distribution of the DRF prediction.

\begin{restatable}{theorem}{asymptoticnormality}\label{thm: asymptoticnormality}
Assume conditions~\ref{forestass1}--\ref{forestass5},~\ref{dataass1}--\ref{dataass7},~\ref{kernelass1}, and~\ref{kernelass2} hold. Then, 
\begin{align}\label{eq:asymptoticnormality}
   \frac{1}{\sigma_n}  \left( \hmun -  \mu(\xbf) \right) \stackrel{D}{\to} N(0, \boldsymbol{\Sigma}_{\xbf}),
\end{align}
where $\boldsymbol{\Sigma}_{\xbf}$ is a self-adjoint HS operator satisfying 
\begin{align}
\langle \boldsymbol{\Sigma}_{\xbf} f, f \rangle=    \frac{\Var(\langle k(\Ybf, \cdot) ,f  \rangle|\Xbf=\xbf)}{\Var(k(\Ybf, \cdot)|\Xbf=\xbf)} > 0
\end{align}
for all $f \in \H$.
\end{restatable}
The expression of $\boldsymbol{\Sigma}_{\xbf}$ is intuitive: if $\boldsymbol{\Sigma}_{\xbf}^{o}$ is the covariance operator of the random element $k(\Ybf, \cdot) \mid \Xbf=\xbf $, 
then $\boldsymbol{\Sigma}_{\xbf}$ equals $\boldsymbol{\Sigma}_{\xbf}^{o}$ standardized by its trace; see for example~\citet[Chapter 7]{hilbertspacebook}. 

We now turn to the question of how to approximate the distribution of $\hmun$ itself.

\subsection{Approximation of the Sampling Distribution}\label{Bootstrapapprox}

In this section, we establish an approach to approximate the sampling distribution of $\xi_n$ based on half-sampling. 
This can afterwards be used to make inference for derived point estimators or functionals. 

Our half-sampling scheme is motivated by~\citet{athey2019generalized} and is as follows. 
For a subset $\HS \subset \{1,\ldots,n\}$ with $s_n \leq |\HS|$, denote by $\hmunHS$ the version of $\hmun$ that only uses trees built with data from $\HS$. 
That is, $\hmunHS$ is the counterpart of $\boldsymbol{\Phi}_{\H}$ in~\citet{athey2019generalized}.  
In~\citet{athey2019generalized}, $\HS$ was randomly drawn without replacement such that $|\HS|=n/2$.
To simplify our theoretical developments in approximating the asymptotic distribution of $\hmun$, we draw $\HS$ by sampling $n$ i.i.d.\ random variables 
$W_i \sim \mathrm{Bernoulli}(1/2)$ and consider 
$\HS=\{i\colon W_i=1 \}$.
The cardinality $|\HS|$ of $\HS$ randomly fluctuates around $n/2$, with $|\HS|/n \to 1/2$ almost surely. Because $\HS$ is chosen at random, the element $\hmunHS$ now has two sources of randomness: one from the data and one from drawing $\HS$. Subsequently, we establish that, if the data are kept fixed and only the randomness of the choice of $\HS$ is considered,
\begin{align}\label{xinHSdef}
   \xi_{n}^{\HS}=\frac{1}{\sigma_n}\left( \hmunHS -  \hmun \right)
\end{align}
converges to the same Gaussian random element as the original process $\xi_n$ in~\eqref{xindef}. 
This allows us to approximate the asymptotic distribution of $\xi_n$ and characteristic quantities such as variances from its subsample versions by randomly drawing $\HS$. 

To establish this result, we build on standard bootstrap arguments as for instance presented in~\citet[Chapter 10]{kosorok2008introduction}. 
Formally, we establish in Theorem~\ref{thm: asymptoticnormalityhalfsampling} that 
\begin{align}\label{eq:asymptoticnormalityhalfsampling}
   \xi_{n}^{\HS}=\frac{1}{\sigma_n}  \left( \hmunHS -  \hmun \right) \xrightarrow[W]{D} N(0, \boldsymbol{\Sigma}_{\xbf})
\end{align}
holds. The symbol $\xrightarrow[W]{D}$ denotes so-called conditional convergence in distribution and is characterized by the condition
\begin{align}\label{conditionalconvergence}
 \sup_{h \in \text{BL}_1(\H) }  \left |  \E[ h(\xi_{n}^{\HS}) \mid \Zcal_{n}] - \E[ h(\xi) ] \right | \stackrel{p}{\to} 0,
\end{align}
where $\text{BL}_1(\H)$ denotes the space of all bounded Lipschitz functions from $\H$ to $\R$ with Lipschitz constant bounded by 1. That is, $h \in \text{BL}_1(\H)$ satisfies $\sup_{f \in \H} |h(f)| \leq 1$ and $|h(f_1) - h(f_2) | \leq \| f_1 - f_2 \|_{\H}$ for all $f_1, f_2 \in \H$. 
This definition is in particular reasonable if we recall that convergence in distribution alone, $\xi_n \stackrel{D}{\to} \xi$, is characterized by $\sup_{h \in \text{BL}_1(\H) }  \left |  \E[ h(\xi_{n})] - \E[ h(\xi) ] \right | \to 0$; see for example~\citet[Theorem 11.3.3]{dudley}. 
Consequently,~\eqref{conditionalconvergence} means that, conditional on the data $\Zcal_{n}$, $\xi_{n}^{\HS}$ converges to $\xi$ in distribution in probability; see for example~\citet{B3};~\citet[Chapter 10]{kosorok2008introduction}. 
Hence, if condition~\eqref{conditionalconvergence} holds, we write~\eqref{eq:asymptoticnormalityhalfsampling}. 

Combining arguments from~\citet{B2, B3} with those from~\citet{athey2019generalized}, we show that:

\begin{restatable}{theorem}{asymptoticnormalityhalfsampling}\label{thm: asymptoticnormalityhalfsampling}
Assume conditions~\ref{forestass1}--\ref{forestass5},~\ref{dataass1}--\ref{dataass7},~\ref{kernelass1}, and~\ref{kernelass2} hold. Then,~\eqref{eq:asymptoticnormalityhalfsampling} holds.
\end{restatable}

Consequently, for ``large'' $n$, the distribution of $\xi_{n}^{\HS}$, given the data, is the same as that of $\xi_n$. To empirically characterize this distribution, we use a similar approximation trick as in~\citet{athey2019generalized}. We grow our forest by (i) drawing $B$ subsets $\HS_1, \ldots, \HS_B$ of $\{1,\ldots,n\}$ as described above, (ii) fitting a DRF with $\ell$ trees and calculating the prediction $\hat{\mu}_n^{\HS_b}(\xbf)$ for each $b=1,\ldots,B$, and (iii) obtaining the overall prediction $\hmun$ as the average over $(\hat{\mu}_n^{\HS_b}(\xbf))_{b=1}^B$. This allows us to obtain both an overall DRF prediction and $B$ i.i.d.\ draws from the distribution of $\frac{1}{\sigma_n}  \left( \hmunHS -  \hmun \right)$. This can then be used to approximate, for instance, the variance of $F(\hmun)$ for some function $F$. The following result establishes consistency of this approach for linear and continuous $F\colon\H \to \R^q$. 

\begin{restatable}{corollarytwo}{variancestimation}\label{thm: variancestimation}
Assume conditions~\ref{forestass1}--\ref{forestass5},~\ref{dataass1}--\ref{dataass7},~\ref{kernelass1}, and~\ref{kernelass2} hold. Then, for any $F\colon \H \to \R^q$ linear and continuous,
\begin{align}
    \E\left[\left. \frac{1}{\sigma_n^2}  \left( F( \hmunHS ) -  F(\hmun) \right) \left( F( \hmunHS ) -  F(\hmun) \right)^{\top}   \,\right\vert\, \Zcal_n \right] \stackrel{p}{\to} F \circ \boldsymbol{\Sigma}_{\xbf}.
\end{align}
\end{restatable}

This in particular implies the result for $F$ appropriately differentiable. Crucially, it is also possible to estimate $\sigma_n$ itself.

\begin{restatable}{corollarytwo}{variancestimationtwo}\label{thm: variancestimation2}
Assume conditions~\ref{forestass1}--\ref{forestass5},~\ref{dataass1}--\ref{dataass7},~\ref{kernelass1}, and~\ref{kernelass2} hold. Then,
\begin{align}\label{eq: sigmanestimation}
    \frac{\E[\| \hmunHS  -  \hmun \|_{\H}^2 \mid \Zcal_n]}{\sigma_n^2} \stackrel{p}{\to} 1.
\end{align}
\end{restatable}

\begin{remark}\label{problematicremark}
The class of suitable differentiable  
functions 
$F\colon\H \to \R^q$ depends on the chosen kernel $k$.
We will focus on the Gaussian kernel. This has several advantages: the Gaussian kernel meets all assumptions \kernelass{1}--\kernelass{3} and metrizes weak convergence. Thus, the convergence in $\H$ in~\eqref{eq: rate} can be interpreted as convergence of $\hPYgX$ to $\PYgX$ in the weak topology. Moreover, the Gaussian kernel can be computationally efficiently approximated with the techniques in~\citet{DRF-paper}. However, the RKHS induced by the Gaussian kernel is a relatively small space of functions. 
For instance, for $d=1$, the identity function $f(y)=y$ is not contained in $\H$ for the Gaussian kernel~\citep[Theorem 3]{Minh2010}. In particular, it is not possible to write $f(y)=\langle f, k(y,\cdot) \rangle$. Thus, if we desire to estimate the conditional mean of $Y$ with $\hmun$, asymptotic normality is not immediately guaranteed by our result. However, because $\H$ is dense in the space of bounded and continuous functions from $\R^d$ to $\R$~\citep{Minh2010}, by exploiting some smoothness arguments,
asymptotic normality is expected to hold
for a wide range of functionals, and crucially also for functions into $\R^q$ for $q > 1$.
\end{remark}

Algorithm~\ref{pseudocodeU} illustrates the approach to obtain both the weights $\wbf$ and the approximation to the sampling distribution $\wbf_1, \ldots, \wbf_B$. The computational complexity of this procedure is the same as that of the original DRF algorithm. Indeed, DRF is fit $B$ times with $L$ trees, but this is equivalent to  fitting one DRF with $N=B \cdot L$ trees. Thus, the complexity remains $\O\left(R \times N \times \text{mtry} \times n \log n\right)$, whereby $R$ is the number of random features used to approximate the splitting criterion in~\eqref{eq: splitting}. We refer to~\citet{DRF-paper} for details.

\begin{algorithm}[htp]
\caption{Pseudocode for Distributional Random Forest with Uncertainty. The functions \textsc{BuildForest} and \textsc{GetWeights} are defined in Algorithm~\ref{pseudocode} in Appendix~\ref{appendix: pseudocode}.} \label{pseudocodeU}
\begin{algorithmic}[1]

\Procedure{BuildForest2}{set of samples $\mathcal{D} = \{(\bold{x}_i, \bold{y}_i)\}_{i=1}^n$, number of trees $N$, number of groups $B$}
 \State  $L \gets$ \textsc{round}(N/B)
    \For{$b=1,\ldots, B$}
        \State $\HS \gets $  Random Subsample from $\mathcal{D}$
        \State $\mathcal{F}_b \gets$ \textsc{BuildForest}($\HS$, $L$) \Comment{Build $b$th forest}
    \EndFor
    \State \textbf{return} $\mathcal{F} = \{\mathcal{F}_1, \ldots, \mathcal{F}_B\}$
\EndProcedure

\item[]

\Procedure{GetWeights2}{forests $\mathcal{F}$, test point $\bold{x}$} \Comment{Computes the weighting function with uncertainty}
    \For{$b = 1,\ldots,|\mathcal{F}|$}
        \State $w_b \gets $ \textsc{GetWeights}($\mathcal{F}_b$, $\xbf$)
    \EndFor
    \State $w = \frac{1}{B} \sum_{b=1}^B w_b $
    \State \textbf{return} $w, w_1, \ldots, w_B$
\EndProcedure

\end{algorithmic}
\end{algorithm}

\section{Application: Conditional Distributional Treatment Effect}\label{sec: Experimental}

A frequent measure to assess the effectiveness of a binary treatment $W$ given some covariates $\Xbf=\xbf$ is the CATE, $\E[Y^{do(W=1)} - Y^{do(W=0)}\mid \Xbf=\xbf]$, where we use the do-notation of~\citet{pearl1995}.

As in~\citet{CATEGeneralization}, we assume that \emph{strong ignorability} holds, that is, (i) unconfoundedness
$W \indep (\Ybf^0, \Ybf^1) \mid \Xbf$ 
and (ii) overlap 
$0 < \Prob(W=1 \mid \Xbf) < 1$.
In this case, the CATE can be estimated as a difference in estimated conditional expectations at $\Xbf=\xbf$. That is, the expected mean difference between the treatment and control groups among subjects with properties $\xbf$ is considered. Although the CATE allows us to take treatment effect heterogeneity into account due to conditioning on the covariates $\Xbf$, it fails to capture distributional differences between the treatment and control groups beyond the mean.  
The conditional distributional treatment effect (\codite)~\citep{CATEGeneralization} alleviates this problem. The idea of \codite\ (with the conditional mean embedding) is to not only compare expected values of the treatment and control groups, but to extend the comparison to more general aspects of the distributions. To achieve this, a kernel estimator of the conditional mean embedding, CME, is used~\citep{CMEinDynamicalSystems, CMEinGraphicalModels, OurapproachtoCME}. For instance, to test whether there are any distributional differences between the treatment and the control groups, CME's of both groups are computed and compared. The kernel method of~\citet{CATEGeneralization} requires choosing two kernels and does not come with formal hypothesis testing. In contrast, we can estimate the CME's of the two groups by two DRF's in a locally adaptive way instead of choosing a kernel for the covariate space. 
Moreover, we are able to introduce tests and confidence bands at a test point $\xbf$ using the Gaussian Hilbert space element approximation we derived above. 

Let us denote by $\hmunz$ the DRF estimate in the control group ($W=0$) and by $\hmuno$ the estimate in the treatment group ($W=1$), and let $\PYgXz$ and $\PYgXo$ be the associated conditional distributions of the control and treatment groups at the test point $\xbf$, respectively. The conditional witness function~\citep{CATEGeneralization} 
\begin{align}\label{witnessfunc}
    \R^d\ni\ybf \mapsto \hmuno(\ybf) - \hmunz(\ybf)\in\H
\end{align}
captures differences between the two conditional distributions $\PYgXz$ and $\PYgXo$ as a function of the response value $\ybf$. The true conditional witness function is given by
\begin{align*}
    \mu_1(\xbf)(\ybf)-\mu_0(\xbf)(\ybf)=\E[k(\Ybf^1, \ybf) \mid \Xbf=\xbf]-\E[k(\Ybf^0, \ybf) \mid \Xbf=\xbf].
\end{align*}
Areas of $\ybf$-values where the conditional witness function is positive or negative indicate where the conditional density of one group is higher or lower than the other~\citep{CATEGeneralization}. 
If the conditional witness function is non-zero, there are  distributional differences between the treatment and the control group. Such a comparison is especially helpful if the conditional mean estimates in the two groups are equal, resulting in a conditional treatment effect of $0$ on the mean level.

Our developments in this section are as follows. First, we present a formal test for assessing whether the conditional response distributions of the treatment and control groups are equal. Particularly, we develop a test for 
\begin{equation}\label{eq:codite-test}
    H_0\colon\PYgXz=\PYgXo \quad \mathrm{vs.} \quad H_A\colon \PYgXz\neq\PYgXo
\end{equation}
using the statistic
$\norm{\hmuno-\hmunz}_{\H}^2$, which equals the norm of the conditional witness function in the Hilbert space. We establish that our test is asymptotically valid and, given a $\int$spd kernel as in~\ref{kernelass3}, the power of our test converges to $1$. Second, we provide a simultaneous asymptotic confidence band for the conditional witness function  itself. 
These two developments involve the distribution of the squared norm of a Gaussian random element $\xi$, $\|\xi\|_{\H}^2$, which is  intractable~\citep{gretton2012optimal}. Our half-sampling approach presents a convenient way to approximate this distribution. 

Before we present our results, we introduce some notation. Denote by $n_0$ the size of the control group and by $n_1$ the size of the treatment group. For simplicity, we assume that $n_0/n_1 \to 1$, but it is possible to relax this condition.
Let $(\Yibf^0, \Xibf)$, $i=1,\ldots,n_0$ and $(\Yibf^1, \Xibf)$, $i=1,\ldots,n_1$ denote i.i.d.\ samples from the control and treatment groups, respectively, and let
$\Zcal_{n_j}^j=\{(k(\Ybf^j_1, \cdot), \Xbf_1), \ldots, (k(\Ybf^j_{n_j}, \cdot), \Xbf_{n_j})  \}$
for $j \in \{0,1\}$ denote the respective observations with response elements of the Hilbert space $\H$.
We denote the concatenated data from both groups by $\Zcal_{n_{01}}=\left( \Zcal_{n_0}, \Zcal_{n_1} \right)$, and introduce the total number of observations $n_{01}=n_0+n_1$. We assume that the observations from the treatment and control groups are independent and that strong ignorability holds as in~\citet{CATEGeneralization}.  Furthermore, let  $\xi_j \sim N(0, \boldsymbol{\Sigma}^{j}_{\xbf})$ 
for $j \in \{0,1\}$, where for all $f \in \H$
\begin{align}
\langle \boldsymbol{\Sigma}_{\xbf}^j f, f \rangle=    \frac{\Var(\langle k(\Ybf^j, \cdot) ,f  \rangle|\Xbf=\xbf)}{\Var(k(\Ybf^j, \cdot)|\Xbf=\xbf)}
\end{align}
holds as in Theorem~\ref{thm: asymptoticnormality} with the respective variance-covariance operators from both groups. Finally, let $\sigma_{n_j,j}$ denote the standard deviation as in~\eqref{eq: sigmandef} for the respective groups $j \in \{0,1 \}$. 

The following result describes the asymptotic distribution of the (suitably rescaled) test statistic for the testing problem~\eqref{eq:codite-test}. 
Moreover, the result establishes that the same limiting distribution is obtained if the individual ``subforests'' of the DRF are used as a bootstrap sample, as described in Section~\ref{Bootstrapapprox}. This will allow us to approximate the distribution of the test statistic for testing~\eqref{eq:codite-test} and for formulating a simultaneous confidence band for the conditional witness function. 

\begin{restatable}{corollarytwo}{normresulttwogroups}\label{lemma_normresult}
Assume conditions~\ref{forestass1}--\ref{forestass5} and~\ref{dataass1}--\ref{dataass7} for both groups,~\ref{kernelass1}, and~\ref{kernelass2} hold, together with strong ignorability. Also assume that $n_0, n_1 \to \infty$ with $n_0/n_1 \to 1$. 
Then, for $\HS_0$, $\HS_1$ independent,
\begin{align}\label{conditionalindep4}
     \left \| \frac{1}{\sigma_{n_1,1}} (\hmunHSo-\hmuno) - \frac{1}{\sigma_{n_0,0}}(\hmunHSz - \hmunz) \right \|_{\H}^2  \xrightarrow[W]{D} \| \xi_0-\xi_1 \|_{\H}^2
\end{align}
and
\begin{align}\label{conditionalindep1}
 \left \| \frac{1}{\sigma_{n_1,1}} (\hmuno-\mu_1(\xbf)) - \frac{1}{\sigma_{n_0,0}}(\hmunz - \mu_0(\xbf)) \right \|_{\H}^2 \stackrel{D}{\to} \| \xi_0-\xi_1 \|_{\H}^2. 
\end{align}
Moreover, if the ratio $\sigma_{n_0,0}/\sigma_{n_1,1}$ converges to some real number $c_2(\mathbf{x})$ that is bounded away from $0$ and $\infty$ as the sample sizes $n_0, n_1$ tend to infinity, we obtain
\begin{align}\label{conditionalindep42}
     \frac{1}{\sigma_{n_1,1}^2} \left \|  (\hmunHSo-\hmuno) - (\hmunHSz - \hmunz) \right \|_{\H}^2  \xrightarrow[W]{D} \| \xi_0- c_2(\mathbf{x}) \xi_1 \|_{\H}^2
\end{align}
and
\begin{align}\label{conditionalindep2}
 \frac{1}{\sigma_{n_1,1}^2}\left \|  (\hmuno-\mu_1(\xbf)) - (\hmunz - \mu_0(\xbf)) \right \|_{\H}^2 \stackrel{D}{\to} \| \xi_0- c_2(\mathbf{x}) \xi_1 \|_{\H}^2. 
\end{align}

\end{restatable}

The above result assumes convergence of the ratio $\sigma_{n_0,0}/\sigma_{n_1,1}$. This condition is used to obtain a common scaling factor in~\eqref{conditionalindep42} and~\eqref{conditionalindep2}. With the expressions derived in Theorem~\ref{amazingvarianceproposition}, it reduces to assuming
\begin{align}\label{Techassumption}
    \frac{\Var( \E[\frac{1}{N_{\xbf}^0} \1\{\Xbf_2 \in \mathcal{L}^0(\xbf) \} \mid \Xbf_1] )}{\Var( \E[\frac{1}{N_{\xbf}^1} \1\{\Xbf_2 \in \mathcal{L}^1(\xbf) \} \mid \Xbf_1] )} \to c(\xbf) > 0.
\end{align}
This essentially means that the behavior of the (conditional) variance of the respective leaf node is asymptotically of the same order
in both samples. Given the assumptions on the forest, together with strong ignorability, this seems to be a mild condition. 
The common scaling factor and limiting behavior 
in~\eqref{conditionalindep42} and~\eqref{conditionalindep2} 
allows us to use a bootstrap procedure on the ``subforests'' to approximate the distribution of the test statistic to test~\eqref{eq:codite-test}. 
The convergence in~\eqref{conditionalindep4} and~\eqref{conditionalindep42} should be understood conditional on the joint data $\Zcal_{n_{01}}$ from both  groups.

Subsequently, we describe how Corollary~\ref{lemma_normresult} can be used to formally test the hypothesis~\eqref{eq:codite-test}. 
In particular, we explain how to approximate the distribution of our test statistic $ \sigma_{n_1,1}^{-2} \| \hmunz - \hmuno  \|_{\H}^2$ under the null hypothesis.
Under the null $\PYgXz=\PYgXo$, we have $\mu_0(\xbf)=\mu_1(\xbf)$. Consequently,~\eqref{conditionalindep2} describes the asymptotic distribution of the rescaled test statistic, namely
\begin{align}\label{conditionalindep1d}
    \frac{1}{\sigma_{n_1,1}^2}\left \| \hmunz - \hmuno \right \|_{\H}^2 \stackrel{D}{\to} \| \xi_1- c_2(\xbf)\xi_0 \|_{\H}^2.
\end{align}
Thus, the rescaled test statistic $\frac{1}{\sigma_{n_1,1}^2}\left \| \hmunz - \hmuno \right \|_{\H}^2$ has the same limiting distribution as its resampling bootstrap version
\begin{align}\label{eq:scaled_bootstrap-stat}
    \frac{1}{\sigma_{n_1,1}^2} \left \|  (\hmunHSo-\hmuno) - (\hmunHSz - \hmunz) \right \|_{\H}^2
\end{align}
given the data. Moreover, we can (approximately) obtain this distribution by sampling from $\HS$, irrespective of whether $\PYgXz=\PYgXo$ holds. Hence, the distribution of the rescaled test statistic $\frac{1}{\sigma_{n_1,1}^2}\left \| \hmunz - \hmuno \right \|_{\H}^2$ under the null hypothesis can be obtained by sampling from $\HS$. Particularly, let $c_{n_1, \alpha}$ be the smallest value obtained from $B$ such draws with $B$ sufficiently large such that
\begin{align}\label{finitealphabound}
   \Prob\left(\left.\frac{1}{\sigma_{n_1,1}^2} \left \|  (\hmunHSo-\hmuno) - (\hmunHSz - \hmunz) \right \|_{\H}^2 >c_{n_1, \alpha} \ \right\vert \ \Zcal_{n_{01}}  \right) \leq \alpha
\end{align}
holds. That is, $c_{n_1, \alpha}$ is the $1-\alpha$ quantile of the test statistic simulated under the null. Next, we establish that the same number $c_{n_1, \alpha}$ can be used to formulate a corresponding test for the test statistic computed on the full data.  
Define the test $\phi(\Zcal_{n_{01}})$ for our testing problem by
\[
\phi(\Zcal_{n_{01}})=\1 \left \{\frac{1}{\sigma_{n_1,1}^2}\left \| \hmunz - \hmuno \right \|_{\H}^2 > c_{n_1, \alpha} \right \}. 
\]
The following result establishes that $\phi$ is of level $\alpha$ and that its power converges to $1$. 

\begin{restatable}{theorem}{powerprop}\label{powerprop}
Assume conditions~\ref{forestass1}--\ref{forestass5} and~\ref{dataass1}--\ref{dataass7} for both groups,~\ref{kernelass1}--\ref{kernelass3} hold, together with strong ignorability and~\eqref{Techassumption}. Then, as $n_0, n_1 \to \infty$ such that $n_0/n_1 \to 1$,
\begin{itemize}
    \item[(i)] $\phi$ has a valid type-I error. That is, if $\PYgXz=\PYgXo$,
    \[
    \limsup_{n_0,n_1} \Prob\left(\frac{1}{\sigma_{n_1,1}^2}\left \| \hmunz - \hmuno \right \|_{\H}^2 > c_{n_1, \alpha}\right) \leq \alpha.
    \]
    \item[(ii)] $\phi$ has power going to 1. That is, if $\PYgXz \neq \PYgXo$,
    \[
    \lim_{n_0,n_1} \Prob\left(\frac{1}{\sigma_{n_1,1}^2}\left \| \hmunz - \hmuno \right \|_{\H}^2 > c_{n_1, \alpha}\right) = 1.
    \]
\end{itemize}

\end{restatable}

In practice, the scaling factor $1/\sigma_{n_1,1}^2$ is unknown. In principle, it can be estimated as elaborated in Corollary~\ref{thm: variancestimation2}, as Theorem~\ref{powerprop} also holds with $\sigma_{n_1,1}^2$ exchanged by a consistent estimator. 
However, 
we can directly consider the unscaled resampled statistics~\eqref{eq:scaled_bootstrap-stat}, namely $ \|  (\hmunHSo-\hmuno) - (\hmunHSz - \hmunz)  \|_{\H}^2$, and identify its $1-\alpha$ quantile, which corresponds to $\sigma_{n_1,1}^2 c_{n_1, \alpha}$.

Subsequently, we present a procedure to construct a confidence band for the conditional witness function $\ybf\mapsto\mu_1(\xbf)(\ybf)-\mu_0(\xbf)(\ybf)$ that is valid jointly for all $\ybf$-values. 
Let $c_{n_1, \alpha}$ be as in~\eqref{finitealphabound}. We show in the following theorem that the  interval 
\begin{align}\label{CIband}
    & \bandybf = \nonumber \\
    &[\hmuno(\ybf) - \hmunz(\ybf)-\sqrt{c_{n_1,\alpha} C}\sigma_{n_1,1}, \hmuno(\ybf) - \hmunz(\ybf)+\sqrt{c_{n_1,\alpha} C}\sigma_{n_1,1}]
\end{align}
is a $1-\alpha$ confidence band for the conditional witness function, where
 $C=\sup_{\ybf} k(\ybf,\ybf)$. 
The constant $C$ is finite due to assuming boundedness of the reproducing kernel in Assumption~\ref{kernelass2}.
That is, $\bandybf$ is a confidence band for the conditional witness function that is valid jointly for all $\ybf$. 

\begin{restatable}{theorem}{coverprop}\label{coverprop}
Assume conditions~\ref{forestass1}--\ref{forestass5} and~\ref{dataass1}--\ref{dataass7} for the control and the treatment group, and assume that~\ref{kernelass1} and~\ref{kernelass2} hold together with strong ignorability and~\eqref{Techassumption}. Then, for $ \bandybf$ as  in~\eqref{CIband}, with $n_0, n_1 \to \infty$ such that $n_0/n_1 \to 1$, 
\begin{align}\label{CIguarantee}
  \liminf_{n_0,n_1\to\infty}  \Prob \left(\cap_{\ybf} \{\mu_1(\xbf)(\ybf)-\mu_0(\xbf)(\ybf) \in   \bandybf \}\right) \geq 1-\alpha.
\end{align}
\end{restatable}

Similarly to above, when performing finite sample calculations and $\sigma_{n_1,1}$ is unknown, we can estimate $\sqrt{c_{n_1,\alpha}}\sigma_{n_1,1}$ using the same resampling procedure as above. Furthermore, we have $C=1$ if we use the Gaussian kernel.

\subsection{Computation}

Subsequently, we provide details on the computation of the test statistic $\norm{\hmuno-\hmunz}_{\H}^2$ for testing equality of the distributions of the control and the treatment group as well as the confidence band for the conditional witness function. 

Consider the three real-valued matrices $ \Kbf_0=(k(\Yibf^0,\Ybf^0_{j} ))_{i=1,\ldots n_0,j=1,\ldots n_0}$ and $ \Kbf_1=(k(\Yibf^1,\Ybf^1_{j} ))_{i=1,\ldots n_1,j=1,\ldots n_1}$ and $ \Kbf=(k(\Yibf^0,\Ybf_{j}^1 ))_{i=1,\ldots n_0,j=1,\ldots,n_1}$. 
Denote by $\hwbfz\in\R^{n_0}$ and $\hwbfo\in\R^{n_1}$ the vectors that concatenate the weights from the DRF predictors for the control and treatment groups, respectively. Moreover, for $j\in\{0,1\}$, consider $ \kbf_j=( k(\Ybf_1^j,\cdot ) , \ldots, k(\Ybf_{n_j}^j,\cdot ) )^{\top}$, and denote by $ \kbf_j(\ybf)=( k(\Ybf_1^j,\ybf ) , \ldots, k(\Ybf_{n_j}^j,\ybf ) )^{\top}$ for $\ybf\in\R^d$.
Then, we have
\begin{align*}
    \hmunz &= \sum_{i=1}^{n_0} \hw_{i,0}(\xbf) k(\mathbf{Y}_{i}^{0}, \cdot)=  \hwbfz^{\top}  \kbf_0 ,\\
        \hmuno&= \sum_{i=1}^{n_1} \hw_{i,1}(\xbf) k(\mathbf{Y}_{i}^{1}, \cdot)= \hwbfo^{\top}  \kbf_1,\\
        \hmunHSz &= \sum_{i=1}^{n_0} \hw_{i,0}^{\HSz}(\xbf) k(\mathbf{Y}_{i}^{0}, \cdot)=( \hwbfz^{\HSz})^{\top}  \kbf_0,\\
        \hmunHSo&= \sum_{i=1}^{n_1} \hw_{i,1}^{\HSo}(\xbf) k(\mathbf{Y}_{i}^{1}, \cdot)=( \hwbfo^{\HSo})^{\top}  \kbf_1.
\end{align*}
Subsequently, we compute $c_{n_1,\alpha}\sigma_{n_1,1}^2$ as the $1-\alpha$ quantile of the $B$ many draws from
\begin{align*}
    &\|\hmunHSz  - \hmunz - (\hmunHSo  - \hmuno) \|_{\H}^2\\
    =& ( \hwbf^{\HSz}_0 -  \hwbf_0)^{\top}  \Kbf_0( \hwbf^{\HSz}_0 -  \hwbf_0) 
    + ( \hwbf^{\HSo}_1 -  \hwbf_1)^{\top}  \Kbf_1( \hwbf^{\HSo}_1 -  \hwbf_1) 
    - 2 ( \hwbf^{\HSz}_0 -  \hwbf_0)^{\top}  \Kbf( \hwbf^{\HSo}_1 -  \hwbf_1).
\end{align*}

To test the null hypothesis of having an equal distribution in the control and the treatment group according to~\eqref{eq:codite-test}, we first compute the test statistic $\norm{\hmuno-\hmunz}_{\H}^2$ according to
\begin{align}\label{eq:witnessNormstat}
\left \| \hmunz - \hmuno \right \|_{\H}^2
= \hwbf_0^{\top} \Kbf_0 \hwbf_0+ \hwbf_1^{\top}  \Kbf_1  \hwbf_1  - 2  \hwbf_0^{\top} \Kbf \hwbf_1. 
\end{align}
The confidence band for the conditional witness function is then given by 
\begin{align}\label{eq:witnessConfBand}
    \bandybf =  [ \hwbf_1^{\top}  \kbf_1(\ybf) -  \hwbf_0^{\top}  \kbf_0(\ybf)-\sqrt{c_{n_1,\alpha} C},  \hwbf_1^{\top}  \kbf_1(\ybf) -  \hwbf_0^{\top}  \kbf_0(\ybf)+\sqrt{c_{n_1,\alpha} C}],
\end{align}
where we have $C=1$ for the Gaussian kernel.

\section{Application: General Real-Valued Parameters}\label{sec:real-valued-targets}

The asymptotic normality result for $\hmun$ derived in Section~\ref{sec:theory} 
can also be applied to make inference for $q$-dimensional real-valued parameters $\theta(\xbf)$ that can be expressed as a function $G$ of the underlying conditional distribution $\PYgX$, namely $\theta(\xbf) = G(\PYgX)$. 
The DRF predictor estimates the embedding $\hmun$ of $\PYgX$ in the Hilbert space. This embedding can then be ``pulled back'' from the Hilbert space to the space of probability measures to give an estimator $\hPYgX$ of $\PYgX$ in the sense of~\eqref{eq:phat-intro} that can, in turn, be used to estimate $\theta(\xbf)$. More precisely, we can represent our estimator $\hat\theta(\xbf)$ by $\hat\theta(\xbf) = F(\hmun)$ for some function $F$ that maps from the Hilbert space into $\R^q$. For sufficiently smooth $F$, the asymptotic normality of
$\frac{1}{\sigma_n} (\hat{\theta}(\xbf)-\theta(\xbf))$
follows from Theorem~\ref{thm: asymptoticnormality}. 

In practice, we estimate $\theta(\xbf)$ by $\hat\theta(\xbf) = G(\hPYgX)$, where $\hPYgX$ is the ``pull-back'' of the DRF predictor $\hmun$ as in~\eqref{eq:phat-intro}. 
To compute confidence intervals for the individual components of $\theta(\xbf)$, we first compute subsample estimators $\hat\theta^{\HS_b}(\xbf) = G(\hPYgX^{\HS_b})$ for $b=1,\ldots,B$, where
$\hPYgX^{\HS_b}$ corresponds to the pullback of the subsample DRF predictor $\hat{\mu}^{\HS_b}(\xbf)$. Then, the empirical variance $\widehat\Var(\hat\theta(\xbf))$ of $\hat\theta(\xbf)$ can be estimated by the variance of the $\hat\theta^{\HS_b}(\xbf)$ over $b=1,\ldots,B$, and confidence intervals can be built using the Gaussian approximation. Alternatively, it is possible to compute confidence intervals via the approximate sampling distribution. To pursue this approach, one first
computes the $1-\alpha/2$ quantile $\hat q_{1-\alpha/2}$ and and the $\alpha/2$ quantile $\hat q_{\alpha/2}$ of $\{\hat\theta^{\HS_b}(\xbf) - \hat\theta(\xbf)\}_{b=1\ldots,B}$. Component-wise $1-\alpha$ confidence intervals for two-sided testing of $\theta(\xbf)=\textbf{0}$ are then given by $[\hat\theta(\xbf)-\hat q_{1-\alpha/2}, \hat\theta(\xbf)-\hat q_{\alpha/2}]$. 

For multi-dimensional parameters $\theta(\xbf)$, which corresponds to $q>1$, one can compute simultaneous elliptical confidence balls. If we denote the $q\times q$ covariance matrix obtained from the sample $\{\hat\theta^{\HS_b}(\xbf) - \hat\theta(\xbf)\}_{b=1\ldots,B}$ by $\widehat\Var(\hat\theta(\xbf))$, these consist of all parameters $\tau$ such that the resulting test statistic $\lvert\lvert\widehat\Var(\hat\theta(\xbf))^{-1/2}(\hat\theta(\xbf) - \tau)\rvert\rvert^2$ is smaller than the $1-\alpha$ quantile of a $\chi^2(q)$ distribution with $q$ degrees of freedom. Analogously to above, one may use the approximate sampling distribution of $\lvert\lvert\widehat\Var(\hat\theta(\xbf))^{-1/2}(\hat\theta^{\HS_b}(\xbf)-\hat\theta(\xbf) )\rvert\rvert^2$ instead of the $\chi^2(q)$ distribution.

\section{Empirical Results}\label{sec:empirical}

In this section, we demonstrate the performance of our DRF confidence intervals for the CATE, conditional quantiles,  conditional correlations, and conditional witness functions for simulated data. 
We consider almost exclusively data generating mechanisms that have already been considered by~\citet{DRF-paper}. 
The only adaptation is that we consider $U(-1,1)^p$ distributed covariates $\Xbf$ instead of $U(0,1)^p$ in Section~\ref{sec:copula}.
In all examples except for the conditional witness functions, we grow a forest that consists of $B=100$ subforests with $\ell = 1000$ trees each, and we choose $\beta=0.9$ in assumption~\ref{forestass5}. 
To fit trees, $10$ random features are used for the approximation of the MMD statistic when splitting the nodes, 
and the minimal node size is $5$. 
Moreover, we consider the Gaussian kernel with the median bandwidth heuristic and compute confidence intervals using the Gaussian approximation. 
For the conditional witness functions, we consider forests that consist of $B=200$ subforests with $\ell=1000$ trees each and choose $\beta=0.9$ because estimating whole confidence bands for the conditional witness function is a complicated task. Code of our analysis is available on GitHub (\url{https://github.com/JeffNaef/drfinference}).

We demonstrate that DRF performs well for a wide range of estimation targets $\theta(\xbf)$. The effort of the user is minimal because estimating a DRF does not depend on the actual target(s). However, we noted that more samples are required for valid inference if the dimension $p$ of the covariates is large. However, this effect appears to be somewhat mitigated if the conditional distribution has low intrinsic dimensionality by depending only on a small number of the covariates.

\subsection{Conditional Average Treatment Effect}\label{sec:cate}

Subsequently, we perform inference for CATE's between a control group $W=0$ and a treatment group $W=1$. We thereby follow the approach used in~\citet{DRF-paper} and consider $W$ as a part of the response, using DRF to find the conditional distribution of $(Y,W)\mid \Xbf=\xbf$. This agrees with our view of seeing the (causal) parameter of interest as a function $F$ of the CME $\hmun$ and, under strong ignorability, consistency of this approach follows from the consistency of $\hmun$. This approach is different from~\citet{wager2018estimation, athey2019generalized} who consider $W$ as a part of the covariates. 

First, we consider a situation where the treatment effect is homogeneous but where $Y$ and $W$ are confounded by $X_3$. We simulate data from 
\begin{equation}\label{eq:emp-cate-sim1}
\begin{array}{l}
    \Xbf\sim\mathrm{Unif}(0, 1)^5, 
    \quad W\mid \Xbf\sim\mathrm{Bernoulli}\big(0.25(1 + \beta_{2, 4}(X_3))\big)\\
    Y\mid (\Xbf, W) \sim 2(X_3 - 0.5) + \mathcal{N}(0, 1),
\end{array}
\end{equation}
where $\beta_{a, b}$ denotes the density of a beta-distributed random variable with parameters $a$ and $b$. 
We consider the test point $\xbf = (0.7, 
    0.3, 0.5, 0.68, 0.43)^T$. 
Our results and comparisons to GRF obtained over $1000$ simulation runs are displayed in Figure~\ref{fig:coverage-cate-sim1}.
The performance of DRF improves as the sample size increases and it reaches the nominal coverage level.  GRF undercovers for $n=5000$. 
However, GRF outperforms DRF with respect to coverage for small sample sizes due to its small bias in this example. Moreover, the confidence intervals of GRF are shorter than the ones with DRF. 

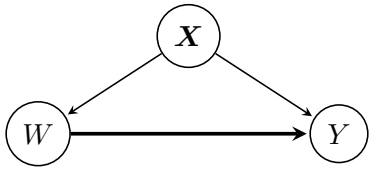
\begin{figure}
    \centering
   \begin{tikzpicture}[
        > = stealth, 
        shorten > = 1pt, 
        auto,
        node distance = 3cm, 
        semithick 
    ]

    \tikzstyle{every state}=[
        draw = black,
        thick,
        fill = white,
        minimum size = 4mm
    ]

    \node[circle, draw=black] (W) at (0,0) {$W$};
    \node[circle, draw=black] (X) at (2,1.3)  {$\boldsymbol{X}$};
    \node[circle, draw=black] (Y) at (4, 0) {$Y$};
    \path[->,line width=1.4pt] (W) edge node {} (Y);
    \path[->] (X) edge node {} (W);
    \path[->] (X) edge node {} (Y);
\end{tikzpicture}
    \caption{Causal graph illustrating the data generating processes in~\eqref{eq:emp-cate-sim1} and~\eqref{eq:emp-cate-sim3}.}
    \label{fig:my_label}
\end{figure}
~
\begin{figure}
\centering
\begin{subfigure}[b]{\textwidth}
\centering
\includegraphics[width=1\textwidth]{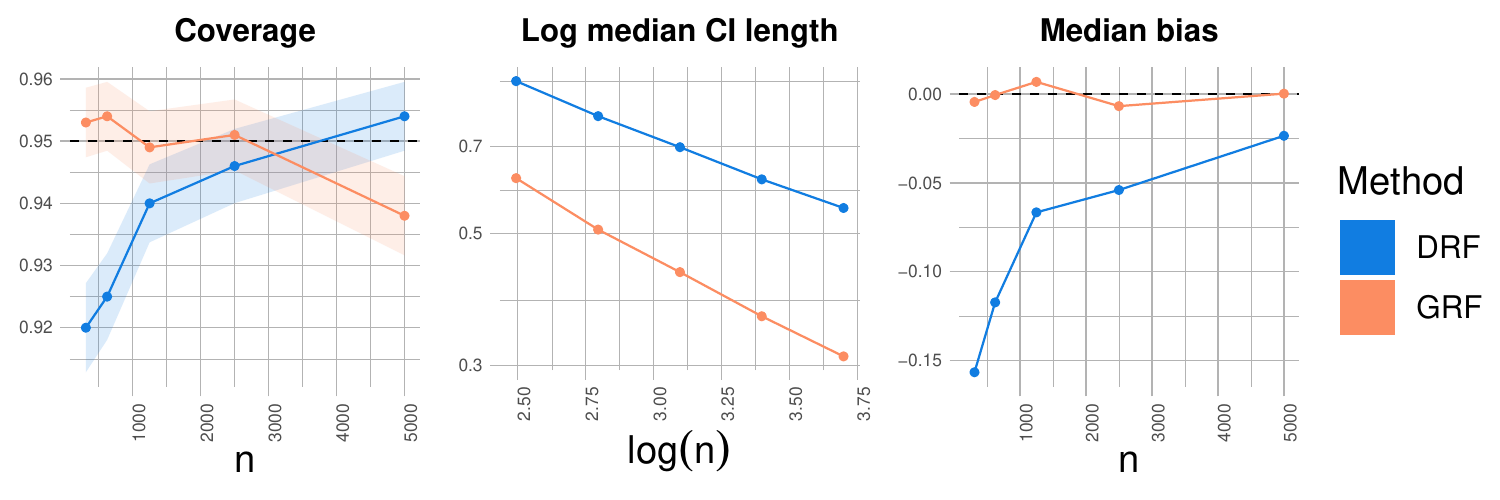}
\subcaption{Without treatment effect}
\label{fig:coverage-cate-sim1}
\end{subfigure}

\begin{subfigure}[b]{\textwidth}
\centering
\includegraphics[width=1\textwidth]{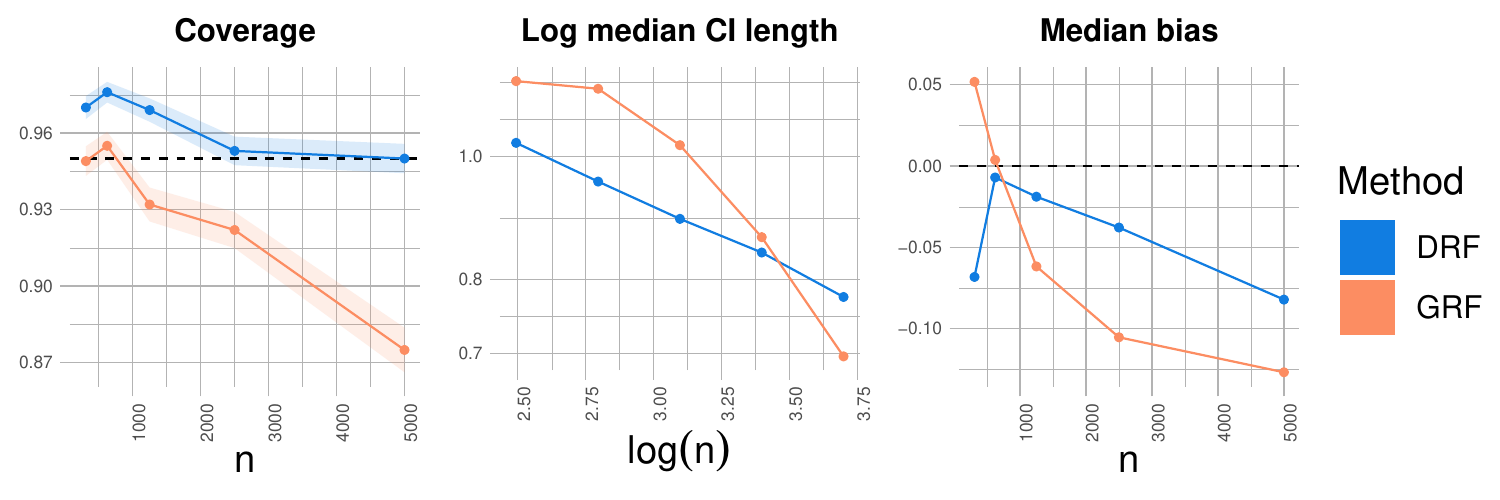}
\subcaption{With treatment effect}
\label{fig:coverage-cate-sim3}
\end{subfigure}

\caption{\label{fig:cate}
	Estimating the CATE of $Y$ given $\Xbf=\xbf = (0.7, 0.3, 0.5, 0.68, 0.43)^T$ with data from~\eqref{eq:emp-cate-sim1} (homogeneous treatment effect and observed confounding) in Figure~\ref{fig:coverage-cate-sim1} and with data from~\eqref{eq:emp-cate-sim3} (heterogeneous treatment effect and observed confounding) in Figure~\ref{fig:coverage-cate-sim3} for different values of $n$ over $1000$ simulation runs. 
	The plots display the coverage (fraction of times the true, and in general unknown, CATE was inside the confidence interval) and log median length of two-sided $95\%$ confidence intervals for the CATE and median bias over $1000$ simulation runs. 
	The shaded regions in the coverage plots represent $95\%$ confidence bands with respect to the $1000$ simulation runs. 
	DRF parameters: $B=100$, $\ell = 1000$, $\beta=0.9$, consider $10$ randomly sampled features to split, minimal node size of $5$. GRF parameters: $50\,000$ trees, other values are left at their default values.
}
\end{figure}

Second,
we consider a situation where the treatment effect is heterogeneous and where $Y$ and $W$ are confounded. 
We simulate data from
\begin{equation}\label{eq:emp-cate-sim3}
\begin{array}{l}
    \Xbf\sim\mathrm{Unif}(0, 1)^5, 
    \quad W\mid \Xbf\sim\mathrm{Bernoulli}\big(0.25(1 + \beta_{2, 4}(X_3))\big)\\
    Y\mid (\Xbf, W) \sim 2(X_3 - 0.5) + (W - 0.2)\cdot\eta(X_1)\eta(X_2) + \mathcal{N}(0, 1),
\end{array}
\end{equation}
where $\eta(x) = 1 + (1 + \exp{-20(x-1/3)})^{-1}$ and $\beta_{a, b}$ denotes the density of a beta-distributed random variable with parameters $a$ and $b$. That is, the treatment effect is heterogeneous because different values of $\Xbf$ result in a different treatment effect, and confounding via $\Xbf$ is present because $W$ also depends on $\Xbf$. We consider the test point $\xbf = (0.7, 
    0.3, 0.5, 0.68, 0.43)^T$. 
Our results and comparisons to GRF obtained over $1000$ simulation runs are displayed in  Figure~\ref{fig:coverage-cate-sim3}.
For small sample sizes $n$, DRF overcovers, but it gradually reaches the nominal $95\%$ level for larger sample sizes. In contrast, GRF fails to reach the nominal $95\%$ level for larger sample sizes due to its bias. 

When estimating the CATE with the GRF algorithm, a centering step to center $Y$ and $W$ with respect to $\Xbf$ is performed. 
With DRF, we found that such an additional centering is not useful. 
With DRF, we used a total number of $10^5$ trees whereas with GRF, we were not able to use as many
due to computational reasons. Since the \texttt{drf} package~\citep{def-package} used is based on \texttt{grf}~\citep{grf-package}, this indicates empirically that the target-tailored splitting criterion of GRF can be computationally considerably more expensive than the general splitting criterion of DRF.

\subsection{Conditional Quantiles}

Subsequently, we consider performing inference for conditional quantiles of $\PYgX$. 
We consider simulated data where the response variable $Y$ experiences a shift in its mean depending on the value of $X_1$, namely
\begin{equation}\label{eq:emp-quantiles}
    \Xbf\sim \mathrm{Unif}(-1, 1)^5, \quad
    Y\sim \mathcal{N}\big(0.8\cdot\one_{X_1>0} , 1\big). 
\end{equation}
The results for estimating three conditional quantiles ($10\%$, $50\%$, and $90\%$), a sample size of $n=5000$, and a range of $x_1$-values are displayed in Figure~\ref{fig:coverage-quantile} and~\ref{fig:coverage-ellipse-quantile}. In Figure~\ref{fig:coverage-quantile}, the coverage for the different quantiles is close to the nominal $95\%$ coverage except at and around the value $x_1=0$ where the mean function of $Y$ experiences a discontinuity. 
Figure~\ref{fig:coverage-ellipse-quantile} displays the joint coverage of all three conditional quantiles $10\%$, $50\%$, and $90\%$.
The coverage is again close to the nominal and slightly higher than it for $x_1$-values away from $0$.
The disturbing effect of the discontinuity at $x_1=0$ is again visible. 

\begin{figure}
	\centering
	\includegraphics[width=\textwidth]{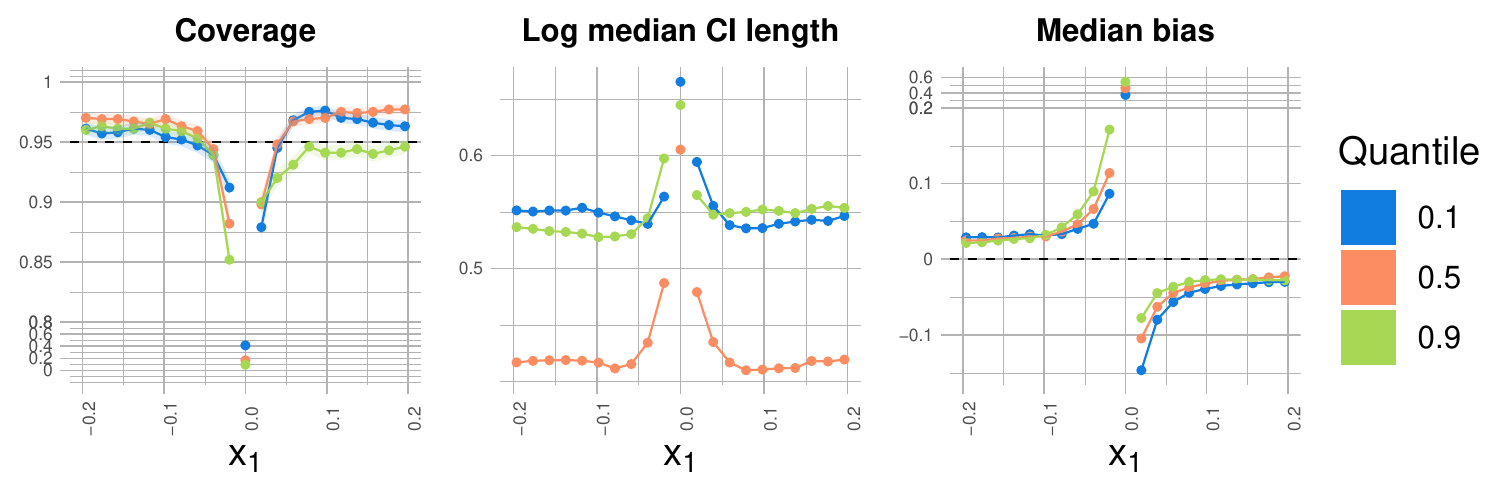}
		\caption[]{\label{fig:coverage-quantile} 
	Estimating conditional quantiles ($10\%$, $50\%$, and $90\%$; differentiated by color) of $Y$ given $X_1=x_1$ from~\eqref{eq:emp-quantiles} (mean shift in $Y$ based on $X_1$) for $n = 5000$ and different values of $x_1$. 
	The plot displays the coverage (fraction of times the true, and in general unknown, conditional quantile was inside the confidence interval) and log median length of two-sided $95\%$ confidence intervals for the conditional quantile and median bias over $1000$ simulation runs. 
	The shaded regions in the coverage plot represent $95\%$ confidence bands with respect to the $1000$ simulation runs. 
	DRF parameters: $B=100$, $\ell = 1000$, $\beta=0.9$, consider $10$ randomly sampled features to split, minimal node size of $5$. 
	}
\end{figure}

\begin{figure}
	\centering
	\includegraphics[width=0.7\textwidth]{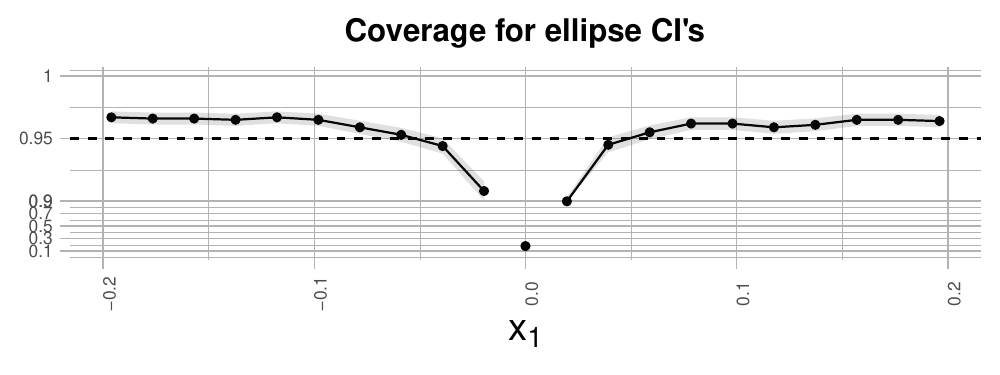}
		\caption[]{\label{fig:coverage-ellipse-quantile} 
	Ellipsoid confidence intervals for the vector of the three conditional quantiles from Figure~\ref{fig:coverage-quantile}.
	The shaded regions in the coverage plot represent $95\%$ confidence bands with respect to the $1000$ simulation runs. 
	}
\end{figure}

\subsection{Conditional Correlation}\label{sec:copula}

Conditional copulas allow us to represent conditional multivariate distributions $\P(\Ybf\le \ybf\mid \Xbf=\xbf) = \P(Y_1\le y_1,\ldots,Y_d\le y_d\mid \Xbf=\xbf)$ in terms of the marginal distributions $\P(Y_i\le y\mid \Xbf = \xbf) = F_{Y_i\mid \Xbf = \xbf}(y)$ for $1\le i\le d$. This technique is frequently employed in fields such as risk analysis or finance~\citep{cherubini2004copula}. 
More precisely, Sklar's theorem~\citep{sklar1959fonctions} asserts the existence of a so-called conditional copula $C_{\xbf}$ at the test point $\xbf$, which is a CDF on $[0,1]^d$, satisfying
\[
    \P(\Ybf\le \ybf \mid  \Xbf=\xbf)
    = C_{\xbf}\big( F_{Y_1\mid \Xbf = \xbf}(y),\ldots,F_{Y_d\mid \Xbf = \xbf}(y)\big).
\]

The DRF algorithm may estimate conditional copulas fully nonparametrically or by estimating the parameters of a certain parametric model. For example, if the data comes from a conditional Gaussian copula $\Ybf\mid\Xbf=\xbf\sim C_{\rho(\xbf)}^{\mathrm{Gauss}}$, it is enough to estimate the conditional correlation function $\rho(\xbf)$ that characterizes distributional heterogeneity.
This is a difficult task because distributional heterogeneity may come from the interdependence of the  marginal CDF's due to the copula and may not exclusively occur in the marginals. Because the MMD splitting criterion of DRF is a distributional metric, DRF is able to detect multivariate distributional changes~\citep{gretton2007kernel}.

Subsequently, we consider the 
conditional Gaussian copula $\Ybf=(Y_1,Y_2)\mid \Xbf=\xbf\sim C_{\rho(\xbf)}^{\mathrm{Gauss}}$ with $\Xbf=(X_1,\ldots, X_5)\sim U(-1,1)^5$ and the conditional correlation function $\rho(\xbf) = \Cor(Y_1,Y_2\mid\Xbf=\xbf) = x_1$. 
That is, both $Y_1$ and $Y_2$ follow a standard Gaussian distribution $N(0,1)$ marginally, but their conditional correlation is characterized by $\rho(\xbf) = x_1$.~\citet{DRF-paper} use a slightly different data generating mechanism because they consider a uniform distribution of the covariates with the support $[0,1]$ instead of $[-1,1]$. We consider $[-1,1]$ such that the conditional correlation at $x_1=0$ does not lie at the boundary of the considered $x_1$-values because this would artificially introduce boundary effects similar to the discontinuity effect with conditional quantile estimation above. 

We estimate and make inference for $\rho(\xbf) = x_1$ for a range of values $x_1$ and different sample sizes $n$. Figure~\ref{fig:coverage-copula} illustrates our results. For a sample size of $n=5000$ (displayed in red), our two-sided DRF confidence intervals achieve the nominal $95\%$ coverage rate for $x_1$-values that are not too close to either $-1$ or $1$. For $x_1$-values, and hence conditional correlation values $\Cor(Y_1,Y_2\mid \Xbf=\xbf)$, that are close to either $-1$ or $1$, we see some degeneration behavior because these values imply the special cases that $Y_1$ and $Y_2$ are completely dependent from each other. 

\begin{figure}
	\centering
	\includegraphics[width=0.9\textwidth]{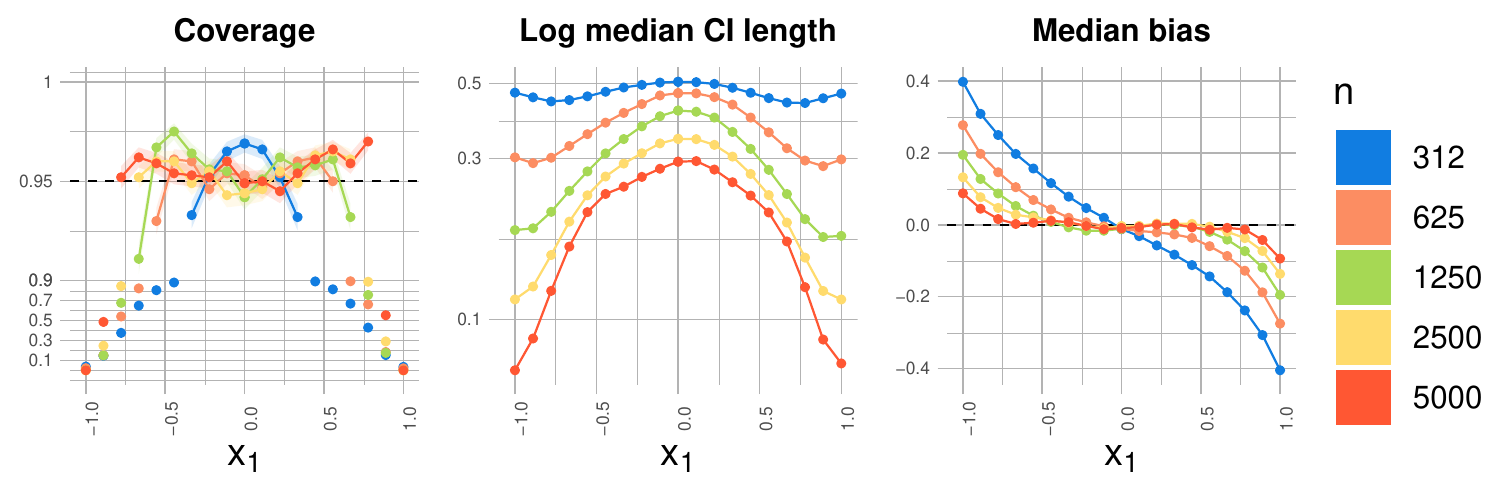}
		\caption[]{\label{fig:coverage-copula} 
	Estimating conditional correlations $\rho(\xbf)=\Cor(Y_1, Y_2\mid \Xbf=\xbf)$ of data from the 
	conditional Gaussian copula  $\Ybf=(Y_1,Y_2)\mid \Xbf=\xbf\sim C_{\rho(\xbf)}^{\mathrm{Gauss}}$ with $\Xbf=(X_1,\ldots, X_5)\sim U(-1,1)^5$ for a range of $x_1$-values ($x$-axis) and sample sizes $n$ (differentiated by color). 
	The plot displays the coverage (fraction of times the true, and in general unknown, conditional correlation was inside the confidence interval) and log median length of two-sided $95\%$ confidence intervals for the conditional correlation and median bias over $1000$ simulation runs. 
	The shaded regions in the coverage plot represent $95\%$ confidence bands with respect to the $1000$ simulation runs. 
	In the coverage plot, for $x_1=-1$ and $x_1=1$, the dots from all three values of $n$ are on top of each other. 
	DRF parameters: $B=100$, $\ell = 1000$, $\beta=0.9$, consider $10$ randomly sampled features to split, minimal node size of $5$. 
	}
\end{figure}

\subsection{Witness Function for conditional distributional treatment effect}\label{sec:witnessfoo}

In Section~\ref{sec: Experimental}, we outlined how to
test for distributional differences between two treatment groups and how to compute simultaneous confidence bands for the corresponding conditional witness function. 
To illustrate the performance of DRF in this use case,  
we revisit the two data generating mechanisms~\eqref{eq:emp-cate-sim1} and~\eqref{eq:emp-cate-sim3} that we considered when we analyzed the CATE in Section~\ref{sec:cate}. 
In the first case with data from~\eqref{eq:emp-cate-sim1}, there is no treatment effect, and the treatment ($W=1$) and the control ($W=0$) groups are equally distributed. In the second case with data from~\eqref{eq:emp-cate-sim3}, there is a treatment effect. 

To formally test if the distributions of the treatment and control groups are different at all, we simulate $1000$ data sets of sample size $n=5000$ each from the two data generating mechanisms and compute the test statistic $\norm{\hmuno-\hmunz}_{\H}^2$ according to~\eqref{eq:witnessNormstat}. For each of the $1000$ runs, we compute a p-value for testing the null hypothesis that the two embeddings from the treatment and control groups are the same against a two-sided alternative using the approximate bootstrap sample distribution of the test statistic obtained from the $B$ many subforests. Figure~\ref{fig:witnessNorm} displays our findings for the data generating mechanism~\eqref{eq:emp-cate-sim1} with equal distributions and illustrates that the p-values are dominated by a $\mathrm{Uniform}(0, 1)$ distribution, which is given by the gray line. Consequently, the p-values seem to be valid.
In particular, $3.6\%$ (this number has a $95\%$ confidence interval of $(0.0311,0.0409)$) of them are below the nominal $0.05$ level. 
With the data generating mechanism~\eqref{eq:emp-cate-sim3}, all p-values equal the smallest possible value, and the null hypothesis is always rejected.

\begin{figure}
\centering
\includegraphics[width=\textwidth]{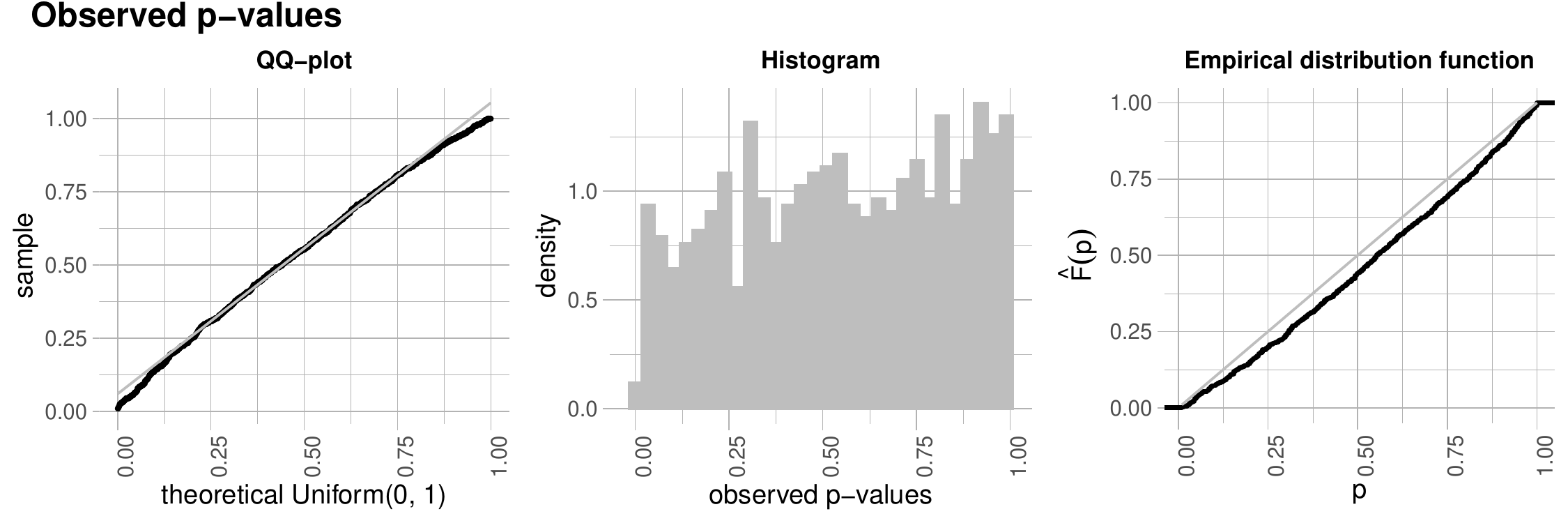}

\caption{\label{fig:witnessNorm}
Two-sided p-values (QQ-plot, histogram, and empirical distribution function) from $1000$ repetitions for testing the null hypothesis that the treatment and control groups have equal distributional embeddings at level $\alpha=5\%$ with data of sample size $n=5000$ from~\eqref{eq:emp-cate-sim1} at the test point $\xbf=(0.7,0.3,0.5,0.68,0.43)^T$. 
DRF parameters: $B = 200$, $\ell=1000$, $\beta=0.9$,
consider $10$ randomly sampled features to split, minimal node size of $5$.
}
\end{figure}

To investigate where the treatment and control distributions differ, we estimate the whole conditional witness function and compute simultaneous confidence bands according to~\eqref{eq:witnessConfBand}. Figure~\ref{fig:witnessfoo} illustrates our results. 
With the data from~\eqref{eq:emp-cate-sim1} where the treatment and control distributions coincide, $99.8\%$ ($95\%$ confidence interval of $(0.9968, 0.9992)$) of the simultaneous $95\%$ confidence bands cover the true underlying conditional witness function that constantly equals 0. 
Although our method overcovers in this situation, Figure~\ref{fig:witness3} illustrates that the power goes to $1$ under the alternative because no simultaneous confidence band contains the constant zero function. In this case, the true conditional witness function is covered in $96.5\%$ of the cases ($95\%$ confidence interval of $(0.9601, 0.9699)$).

These simulations illustrate the practical applicability and usefulness of our developments of the conditional distributional treatment effect in Section~\ref{sec: Experimental}. 
This approach allows us to capture differences between two distributions that may not be represented by mean differences alone. Moreover, our theoretical developments can be directly translated into practice and consequently enable us to perform formal tests that involve test statistics with highly complex and generally intractable distributions. 

\begin{figure}
\centering
\begin{subfigure}[b]{0.48\textwidth}
\centering
\includegraphics[width=1\textwidth]{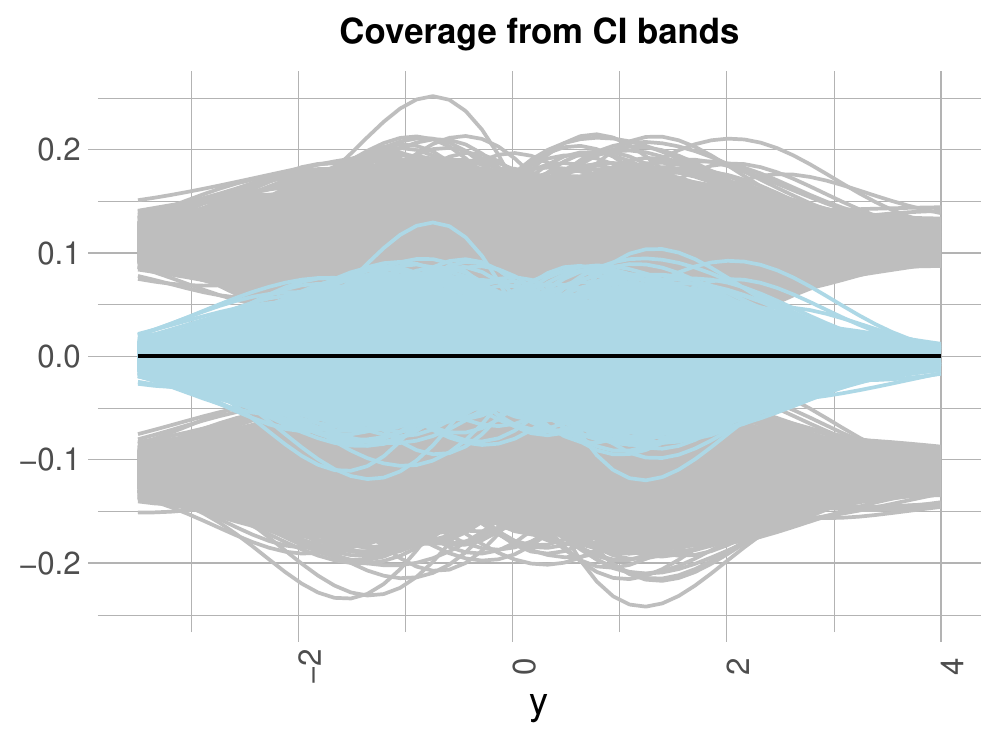}
\subcaption{Without treatment effect}
\label{fig:witness1}
\end{subfigure}
~
\begin{subfigure}[b]{0.48\textwidth}
\centering
\includegraphics[width=1\textwidth]{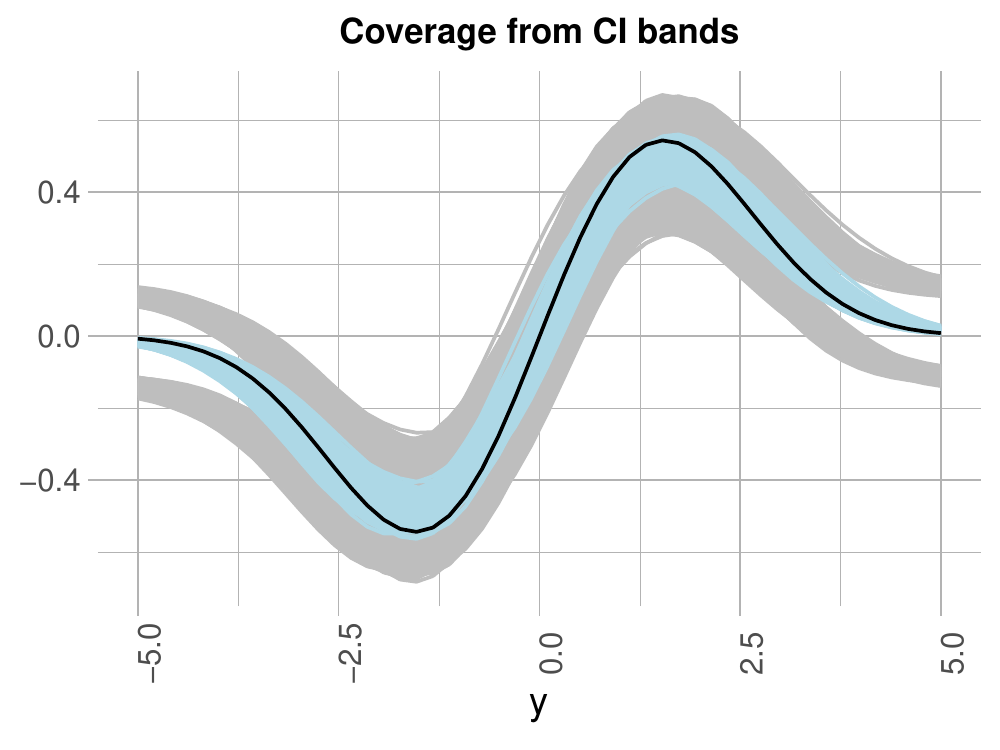}
\subcaption{With treatment effect}
\label{fig:witness3}
\end{subfigure}

\caption{\label{fig:witnessfoo}
	Simultaneous $95\%$ confidence bands (gray) and conditional witness function estimators (blue) over $1000$ repetitions of the true conditional witness function (black) for data of sample size $n=5000$ without~\eqref{eq:emp-cate-sim1} in Figure~\ref{fig:witness1} and with~\eqref{eq:emp-cate-sim3} in Figure~\ref{fig:witness3} treatment effect at the test point $\xbf=(0.7,0.3,0.5,0.68,0.43)^T$. 
 DRF parameters: $B = 200$, $\ell=1000$, $\beta=0.9$,
consider $10$ randomly sampled features to split, minimal node size of $5$.
}
\end{figure}

\section{Conclusion}\label{sec:conclusion}

We developed results about the asymptotic distribution of the Distributional Random Forest (DRF)~\citep{DRF-paper}, which is a forest-based~\citep{breiman2001random} method to nonparametrically estimate Hilbert space embeddings of multivariate conditional distributions in a locally adaptive fashion. 
The general approach of DRF allows us to estimate a wide range of multivariate targets from one and the same DRF estimator. 
Because the DRF prediction is Hilbert space-valued, we formulated and developed new theory for Random Forests operating in Hilbert spaces, building on~\citet{wager2018estimation}. In particular, we explicitly characterized the exact asymptotic behavior of the variance of the DRF prediction. Moreover, we established a bootstrap-type result that allowed us to approximate its distribution in a computationally efficient way.

We presented two strands of applications: 
we formally tested two treatment groups for distributional differences and investigated where these differences occur, 
and we estimated and made inference for low-dimensional parameters like the conditional average treatment effect (CATE), conditional quantiles, and conditional correlations. 
The former application is particularly important to determine differences between the treatment and the control group if the distribution of the two groups are different beyond the mean. 
To simplify the application of our theory in this former use case, 
we fitted two DRF's, one for each treatment group, similar to~\citet{CATEGeneralization}. 
Simulation studies demonstrated the performance and usefulness of our developed inference results for the DRF for these two strands of applications.

\section*{Acknowledgements}

CE and PB received funding from the European Research Council (ERC) under the European Union’s Horizon 2020 research and innovation programme (grant agreement No. 786461).

\clearpage
\appendix

\section{Derivations and Proofs}
\label{appendix: proofs}

\paragraph{Preliminaries.}
First, we recall some of the notation and definitions from the main text. 
Let $\left(\Omega, \mathcal{A}, \mathbb{P}\right)$ denote the underlying probability space.
Throughout, let $(\H, \langle, \cdot, \rangle)$ denote the RKHS associated with the kernel $k$. 
We assume that $k$ is \emph{bounded} and \emph{continuous} in its two arguments. Boundedness of $k$ ensures that $\mu$ is indeed defined on all of $\mathcal{M}_{b}(\R^d)$, and continuity of $k\colon \R^d \times \R^d \to \R$ ensures that $\H$ is \emph{separable}. Thus, we assume throughout that \kernelass{2} holds, such that measurability issues can be avoided. Let us denote by $\xi\colon (\Omega, \mathcal{A}) \to (\H,\mathcal{B}(\H))$ a map from $\Omega$ to $\H$. Separability implies that such a map $\xi$
is measurable if and only if $\langle \xi,f\rangle$ is measurable for all $f \in \H$.  
Moreover, it can easily be checked that $\Phi(P)$ is linear on $\mathcal{M}_{b}(\R^d)$. Separability of $\H$ and $\E[\| \xi \|_{\H}] < \infty$ mean that the integral
\[
\E[\xi] = \int_{\Omega} \xi d \P,
\]
is well defined and that 
\[
F(\E[\xi])=\E[F(\xi)],
\]
for any continuous linear function $F\colon\H \to \R$.\footnote{Here and below, $F(\xi)$ is meant to denote $F(\xi(\omega))$ for all $\omega \in \Omega$. } In particular, $\E[\langle \xi, f  \rangle]= \langle \E[\xi], f \rangle$ for all $f \in \H$. Moreover, for $q \geq 1$, denote by
\begin{align*}
    \mathcal{L}^q(\Omega, \mathcal{A}, \H) &= \{\xi\colon (\Omega, \mathcal{F}) \to (\H,\mathcal{B}(\H)) \text{ measurable, with } \E[\|\xi\|^q] < \infty]  \}\\
    \mathbb{L}^q(\Omega, \mathcal{A}, \H) &= \text{Set of equivalence classes in $\mathcal{L}^q(\Omega, \mathcal{A}, \H)$}\\
    \Var(\xi)&=\E[\|\xi - \E[\xi]\|_{\H}^2]=\E[\|\xi\|_{\H}^2] - \|\E[\xi]\|_{\H}^2, \ \ \xi \in \mathcal{L}^2(\Omega, \mathcal{A}, \H)\\
    \Cov(\xi_1,\xi_2)&= \E[ \langle \xi_1 - \E[\xi_1],\xi_2 - \E[\xi_2] \rangle]=\E[\langle \xi_1 ,\xi_2  \rangle ] - \langle \E[\xi_1], \E[\xi_2] \rangle, \ \ \xi_1,\xi_2 \in \mathcal{L}^2(\Omega, \mathcal{A}, \H).
\end{align*}
Furthermore, it is well-known that $(\mathbb{L}^q, \| \cdot \|_{\mathbb{L}^q(\H)})$ is a Banach space with 
\[
\| \xi \|_{\mathbb{L}^q(\H)}= \E[\| \xi \|_{\H}^q]^{1/q}.
\]
This allows us to
also define \emph{conditional} expectations. For a sub $\sigma$-algebra $\mathcal{F} \subset \mathcal{A}$ and an element $\xi \in \mathcal{L}^1(\Omega, \mathcal{A}, \H) $, the conditional expectation $\E[\xi\mid \mathcal{F}]$ is the (a.s.) unique element  such that
\begin{itemize}
    \item[(C1)]  $\E[\xi\mid \mathcal{F}] \colon (\Omega, \mathcal{F}) \to (\H,\mathcal{B}(\H))$ is measurable and $\E[\xi\mid \mathcal{F}]\in   \mathbb{L}^1(\Omega, \mathcal{F}, \H)$,
    \item[(C2)] $\E[ \xi \1_{F} ] = \E[ \E[\xi\mid \mathcal{F}]\1_{F} ]$ for all $F \in \mathcal{F}$;
\end{itemize} 
see for instance~\citet{UMEGAKI197049} or~\citet[Chapter 1]{pisier_2016}.
Particularly, condition (C2) implies that $\E[ \E[\xi\mid \mathcal{F}] ] = \E[ \E[\xi\mid \mathcal{F}]\1_{\Omega} ] = \E[\xi] $ due to $\Omega \in \F$ for any $\sigma$-algebra. It can also be shown that $F(\E[\xi\mid \mathcal{F}])=\E[F(\xi)\mid\mathcal{F}]$ for all linear and continuous $F\colon \H \to \R$ and that $\|\E[\xi\mid \mathcal{F}]  \|_{\H}\leq \E[ \|\xi \|_{\H} \mid \mathcal{F}]$~\citep[Chapter 1]{pisier_2016}. Moreover, it can be shown that
\begin{itemize}
    \item[(C3)] For $\xi \in \mathbb{L}^2(\Omega, \mathcal{A}, \H)$, $\E[\xi \mid \mathcal{F}]$ is the orthogonal projection onto $\mathbb{L}^2(\Omega, \mathcal{F}, \H)$;
\end{itemize}
see~\citep{UMEGAKI197049}. 
Although the conditional expectation $\E[\xi\mid \mathcal{F}]$, similarly to real-valued conditional expectations, is only defined a.s., we do not explicitly state this in our developments below.

We denote by
$\E[\xi\mid\Xbf]=\E[\xi\mid \sigma(\Xbf)]$. The following Proposition shows that this notion is well defined and establishes further properties of Hilbert space-valued conditional expectations. 

\begin{proposition}[Proposition 6 in~\citet{DRF-paper}]\label{condexp1}
Let $\left(\H_1, \langle\cdot,\cdot\rangle_1\right)$ and $\left(\H_2, \langle\cdot,\cdot\rangle_2\right)$ be two separable Hilbert spaces, $\Xbf, \Xbf_1, \Xbf_2  \in \mathcal{L}^1(\Omega, \mathcal{A}, \H_1) $, and $\xi_1, \xi_2, \xi \in \mathcal{L}^1(\Omega, \mathcal{A}, \H_2)$.\footnote{We recall that all equalities technically only hold almost surely.}
\begin{itemize}
    \item[(C4)] There exists a measurable function $h\colon (\H_1, \mathcal{B}(\H_1)) \to (\H_2, \mathcal{B}(\H_2)) $ such that $\E[\xi\mid\sigma(\Xbf)]=h(\Xbf)=\E[\xi\mid\Xbf]$.
    \item[(C5)] For $\xi_1 \in \mathcal{L}^2(\Omega, \mathcal{A}, \H_1)$ and $\xi_2 \in \mathcal{L}^2(\Omega, \sigma(\Xbf), \H_1)$,  $\E[ \langle \xi_1, \xi_2 \rangle_{\H_1} \mid \Xbf ] =  \langle \E[\xi_1 \mid \Xbf ], \xi_2 \rangle_{\H_1}$ holds.
        \item[(C6)] If $\Xbf_2$ and $(\xi, \Xbf_1)$ are independent, then $\E[\xi\mid \Xbf_1,\Xbf_2 ] =\E[\xi \mid \Xbf_1 ]$. 
        \item[(C7)] $\E[ \E[\xi \mid \Xbf_1, \Xbf_2 ] \mid \Xbf_1]= \E[ \E[\xi \mid \Xbf_1 ] \mid \Xbf_1, \Xbf_2] = \E[\xi \mid \Xbf_1]$.
\end{itemize} 
\end{proposition}

Condition (C4) in particular allows us to consider $\E[\xi \mid \sigma(\Xbf)]$ as a function in $\Xbf$ and thus justifies the notation $\E[\xi \mid \Xbf]$ and all the subsequent derivations. We may also define \emph{conditional independence} through conditional expectation: with the notation of Proposition~\ref{condexp1}, $\xi$ and $\Xbf_1$ are conditionally independent given $\Xbf_2$ if $\E[f(\xi) \mid \Xbf_1, \Xbf_2]=\E[f(\xi) \mid \Xbf_1]$ for all bounded and measurable $f\colon(\H_2, \mathcal{B}(\H_2)) \to (\R, \mathcal{B}(\R))$; see~\citet[Proposition 2.3]{conditionalindep}. This leads to two further important properties:

\begin{proposition}[Proposition 7 in~\citet{DRF-paper}]\label{condexp2}
Let $\left(\H_1, \langle\cdot,\cdot\rangle_1\right)$ and $\left(\H_2, \langle\cdot,\cdot\rangle_2\right)$ be two separable Hilbert spaces, $\Xbf, \Xbf_1, \Xbf_2  \in \mathcal{L}^1(\Omega, \mathcal{A}, \H_1) $, and $\xi_1, \xi_2, \xi \in \mathcal{L}^1(\Omega, \mathcal{A}, \H_2)$.
\begin{itemize}
    \item[(C8)] If $\xi$ and $\Xbf_2$ are conditionally independent given $\Xbf_1$, then $\E[\xi \mid \Xbf_1,\Xbf_2 ] =\E[\xi \mid \Xbf_1 ] $,
    \item[(C9)] If $\xi_1$, $\xi_2$ are conditionally independent given $\Xbf$, then $ \E[ \langle \xi_1, \xi_2 \rangle \mid \Xbf  ]=\langle \E[ \xi_1 \mid \Xbf ], \E[ \xi_2 \mid \Xbf ] \rangle .$
\end{itemize} 
\end{proposition}

For $\xbf\in \R^p$, denote by $P_{\xbf}$ the conditional distribution of $\Ybf$ given $\Xbf=\xbf$ on $\R^d$. For two functions $f$ and $g$ with $\liminf_{s \to \infty} g(s) > 0$, we denote 
$f(s)= \mathcal{O}(g(s))$ if 
\[
\limsup_{s \to \infty} \frac{|f(s)|}{g(s)} \leq C
\]
for some $C > 0$. If $C=1$, we write $f(s) \precsim g(s)$. For a sequence of random variables $X_n\colon \Omega \to \R$ and a sequence of real numbers $a_n \in (0,+\infty)$, $n \in \N$, we write  $X_n=\O_p(a_n)$ if 
\[
\lim_{M \to \infty} \sup_{n} \P(a_n^{-1} |X_n| > M) = 0,
\]
that is, $X_n$ is bounded in probability. We write $X_n=\o_p(a_n)$ if $a_n^{-1} X_n$ converges to zero in probability. Similarly, for $(S,d)$ a separable metric space, $\Xbf_n\colon (\Omega, \mathcal{A}) \to (S, \mathcal{B}(S))$, $n \in \N$ and $\Xbf\colon (\Omega, \mathcal{A}) \to (S, \mathcal{B}(S))$ measurable, we write $\Xbf_n \stackrel{p}{\to} \Xbf$, if $d(\Xbf_n, \Xbf)=o_{p}(1)$.

Finally, let $\Xbf \in \mathcal{L}^2(\Omega, \mathcal{A}, \H_1)$ and $\xi \in \mathcal{L}^2(\Omega, \mathcal{A}, \H_2)$, and assume that $A \subset \Omega$ depends on $\Xbf$, $A=A(\Xbf)$. Thus, for $\Xbf$ fixed to a certain value, $A$ is a fixed set. If $ \P(A \mid \Xbf) > 0$ almost everywhere, we define
\[
\E[\xi \mid  A ]=\E[\xi \mid  \Xbf, A ] = \frac{\E[\xi \1_A \mid \Xbf ]}{\P( A \mid \Xbf)} \in \mathcal{L}^2(\Omega, \sigma(\Xbf), \H_2).
\]
Then, we have by construction that
\begin{align}
    \E[\xi \1_A \mid \Xbf ] = \E[\xi \mid  \Xbf, A ] \cdot \P( A \mid \Xbf).
\end{align}

Let again $\Phi(\xbf)=\Phi(P_{\xbf})$ be the embedding of the true conditional distribution into $\H$. 
It has the following three properties.

\begin{lemma}[Lemma 8 in~\citet{DRF-paper}]
It holds that
$\E[\Phi(\diracY) \mid\Xbf \myeq \xbf]= \Phi(P_{\xbf}).$
\end{lemma}

For a more compact notation in the following Lemma, let $N=\{1,\ldots,n\}$, and let for $A \subset N$ and $k \leq |A|$, let $C_{k}(A)$ be the set of all subsets of size $k$ drawn from $A$ without replacement, with $C_0=\emptyset$. The following lemma presents a U-statistic expansion that we afterward apply to an individual tree of our DRF forest.

\begin{lemma}[Lemma 9 in~\citet{DRF-paper}]\label{Hdecomposition}

Let $\left(\H_1, \langle\cdot,\cdot\rangle_1\right)$ and $\left(\H_2, \langle\cdot,\cdot\rangle_2\right)$ be two separable Hilbert spaces, and let $\Zbf_1, \ldots, \Zbf_n$ be i.i.d.\ copies of a random element $\Zbf\colon (\Omega, \A) \to (\H_1, \mathcal{B}(\H_1))$. Write $\mathcal{Z}_n=(\Zbf_1, \ldots, \Zbf_n)$, and let $T\colon(\H_1^n, \mathcal{B}(\H_1^n)) \to  (\H_2, \mathcal{B}(\H_2))$ measurable with 
$\E[\| T(\mathcal{Z}_n) \|^2_{\H_2} ] < \infty$. If $T$ is symmetric, there exist functions $T_j$, $j=1,\ldots,n$, such that
\begin{equation}\label{ANOVA}
    T(\mathcal{Z}_n) = \E[T(\mathcal{Z}_n)] + \sum_{i=1}^n T_1(\Zbf_i)  + \sum_{i_1 < i_2} T_2(\Zbf_{i_1}, \Zbf_{i_2}) + \cdots + T_n(\mathcal{Z}_n),
\end{equation}
and it holds that
\begin{align}\label{vardecomp}
    \Var(T(\mathcal{Z}_n))=\sum_{i=1}^n \binom{n}{i} \Var(T_i(\Zbf_1, \ldots, \Zbf_i) )
\end{align}
and
\[
T_{1}(\Zbf_i)=\E[T(\mathcal{Z}_n)\mid\Zbf_i] - \E[T(\mathcal{Z}_n)].
\]

\end{lemma}

\noindent
Subsequently, we apply this expansion to an individual tree of our DRF predictor. Let $\hmun$ be as in~\eqref{finalestimator}, namely
\begin{align}\label{finalestimator2}
    \hmun = \binom{n}{s_n}^{-1}  \sum_{i_1 < i_2 < \ldots < i_{s_n}} \E_{\varepsilon} \left[ T(\xbf, \varepsilon; \Zbf_{i_1}, \ldots, \Zbf_{i_{s_n}}) \right],
\end{align}
where the sum is taken over all $\binom{n}{s_n}$ possible subsamples $\Zbf_{i_1}, \ldots, \Zbf_{i_{s_n}}$ of $\Zbf_{1}, \ldots, \Zbf_{n}$ and $s_n \to \infty$ with $n$ and where
\begin{align*}
    T(\xbf, \varepsilon; \Zbf_{1}, \ldots, \Zbf_{s_n})=\sum_{j=1}^{s_n} \frac{\1(\Xbf_{j} \in \mathcal{L}(\xbf))}{|\mathcal{L}(\xbf)|} k(\Ybf_j, \cdot ).
\end{align*}
For simplicity we write here the sum from $j=1, \ldots, s_n$, though it should be understood that $\1(\Xbf_{j} \in \mathcal{L}(\xbf))=0$ for $j$ that are used for tree building and not to populate the leaves. 

We introduce the following additional notation similar to Section~\ref{sec:theory}. Let $\Zcal_{s_n}=\left(\Zbf_{1}, \ldots,\Zbf_{s_n}  \right)$ concatenate $s_n$ i.i.d.\ copies of $\Zbf$, and define for $j=1,\ldots, s_n$
\begin{align*}
    \Var(T)&=\Var(T(\xbf, \varepsilon; \Zcal_{s_n})),\\
    \Var(T_j)&=\Var(\E[T(\xbf, \varepsilon; \Zcal_{s_n}) \mid \Zbf_1, \ldots, \Zbf_j ])
\end{align*}
We note that, due to i.i.d.\ sampling, 
what kind of subset $\Zbf_{i_1}, \ldots,\Zbf_{i_{s_n}}$ we are considering
affects neither variance nor expectation, as long as $\Zbf_1, \ldots, \Zbf_j$ are part of $\Zcal_{s_n}$. As such, we always take $\Zcal_{s_n}$ in a slight abuse of notation.
Using composition~\eqref{ANOVA} on $\hmun$ gives
\begin{align}\label{ANOVAmux}
\hmun&= \E[T(\Zcal_{s_n}) ] + \binom{n}{s_n}^{-1}\Big( \binom{n-1}{s_n-1} \sum_{i=1}^n T_{1}(\Zbf_i) +  \binom{n-2}{s_n-2} \sum_{i_1 < i_2} T_{2}(\Zbf_{i_1}, \Zbf_{i_2}) \nonumber \\
&\quad+ \ldots + \sum_{i_1 < i_2 < \ldots < i_{s_n}} T_{s_n}(\Zbf_{i_1}, \ldots, \Zbf_{i_{s_n}}) \Big). 
\end{align}
This representation was used in~\citet{DRF-paper} to prove that the variance of $\hmun$ can be bounded by the scaled variance of a single tree:
\begin{lemma}[Lemma 10 in~\citet{DRF-paper}]\label{variancebound}
Let $\hmun$ be as in~\eqref{finalestimator2}, and assume $T(\xbf, \varepsilon; \Zcal_{s_n})$ satisfies \forestass{3} and
    $\Var(T) < \infty$.
Then,
\begin{align}
   \Var(\hmun )&\leq \frac{s_n^2}{n} \Var(T_1 ) + \frac{s_n^2}{n^2} \Var(T) \\
   &\leq \left( \frac{s_n}{n} + \frac{s_n^2}{n^2} \right) \Var(T). 
\end{align}
\end{lemma}

Subsequently, we derive a first-order approximation of the whole forest and of an individual tree. In the following, we denote the second element of~\eqref{ANOVAmux} by 
\begin{align}\label{firstorderapproxmu}
    \tilde{\mu}_n(\xbf)=\binom{n}{s_n}^{-1} \binom{n-1}{s_n-1} \sum_{i=1}^n T_{1}(\Zbf_i)=\frac{s_n}{n} \sum_{i=1}^n T_{1}(\Zbf_i),
\end{align}
which is the first order approximation of $\mu_n(\xbf)$.
Similarly,  applying~\eqref{ANOVAmux} to a tree $T(\Zcal_{s_n})=\E_{\varepsilon} \left[ T(\xbf, \varepsilon; \Zcal_{s_n}) \right]$, we obtain the expansion
\begin{align*}
    T(\Zcal_{s_n})= \E[T(\Zcal_{s_n})] + \sum_{i=1}^{s_n} T_1(\Zbf_i)  + \sum_{i_1 < i_2} T_2(\Zbf_{i_1}, \Zbf_{i_2}) + \ldots T_{s_n}(\Zcal_{s_n}). 
\end{align*}
Consequently, we define
\begin{align}\label{firstorderapproxT}
    \tilde{T}(\Zcal_{s_n})= \sum_{i=1}^{s_n} T_1(\Zbf_i)=\sum_{i=1}^{s_n} \E[T(\Zcal_{s_n})\mid\Zbf_i] - \E[T(\Zcal_{s_n})].
\end{align}
Contrary to $T(\Zcal_{s_n})$, $\tilde{T}(\Zcal_{s_n})$ is a sum of independent random elements on $\H$ and thus much easier to handle. A key argument will thus be to show that $\tilde{T}(\Zcal_{s_n})$ approximates $T(\Zcal_{s_n})$ asymptotically. 

Consider the leaf $\Lcal(\xbf)$ of the tree $T(\Zcal_{s_n})$ that contains the test point $\xbf$. To emphasize the dependence of such a leaf node on the training data, we will sometimes write $\Lcal(\xbf, \Zcal_{s_n})$ instead of $\Lcal(\xbf)$ in the following.

As in~\citet{meinshausen2006quantile, wager2017estimation}, the crucial part of proving that a Random Forest is consistent is to establish that the diameter of the leaf $\Lcal(\xbf, \Zcal_{s_n})$ goes to zero in probability. In particular, we need a refined result from~\citet{wager2018estimation} below. To save space in our proofs, we subsequently use
\begin{align}
    \xi_i=k(\Ybf_i, \cdot),
\end{align}
for $i=1,\ldots, n$. 

\begin{lemma}[Lemma 2 of~\citet{wager2017estimation}, adapted] \label{lemma2}
Let $T$ be a tree satisfying \forestass{2} and \forestass{4} 
that is trained on data $\Zcal_{s_n}=(\xi_1,\Xbf_1), \ldots, (\xi_{s_n},\Xbf_{s_n})$, and let $\Lcal(\xbf, \Zcal_{s_n})$ be the leaf of $T(\xbf, \varepsilon; \Zcal_{s_n})$ containing $\xbf$. Suppose that assumption~\dataass{1} holds for $\Xbf_1,\ldots, \Xbf_{s_n}$. Then,
\begin{align}
    \P \left( \text{\textup{diam}} (\Lcal(\xbf, \Zcal_{s_n})) \geq \sqrt{p} \left( \frac{s_n}{2k-1}\right) ^{-0.51 \frac{\log((1-\alpha)^{-1})}{\log(\alpha^{-1})} \frac{\pi}{p}}   \right) \leq p \left( \frac{s_n}{2k-1}\right) ^{-1/2 \frac{\log((1-\alpha)^{-1})}{\log(\alpha^{-1})} \frac{\pi}{p}}.
\end{align}
\end{lemma}

\begin{lemma}[Lemma 12 in~\citet{DRF-paper}]\label{helperlemma}
Let $T(\xbf, \varepsilon; \Zcal_{s_n})$ be a tree satisfying \forestass{1} and \forestass{5}, 
and let $\Lcal(\xbf, \Zcal_{s_n})$ be the leaf of $T(\xbf, \varepsilon; \Zcal_{s_n})$ containing $\xbf$. Then,
\begin{align}\label{star1star}
    \E[T(\Zcal_{s_n})] = \E[ \E[\xi_1 \mid \Xbf_1 \in \Lcal(\xbf, \Zcal_{s_n}) ] ]
\end{align}
and 
\begin{align}\label{star2star}
    \Var(T(\Zcal_{s_n}) ) \leq  \sup_{\xbf \in [0,1]^p} \E[ \| \xi_1 \|_{\H}^2 \mid \Xbf=\xbf].
\end{align}
\end{lemma}

\begin{corollarytwo}[Corollary 13 in~\citet{DRF-paper}]\label{bias}
In addition to the conditions of Lemma~\ref{lemma2}, assume \dataass{2} and that the trees $T(\xbf, \varepsilon; \Zcal_{s_n})$ in the forest satisfy \forestass{1} and \forestass{4}. Then, we have
\begin{equation}\label{biasbound}
    \| \E[\hmun] - \mu(\xbf) \|_{\H} = \O\left( s_n^{-1/2 \frac{\log((1-\alpha)^{-1})}{\log(\alpha^{-1})} \frac{\pi}{p}}\right)
\end{equation}
and
\begin{equation}\label{propabiltiyconv1}
\| \E[\xi \mid \Xbf \in \Lcal(\xbf, \Zcal_{s_n})]\|_{\H} \stackrel{p}{\to} \| \E[\xi \mid \Xbf=\xbf] \|_{\H}.
\end{equation}
If moreover \dataass{3} holds, then we have
\begin{align}\label{propabiltiyconv2}
    \E[ \|\xi\|_{\H}^2 \mid \Xbf \in \Lcal(\xbf, \Zcal_{s_n})]  &\stackrel{p}{\to} \E[\|\xi\|_{\H}^2 \mid \Xbf=\xbf].
\end{align}
\end{corollarytwo}

\begin{lemma}\label{vardifboundlemma}
Let $\xi_{1,n}, \xi_{2,n} \in \mathcal{L}^2(\Omega, \mathcal{A}, H)$ for  $n \in \N$, and assume that we have
\begin{itemize}
    \item[(I)] $\Var(\xi_{1,n})=\O(g_1(n))$ and $\Var(\xi_{2,n})=\O(g_1(n))$,
    \item[(II)] $\Var(\xi_{1,n}-\xi_{2,n}) = \O(g_2(n))$
\end{itemize}
for some functions $g_1,g_2\colon \N  \to \N$.
Then, $ |\Var(\xi_{1,n}) - \Var(\xi_{2,n})|= \O(g_2(n)) + \O(\sqrt{g_1(n)} \sqrt{g_2(n)})$.
\end{lemma}

\begin{proof}
It holds that
\begin{align}\label{l2normtrick}
   \left|\sqrt{\Var(\xi_{1,n})} - \sqrt{\Var(\xi_{2,n})}\right|& = \left|\|\xi_{1,n}-\E[\xi_{1,n}]\|_{\mathcal{L}^2} - \|\xi_{2,n}-\E[\xi_{2,n}]\|_{\mathcal{L}^2}\right| \nonumber  \\
   &\leq  \|\xi_{1,n} -  \xi_{2,n} -(\E[\xi_{1,n}] -\E[\xi_{2,n}]) \|_{\mathcal{L}^2} \nonumber\\
   &=\sqrt{\Var(\xi_{1,n}-\xi_{2,n})},
\end{align}
where we used the reverse triangle inequality in the second step.
Thus, we in particular have $\sqrt{\Var(\xi_{1,n})} \leq \sqrt{\Var(\xi_{2,n})} + \sqrt{\Var(\xi_{1,n}-\xi_{2,n})}$
or
\begin{align*}
    \Var(\xi_{1,n}) \leq \Var(\xi_{2,n}) + \Var(\xi_{1,n}-\xi_{2,n}) + 2\sqrt{\Var(\xi_{2,n})} \sqrt{\Var(\xi_{1,n}-\xi_{2,n})}.
\end{align*}
Symmetrically, it holds that
\begin{align*}
    \Var(\xi_{2,n}) \leq \Var(\xi_{1,n}) + \Var(\xi_{1,n}-\xi_{2,n}) + 2\sqrt{\Var(\xi_{1,n})} \sqrt{\Var(\xi_{1,n}-\xi_{2,n})}
\end{align*}
so that by assumption $\Var(\xi_{1,n}) -  \Var(\xi_{2,n}) = \O(g_2(n)) + \O(\sqrt{g_1(n)} \sqrt{g_2(n)})$. 
\end{proof}

Define in the following the number of data points belonging to the same leaf as $\xbf$ as $N_{\xbf}=|\{j\colon \Xbf_j \in \Lcal(\xbf)  \}|$ and let 
\begin{align}\label{Sdef}
    S_i=\frac{\1\{ \Xbf_i \in \Lcal(\xbf) \}}{N_{\xbf}},
\end{align}
be the weight associated with each observation $i$ in a tree $T(\xbf, \varepsilon; \Zcal_{s_n})$, such that
\[
T(\xbf, \varepsilon; \Zcal_{s_n})=\sum_{i=1}^{s_n} S_i k(\mathbf{Y}_i, \cdot).
\]
We will make use the following property of the $S_i$:
\begin{align}\label{eq:expectS1}
    1=\E\left[\sum_{i=1}^{s_{n}} S_i \right] = \sum_{i=1}^{s_{n}}\E[ S_i] = s_n \E[S_1].
\end{align}
In particular,
\begin{align}\label{eq:varDecompSprime2}
\Var(\E[S_1|\Xbf_1])\leq \E[\E[S_1|\Xbf_1]^2] \leq \E[\E[S_1|\Xbf_1]]=\E[S_1]=\O(s_n^{-1})
\end{align}

\begin{lemma}[Lemma 4 of~\citet{wager2017estimation} slightly adapted] \label{lemma4}
Suppose $\Xbf_1,\Xbf_2, \ldots$ are independent and identically distributed on $[0,1]^p$ with a density $f$ that is bounded away from infinity, and let $T(\xbf, \varepsilon; \Zcal_{s_n})$ be $\alpha$-regular \forestass{4}. Then, there is a constant $C_{f,p}$ depending on $f$ and $p$ such that,
\begin{equation}
    s_n \Var(\E[S_1 | \Zbf_1 ]) \succsim \frac{1}{\kappa} \frac{C_{f,p}}{\log(s_n)}
\end{equation}
When $f$ is uniform over $[0,1]^p$, the bound holds with $C_{f,p}=2^{-(p+1)} (p-1)!$
\end{lemma}

Let $\tilde{T}(\Zcal_{s_n})$ be the first order approximation of $T(\Zcal_{s_n})=\E_{\varepsilon}[T(\xbf, \varepsilon; \Zcal_{s_n})]$ as in~\eqref{firstorderapproxT}. We now prove that the variance of $\tilde{T}(\Zcal_{s_n})$ does not decrease to zero too fast compared to the variance of $T(\Zcal_{s_n})$, which is a key result that allows us to meaningfully approximate $T(\Zcal_{s_n})$ with $\tilde{T}(\Zcal_{s_n})$. The main result in~\eqref{vincrementality} is called $\nu(s_n)$-incrementality of the tree $T(\xbf, \varepsilon; \Zcal_{s_n})$ in~\citet[Definition 6]{wager2017estimation}. Before we introduce the result, we note that, due to the orthogonal decomposition in~\eqref{ANOVAmux}, we have 
\begin{align*}
    \Var(\tilde{T}(\Zcal_{s_n}))=s_n \Var(\E[T(\Zcal_{s_n})| \Zbf_1])\leq \Var(T(\Zcal_{s_n})).
\end{align*}
Thus in particular, if $\Var(T(\Zcal_{s_n})) < \infty$, we also have $\Var(\E[T(\Zcal_{s_n})| \Zbf_1])=\O(s_n^{-1})$.

\begin{theorem}\label{majorsteptheorem}
Suppose that the tree $T(\xbf, \varepsilon; \Zcal_{s_n})$ satisfies \forestass{1} and \forestass{4}. Suppose in addition that $\dataass{1}-\dataass{4}$ hold. Then,
\begin{align}\label{claim}
   \Var(\E[T(\Zcal_{s_n})| \Zbf_1]) \succsim \Var(\E[S_1|\Zbf_1]) \Var(\xi|\Xbf=\xbf)
\end{align}
and
\begin{align}\label{vincrementality}
    \frac{\Var(\tilde{T}(\Zcal_{s_n}))}{\Var(T(\Zcal_{s_n}))} \succsim \frac{C_{f,p}}{\log(s_n)^p},
\end{align}
where $C_{f,p}$ is the constant from Lemma~\ref{lemma4}.
\end{theorem}

\begin{proof}
Consider the concatenated data $\Zcal_{s_n}=(\Zbf_1, \ldots, \Zbf_{s_n})$. 
First, assume~\eqref{claim} is true.
In this case, we know from Lemma~\ref{lemma4} that 
\[
\Var(\E[T(\Zcal_{s_n})\mid \Zbf_1]) \succsim \frac{1}{\kappa} \frac{\nu(s_n)}{s_n} \Var(\xi\mid\Xbf=\xbf),
\]
where $\nu(s)=\frac{C_{f,p}}{\log(s)}$.
By Corollary~\ref{bias}, it holds that $\E[\|\xi\|_{\H}^2 | \Xbf \in \Lcal(\xbf,\Zcal_{s_n})] \stackrel{p}{\to} \E[\|\xi\|_{\H}^2 |\Xbf=\xbf]$, 
so that 
\begin{align*}
    \Var(\xi | \Xbf \in \Lcal(\xbf, \Zcal_{s_n})) & = \E[ \| \xi \|_{\H}^2 | \Xbf \in \Lcal(\xbf,\Zcal_{s_n}) ] -\| \E[   \xi  | \Xbf \in \Lcal(\xbf, \Zcal_{s_n})  ]\|_{\H}^2 \stackrel{p}{\to} \Var(\xi \mid \Xbf=\xbf).
\end{align*}
Thus, using the same argument as in the proof of Theorem 5 in~\citet{wager2017estimation}, $\Var(T(\Zcal_{s_n})) \precsim \Var(\xi\mid\Xbf=\xbf)/k$.
Consequently, due to i.i.d.\ sampling, we have
\[
\frac{\Var(\tilde{T}(\Zcal_{s_n}))}{\Var(T(\Zcal_{s_n}))} = \frac{s_n \Var(\E[T(\Zcal_{s_n})\mid\Zbf_1])}{\Var(T(\Zcal_{s_n}))} \succsim \nu(s),
\]
which establishes the result.

Before we verify~\eqref{claim}, we note that, as we use double-sampling, separate data is used for prediction ($\I$) and leaf building ($\I^c$). Consequently, $\Zbf_1$ might fall 
into the prediction set, $1\in \I$, or the leave building set, $ 1\notin \I$. However, only the former case may contribute to the variance:

\begin{claim}
For some $\varepsilon > 0$,
\begin{align}\label{doublesamplerelation}
     \Var(\E[T(\Zcal_{s_n})\mid \Zbf_1]) =  \Var(\E[T(\Zcal_{s_n})\mid \Zbf_1, 1 \in \I]) + \O(s_n^{-(1+\varepsilon)})
\end{align}
\end{claim}

\begin{claimproof}
By assumption, we have $\P(1 \in \I \mid \Zbf_1)=\P(1 \in \I)=1/2$ for each tree. Thus, we have 
\begin{align*}
    \E[T(\Zcal_{s_n})\mid \Zbf_1]&=\E[T(\Zcal_{s_n}) \1\{ 1 \in \I \} \mid \Zbf_1]+ \E[T(\Zcal_{s_n}) \1\{ 1 \notin \I \} \mid \Zbf_1]\\
    &=\frac{1}{2}\E[T(\Zcal_{s_n}) \mid \Zbf_1, \{ 1 \in \I \}] + \frac{1}{2}\E[T(\Zcal_{s_n}) \mid \Zbf_1,\{ 1 \notin \I \}],
\end{align*}
and consequently 
\begin{align*}
    \Var( \E[T(\Zcal_{s_n})\mid \Zbf_1]) &= \frac{1}{4}\Var(\E[T(\Zcal_{s_n})  \mid \Zbf_1, \{ 1 \in \I \}]) + \frac{1}{4}\Var(\E[T(\Zcal_{s_n})  \mid \Zbf_1, \{ 1 \notin \I \}])\\
    &\quad+ \frac{1}{2}\Cov(\E[T(\Zcal_{s_n}) \mid \Zbf_1, \{ 1 \in \I \} ], \E[T(\Zcal_{s_n})  \mid \Zbf_1, \{ 1 \notin \I \}] ).
\end{align*}
Next, using analogous arguments as in~\citet[Corollary 6]{wager2017estimation}, we have
\begin{align*}
    \Var(\E[T(\Zcal_{s_n})  \mid \Zbf_1, \{ 1 \notin \I \}]) =  \O\left( s_n^{-(1+ C_{\alpha} \frac{\pi}{p})}\right)
\end{align*}
with $C_{\alpha}=\frac{\log((1-\alpha)^{-1})}{\log(\alpha^{-1})}$. Finally, since from above 
\[
\Var(\E[T(\Zcal_{s_n})  \mid \Zbf_1, \{ 1 \in \I \}]) \leq \Var(\E[T(\Zcal_{s_n})| \Zbf_1]) = \O(s_n^{-1}),
\]
it follows that
\begin{align*}
    &\left|\Cov(\E[T(\Zcal_{s_n}) \mid \Zbf_1, \{ 1 \in \I \} ], \E[T(\Zcal_{s_n})  \mid \Zbf_1, \{ 1 \notin \I \}] )\right| \\
    \leq&\left(\Var(\E[T(\Zcal_{s_n})  \mid \Zbf_1, \{ 1 \in \I \}])  \Var(\E[T(\Zcal_{s_n})  \mid \Zbf_1, \{ 1 \notin \I \}]) \right)^{1/2}\\
    =&\O(s_n^{-(1+ 1/2 C_{\alpha} \frac{\pi}{p})}).
\end{align*}
Choosing $\varepsilon=1/2 C_{\alpha} \frac{\pi}{p} > 0$ gives the result.
\end{claimproof}

Because the tree $T$ satisfies \forestass{1} and \forestass{4} and due to assumption \dataass{1}, we can apply Lemma~\ref{lemma4}. Thus, once~\eqref{claim} is proven, Lemma~\ref{lemma4} and~\eqref{doublesamplerelation} imply 
\[
\Var(\E[T(\Zcal_{s_n})\mid \Zbf_1, 1 \in \I]) \succsim C \frac{1}{s_n \log(s_n)}, 
\]
so that the remainder term in~\eqref{doublesamplerelation} is negligible. Consequently, we assume for the remainder of the proof that $1 \in \I$ and absorb 
the randomness due to the data $\{\Zbf_i\colon i \notin \I\}$ for building the leaves into the randomness of the tree. In addition, we also write $s_n$ instead of $s_n/2$ in the tree predictions. That is, we write $T(\Zcal_{s_n})=\sum_{i=1}^{s_n} S_i' \xi_i$ although $T(\Zcal_{s_n})=\sum_{i \in \I} S_i' \xi_i$ with $|\I|=s_n/2$ is technically correct. With~\eqref{doublesamplerelation} and i.i.d.\ sampling, this simply amounts to a change of constants.

In the remainder of the proof, we verify~\eqref{claim}. Note that due to honesty, we have
\begin{align}\label{Honestyconsequence}
    \Var(\E[S_1\mid\Zbf_1])= \Var(\E[S_1\mid\xi_1, \Xbf_1])= \Var(\E[S_1\mid\Xbf_1]).
\end{align}
Thus, it is enough to prove~\eqref{claim} with $\Var(\E[S_1\mid\Xbf_1])$. To do this, we use a truncation trick from~\citet{wager2017estimation}. We define
\begin{align}
    T'(\Zcal_{s_n})&=T(\Zcal_{s_n})\1\{\text{diam}(\Lcal(\xbf, \Zcal_{s_n}))\leq s_n^{-w}\}, \label{Tdashdef}\\
    S_i'&=S_i\1\{\text{diam}(\Lcal(\xbf, \Zcal_{s_n}))\leq s_n^{-w}\}, \text{ where } w=\frac{1}{2} \frac{\pi}{p} \frac{\log \left((1-\alpha)^{-1} \right)}{\log(\alpha^{-1})} \label{wdef},
\end{align}
so that
$T'(\Zcal_{s_n})=\sum_{i=1}^{s_n} S_i' \xi_i$. Crucially, $w$ is chosen such that
\begin{align}\label{Wagerprobbound}
    \Prob( \text{diam}(\Lcal(\xbf, \Zcal_{s_n})) > s_n^{-w})=\mathcal{O}(s_n^{-w}) 
\end{align}
This follows from Lemma~\ref{lemma2}, as in~\citet{wager2017estimation}. 

\begin{claim}
\eqref{claim} holds for $T'$. 
\end{claim}

\begin{claimproof}

We start first with a variance lower bound:

\begin{claim}
\begin{align}\label{20}
    \Var( \E[T'(\Zcal_{s_n})\mid \Zbf_1 ])&= \Var(\E[T'(\Zcal_{s_n})\mid\Xbf_1]) + \Var(\E[T'(\Zcal_{s_n})\mid \xi_1,\Xbf_1 ] - \E[T'(\Zcal_{s_n})\mid\Xbf_1]) \nonumber \\
    &\geq  \Var(\E[T'(\Zcal_{s_n})\mid \xi_1,\Xbf_1 ] - \E[T'(\Zcal_{s_n})\mid\Xbf_1]).
\end{align}
\end{claim}

\begin{claimproof}
We need to prove the first equality and start with the decomposition
\begin{align*}
    &\Var( \E[T'(\Zcal_{s_n})\mid \Zbf_1 ])=\Var(\E[T'(\Zcal_{s_n})\mid \xi_1,\Xbf_1 ] - \E[T'(\Zcal_{s_n})\mid\Xbf_1] + \E[T'(\Zcal_{s_n})\mid\Xbf_1]). 
\end{align*}
Consider for $\mathcal{A}=\sigma(\sigma(\Xbf_1),\sigma(\xi_1))$ the space
\[
\mathbb{L}^2(\Omega, \sigma(\Xbf_1), H) \subset \mathbb{L}^2(\Omega, \mathcal{A}, H).
\]
This space is a Hilbert space with the inner product
\[
\langle \xi_1,\xi_2 \rangle_{\mathbb{L}^2} = \E[ \langle \xi_1,\xi_2  \rangle_{H} ].
\]
Moreover, $\E[T'(\Zcal_{s_n})\mid\Xbf_1]=\E[\E[T'(\Zcal_{s_n})\mid\mathcal{A}] \mid \Xbf_1  ]$ is a projection from $\E[T'(\Zcal_{s_n})]\mid\mathcal{A}] \in \mathbb{L}^2(\Omega, \mathcal{A}, H)$ to $\mathbb{L}^2(\Omega, \sigma(\Xbf_1), H)$. Thus, we have
\begin{align*}  &\Cov(\E[T'(\Zcal_{s_n})\mid\Xbf_1],\E[T'(\Zcal_{s_n})\mid\xi_1,\Xbf_1] - \E[T'(\Zcal_{s_n})\mid\Xbf_1])  \\
   = &\langle \E[T'(\Zcal_{s_n})\mid\Xbf_1],\E[T'(\Zcal_{s_n})\mid\xi_1,\Xbf_1] - \E[T'(\Zcal_{s_n})\mid\Xbf_1] \rangle_{\mathbb{L}^2}=0.
\end{align*}
\end{claimproof}

Now, by honesty (i) $\xi_i$ is independent of $S_i'$ conditional on $\Xbf_i$, and more generally, (ii) $\xi_i$ is independent of $S_j'$, $j=1,\ldots,n$, conditional on $\Xbf_i$. Thus, using (i), (ii), and the independence of $\xi_1$ from $\xi_j$, $j > 1$, we have
\begin{align*}
    \E[ T'(\Zcal_{s_n}) \mid\Xbf_1, \xi_1 ] &= \E[S_1' \xi_1\mid\Xbf_1, \xi_1] + \sum_{i=2}^{n} \E[S_i' \xi_i\mid\Xbf_1, \xi_1]\\
    &=\E[S_1' \mid\Xbf_1, \xi_1] \E[ \xi_1\mid\Xbf_1, \xi_1] + \sum_{i=2}^{n} \E[S_i' \xi_i\mid\Xbf_1]\\
    &=\E[S_1' \mid\Xbf_1] \xi_1 + \sum_{i=2}^{n} \E[S_i' \xi_i\mid\Xbf_1]
\end{align*}
Similarly,
\begin{align*}
  \E[ T'(\Zcal_{s_n}) \mid\Xbf_1 ] &=\E[S_1' \xi_1\mid\Xbf_1] + \sum_{i=2}^{n} \E[S_i' \xi_i\mid\Xbf_1]\\
  &=\E[S_1'\mid\Xbf_1]  \E[ \xi_1\mid\Xbf_1] + \sum_{i=2}^{n} \E[S_i' \xi_i\mid\Xbf_1]
\end{align*}
and consequently 
\begin{align}\label{21}
   \Var( \E[ T'(\Zcal_{s_n}) \mid\Xbf_1, \xi_1 ]  - \E[ T'(\Zcal_{s_n}) \mid\Xbf_1 ] )= \Var(\E[S_1' \mid\Xbf_1] (\xi_1 -   \E[ \xi_1\mid\Xbf_1])).
\end{align}
Furthermore, we can refine this statement to:

\begin{claim}
\begin{align}\label{23}
    \Var(\E[S_1' \mid\Xbf_1] (\xi_1 -   \E[ \xi_1\mid\Xbf_1]))  = \Var(\E[S_1' \mid\Xbf_1] (\xi_1 - \mu(\xbf)  )   )   + \mathcal{O}(s_n^{-(1+2w)} ),
\end{align}
where $w$ is defined as in~\eqref{wdef}.
\end{claim}

\begin{claimproof}

We have
\begin{align}\label{22}
    &\Var(\E[S_1' \mid\Xbf_1] (\xi_1 -   \E[ \xi_1\mid\Xbf_1]))  \nonumber \\
    =&\Var(\E[S_1' \mid\Xbf_1] (\xi_1 -   \E[ \xi_1\mid\Xbf_1] + \mu(\xbf)- \mu(\xbf) )) \nonumber \\
    =&\Var(\E[S_1' \mid\Xbf_1] (\xi_1 - \mu(\xbf)  ) - \E[S_1' \mid\Xbf_1]( \E[ \xi_1\mid\Xbf_1]- \mu(\xbf) )  ) \nonumber \\
    =&\Var(\E[S_1' \mid\Xbf_1] (\xi_1 - \mu(\xbf)  )   ) + \Var(\E[S_1' \mid\Xbf_1]( \E[ \xi_1\mid\Xbf_1] - \mu(\xbf) )) \nonumber \\
    &\quad- \Cov(\E[S_1' \mid\Xbf_1] (\xi_1 - \mu(\xbf)), \E[S_1' \mid\Xbf_1]( \E[ \xi_1\mid\Xbf_1]- \mu(\xbf) ) ). 
\end{align}
Because $\E[S_1' \mid\Xbf_1]$ maps into $\R_{\geq 0}$, we have
\begin{align}\label{22_1}
    \Var(\E[S_1' \mid\Xbf_1]( \E[ \xi_1\mid\Xbf_1]- \mu(\xbf))) & \leq \E[ \|\E[S_1' \mid\Xbf_1]( \E[ \xi_1\mid\Xbf_1]) - \mu(\xbf))   \|_{\H}^2 ] \nonumber \\
    &= \E[\E[S_1' \mid\Xbf_1]^2 \| \E[ \xi_1\mid\Xbf_1]- \mu(\xbf)     \|_{\H}^2 ] \nonumber \\
    & \leq \E[\E[S_1'^2 \mid\Xbf_1] C^2\|\Xbf_1 - \xbf \|_{\R^p}^2 ] \nonumber\\
    & \leq \E[S_1'^2  ] C^2 s_n^{-2w},
\end{align}
where we used assumption \dataass{2} for the third inequality and where the last step followed because $\E[S_1'^2 \mid\Xbf_1]=0$, for $\|\Xbf_1 - \xbf \|_{\R^p} > s_n^{-w}$ by definition of $S_1'=S_1\1\{ \text{diam}(\Lcal(\xbf, \Zcal_{s_n}))\leq s_n^{-w}\} $. Due to similar arguments, we have 
\begin{align}\label{22_2}
    &|\Cov(\E[S_1' \mid\Xbf_1] (\xi_1 - \mu(\xbf)), \E[S_1' \mid\Xbf_1]( \E[ \xi_1\mid\Xbf_1]- \mu(\xbf)) )| \nonumber\\
    =&|\E[  \langle \E[S_1' \mid\Xbf_1] (\xi_1 - \mu(\xbf)), \E[S_1' \mid\Xbf_1]( \E[ \xi_1\mid\Xbf_1]- \mu(\xbf) ) \rangle   ]  \nonumber \\
    &\quad -\langle \E[\E[S_1' \mid\Xbf_1] (\xi_1 - \mu(\xbf))], \E[\E[S_1' \mid\Xbf_1](  \E[ \xi_1\mid\Xbf_1]- \mu(\xbf) )] \rangle| \nonumber\\ 
    \leq&  |\E[  \langle \E[S_1' \mid\Xbf_1] (\xi_1 - \mu(\xbf)), \E[S_1' \mid\Xbf_1]( \E[ \xi_1\mid\Xbf_1]- \mu(\xbf) ) \rangle   ]|  \nonumber \\
    &\quad +|\langle \E[\E[S_1' \mid\Xbf_1] (\xi_1 - \mu(\xbf))], \E[\E[S_1' \mid\Xbf_1](  \E[ \xi_1\mid\Xbf_1]- \mu(\xbf) )] \rangle| \nonumber\\
    =&|\E[  \langle \E[S_1' \mid\Xbf_1] (\E[ \xi_1\mid\Xbf_1] - \mu(\xbf)), \E[S_1' |\Xbf_1]( \E[ \xi_1\mid\Xbf_1]- \mu(\xbf) ) \rangle   ]|  \nonumber\\
    &\quad +|\langle \E[\E[S_1' \mid\Xbf_1] (\E[ \xi_1\mid\Xbf_1] - \mu(\xbf))], \E[\E[S_1' \mid\Xbf_1](  \E[ \xi_1\mid\Xbf_1]- \mu(\xbf) )] \rangle|\nonumber\\
    =&\E[  \| \E[S_1' \mid\Xbf_1] (\E[ \xi_1\mid\Xbf_1] - \mu(\xbf)) \|_{\H}^2   ] + \|\E[\E[S_1' \mid\Xbf_1](  \E[ \xi_1\mid\Xbf_1]- \mu(\xbf) )]\|_{\H}^2 \nonumber\\
    \leq&  \E[S_1'^2  ] C^2 s_n^{-2w} + \E[\| \E[S_1' |\Xbf_1](  \E[ \xi_1\mid\Xbf_1]- \mu(\xbf))\|_{\H}^2 ] \nonumber\\
    \leq&  2 \E[S_1'^2  ] C^2 s_n^{-2w}
\end{align}
Observe that we have 
\[
\E[S_1'^2  ] \leq \E[S_1^2  ] \leq \E[S_1  ] = \frac{1}{s_n} \sum_{i=1}^{s_n} \E[S_i]=\frac{1}{s_n}  \E\left[\sum_{i=1}^{s_n} S_i\right] = \frac{1}{s_n},
\]
due to $S_1 \in [0,1]$, $\E[S_1]=\ldots=\E[S_n]$ and $\sum_{i=1}^s S_i=1$.
Finally, combining this observation with~\eqref{22_1} and~\eqref{22_2}
gives~\eqref{23}.

\end{claimproof}

Next, we establish 

\begin{claim}
\begin{align}\label{25}
\Var(\E[S_1' |\Xbf_1] (\xi_1 - \mu(\xbf)  )   ) =  \Var(\E[S_1'  |\Xbf_1])\Var(\xi_1|\Xbf=\xbf)  + \mathcal{O}(s_n^{-(1+w)}) + \mathcal{O}(s_n^{-2}).
\end{align}
\end{claim}

\begin{claimproof}

We have
\begin{align}\label{26}
    &\Var(\E[S_1' |\Xbf_1] (\xi_1 - \mu(\xbf)  )   )  \nonumber \\
    =& \E[ \E[S_1' |\Xbf_1]^2 \|\xi_1 - \mu(\xbf) \|_{\H}^2  ] - \E[\E[S_1' |\Xbf_1] \|\xi_1 - \mu(\xbf)  \|_{\H} ]^2 \nonumber \\
    =& \E \left[ \E[S_1' |\Xbf_1]^2 \E[\|\xi_1 - \mu(\xbf) \|^2 | \Xbf_1]  \right] - \E \left[\E[S_1' |\Xbf_1] \E[ \|\xi_1 - \mu(\xbf) \|_{\H} |\Xbf_1]   \right]^2.
\end{align}
The second term in~\eqref{26} can be bounded by 
\begin{align*}
    \E \left[\E[S_1' |\Xbf_1] \E[ \|\xi_1 - \mu(\xbf) \|_{\H} |\Xbf_1]  ) \right]^2 &\leq \E \left[\E[S_1' |\Xbf_1] \E[ \|\xi_1 \|_{\H} + \| \mu(\xbf) \|_{\H} |\Xbf_1]   \right]^2 \\
   & =\E \left[ \E[S_1' |\Xbf_1]  \left( \E[ \|\xi_1 \|_{\H}  |\Xbf_1]   + \| \mu(\xbf) \|_{\H}  \right)  \right]^2\\
    &\leq \E \left[ \E[S_1' |\Xbf_1]  \left( \sup_{\xbf \in [0,1]^p} \E[ \|\xi_1 \|_{\H}  |\Xbf=\xbf]   + \| \mu(\xbf) \|_{\H}  \right)  \right]^2\\
    &= \E \left[ \E[S_1' |\Xbf_1]    \right]^2 \left( \sup_{\xbf \in [0,1]^p} \E[ \|\xi_1 \|_{\H}  |\Xbf=\xbf]   + \| \mu(\xbf) \|_{\H}  \right)^2\\
    &= \E[S_1']^2 \left( \sup_{\xbf \in [0,1]^p} \E[ \|\xi_1 \|_{\H}  |\Xbf=\xbf]   + \| \mu(\xbf) \|_{\H}  \right)^2\\
    &= \mathcal{O}(s_n^{-2}).
\end{align*}
The last step followed because of~\eqref{supbound}, a consequence of \kernelass{1}. 
The first term in~\eqref{26} can be bounded by
\begin{align}\label{26_1}
    &\E \Big[ \E[S_1' |\Xbf_1]^2 \E[\|\xi_1 - \mu(\xbf) \|^2 | \Xbf_1]  \Big] \nonumber\\
    =&\E \Big[ \E[S_1' |\Xbf_1]^2 \left(\E[\|\xi_1 - \mu(\xbf) \|^2 | \Xbf_1] -  \E[\|\xi_1 - \mu(\xbf) \|^2 | \Xbf=\xbf]+ \E[\|\xi_1 - \mu(\xbf) \|^2 | \Xbf=\xbf] \right) \Big] \nonumber \\
    =&\E \Big[ \E[S_1' |\Xbf_1]^2 \left(\E[\|\xi_1 - \mu(\xbf) \|^2 | \Xbf_1] -  \E[\|\xi_1 - \mu(\xbf) \|^2 | \Xbf=\xbf] \right)\Big] + \E[ \E[S_1' |\Xbf_1]^2]\Var(\xi_1|\Xbf=\xbf) .
\end{align}
Because $S_1'$ is defined as $S_1'=S_1\1\{ \text{diam}(\Lcal(\xbf, \Zcal_{s_n})) \leq s_n^{-w}\} $, it is zero if $\|\Xbf_1 - \xbf\|_{\R^p} > s_n^{-w}$. Combining this with Assumption~\dataass{2} and \dataass{3} it follows that 
\begin{align}\label{26_2}
   &\left| \E \left[ \E[S_1'  |\Xbf_1]^2 (\E[\|\xi_1 - \mu(\xbf) \|_{\H}^2 | \Xbf_1] -  \E[\|\xi_1 - \mu(\xbf) \|_{\H}^2 \mid \Xbf=\xbf])\right]\right|  \nonumber \\
    \leq&\E \left[ \E[S_1'  |\Xbf_1]^2  \left |\E[\|\xi_1 - \mu(\xbf) \|_{\H}^2 \mid \Xbf_1] -  \E[\|\xi_1 - \mu(\xbf) \|_{\H}^2 \mid \Xbf=\xbf] \right|\right] \nonumber \\
     =&\E \big[ \E[S_1'  |\Xbf_1]^2  \big |\E[\|\xi_1 \|_{\H}^2 \mid \Xbf_1] + \|\mu(\xbf)\|_{\H}^2 - 2\langle \E[\xi \mid \Xbf_1], \mu(\xbf) \rangle -  \E[\|\xi_1 \|_{\H}^2 \mid \Xbf=\xbf] - \nonumber \\
     &\|\mu(\xbf)\|_{\H}^2 + 2\langle \E[\xi \mid \Xbf=\xbf], \mu(\xbf) \rangle \big|\big] \nonumber \\
     \leq &\E \big[ \E[S_1'  |\Xbf_1]^2 \left( \big |\E[\|\xi_1 \|_{\H}^2 \mid \Xbf_1]-  \E[\|\xi_1 \|_{\H}^2 \mid \Xbf=\xbf] \big | + 2 \big| \langle \E[\xi \mid \Xbf=\xbf]-\E[\xi \mid \Xbf_1], \mu(\xbf) \rangle  \big| \right) \big] \nonumber \\
    \leq&\E \left[ \E[S_1'  |\Xbf_1]^2  \right](C_1 s_n^{-w} + C_2s_n^{-w}) \nonumber \\
    =&\mathcal{O}(s_n^{-(1+w)})
\end{align}
holds, where we used $\E \left[ \E[S_1'  |\Xbf_1]^2  \right]=\mathcal{O}(s_n^{-1})$. Finally, due to 
\begin{align}\label{26_3}
    \E[\E[S_1'  |\Xbf_1]^2] =& \Var(\E[S_1'  |\Xbf_1]) + \E[\E[S_1'  |\Xbf_1] ]^2 \nonumber\\
    =& \Var(\E[S_1'  |\Xbf_1]) + \E[S_1']^2 \nonumber\\
    =& \Var(\E[S_1'  |\Xbf_1]) + \mathcal{O}(1/s_n^2),
\end{align}
we can combine~\eqref{26}-\eqref{26_3} to establish our claim~\eqref{25}.

\end{claimproof}

Combining~\eqref{20},~\eqref{21},~\eqref{23}, and~\eqref{25}, we get that~\eqref{claim} holds for $T'$ due to
\begin{align}\label{truncation_result}
    \Var( \E[T'(\Zcal_{s_n})| \Zbf_1 ]) &\geq  \Var( \E[ T'(\Zcal_{s_n}) |\Xbf_1, \xi_1 ]  - \E[ T'(\Zcal_{s_n}) |\Xbf_1 ] )\nonumber\\
    &= \Var(\E[S_1' |\Xbf_1] (\xi_1 -   \E[ \xi_1|\Xbf_1])) \nonumber\\
                                    &=\Var(\E[S_1' |\Xbf_1] (\xi_1 - \mu(\xbf)  )   )   + \mathcal{O}(s_n^{-(1+2w)} ) \nonumber \\
                                    &= \Var(\E[S_1'  |\Xbf_1])\Var(\xi_1|\Xbf=\xbf)  + \mathcal{O}(s_n^{-(1+w)}) + \mathcal{O}(s_n^{-2})  + \mathcal{O}(s_n^{-(1+2w)} ).
\end{align}

\end{claimproof}

In the next step we replace $S_1'$ with $S_1$ in the expression above.

\begin{claim}
\begin{align}\label{truncation_S}
     \left|\Var(\E[S_1'  |\Xbf_1]) -  \Var(\E[S_1  |\Xbf_1])\right| = \mathcal{O}(s_n^{-(1+w/2)}).
\end{align}
\end{claim}

\begin{claimproof}
We have 
\begin{align}\label{eq:varDecompSprime}
    \Var(\E[S_1|\Xbf_1] - \E[S_1'  |\Xbf_1]) &= \Var( \E[S_1 - S_1'  |\Xbf_1])\nonumber\\
    &= \Var( \E[S_1 \1\{ \text{diam}(\Lcal(\xbf, \Zcal_{s_n})) > s_n^{-w}\}  |\Xbf_1])\nonumber\\
    &\leq \E[\E[S_1 \1\{ \text{diam}(\Lcal(\xbf, \Zcal_{s_n})) > s_n^{-w}\}|\Xbf_1]^2 ]\nonumber\\
    & \leq \E[S_1\1\{ \text{diam}(\Lcal(\xbf, \Zcal_{s_n})) > s_n^{-w}\} ]\nonumber\\
    & = \frac{1}{s} \sum_{i=1}^s \E[S_i \1\{ \text{diam}(\Lcal(\xbf, \Zcal_{s_n})) > s_n^{-w}\} ] \nonumber\\
    &= \frac{1}{s} \E[\1\{ \text{diam}(\Lcal(\xbf, \Zcal_{s_n})) > s_n^{-w}\}  \sum_{i=1}^s S_i  ]\nonumber\\
    &=\frac{1}{s}  \Prob( \text{diam}(\Lcal(\xbf, \Zcal_{s_n})) > s_n^{-w})\nonumber\\
    & = \mathcal{O}(s_n^{-(1+w)})
\end{align}
due to $\sum_{i=1}^{s_n} S_i=1$ and where the last step followed due to~\eqref{Wagerprobbound}. As $\Var(\E[S_1  |\Xbf_1])=\O(s_n^{-1})$ from~\eqref{eq:varDecompSprime2} and analogously $\Var(\E[S_1'  |\Xbf_1])=\O(s_n^{-1})$, it holds by Lemma~\ref{vardifboundlemma} and~\eqref{eq:varDecompSprime} that 
\begin{align*}
    |\Var(\E[S_1'  |\Xbf_1]) -  \Var(\E[S_1  |\Xbf_1])|=\mathcal{O}(s_n^{-(1+w)}) + \mathcal{O}(s_n^{-((2+w)/2)})=\mathcal{O}(s_n^{-(1+w/2)}).
\end{align*}
\end{claimproof}

Thus, we have
\begin{align}\label{truncation_result_{s_n}}
    \Var( \E[T'(\Zcal_{s_n})| \Zbf_1 ]) &\geq \Var(\E[S_1 |\Xbf_1])\Var(\xi_1|\Xbf=\xbf)  + \mathcal{O}(s_n^{-(1+\varepsilon)}),
\end{align}
for some $\varepsilon > 0$. Because we have $\Var(\xi_1|\Xbf=\xbf) > 0$ by assumption and due to $\Var(\E[S_1  |\Xbf_1]) = \Var(\E[S_1  |\Zbf_1]) \succsim C (s_n\log(s_n))^{-1} $ by Lemma~\ref{lemma4}, we finally have
\begin{align}
   \liminf_{n \to \infty} \frac{\Var( \E[T'(\Zcal_{s_n})| \Zbf_1 ])}{\Var(\E[S_1  |\Xbf_1])\Var(\xi_1|\Xbf=\xbf) }  \geq 1,
\end{align}
or~\eqref{claim} for $T'(\Zcal_{s_n})$ instead of $T(\Zcal_{s_n})$.

Now, it also holds that:

\begin{claim}
\begin{align}\label{BounddiffbetweenTandTd}
|\Var( \E[T(\Zcal_{s_n})| \Zbf_1 ]) -  \Var( \E[T'(\Zcal_{s_n})| \Zbf_1 ])| =   \mathcal{O}(s_n^{-(1+w/2)})    
\end{align}

\end{claim}

\begin{claimproof}
First, observe that we have
\begin{align*}
    \Var( \E[T(\Zcal_{s_n})| \Zbf_1 ] - \E[T'(\Zcal_{s_n})| \Zbf_1 ]) &= \Var( \E[T(\Zcal_{s_n})  \1\{ \text{diam}(\Lcal(\xbf, \Zcal_{s_n})) > s_n^{-w} \}  | \Zbf_1 ]).
\end{align*}
Using composition~\eqref{ANOVA} on $T''(\Zcal_{s_n})=T(\Zcal_{s_n})  \1\{ \text{diam}(\Lcal(\xbf, \Zcal_{s_n})) > s_n^{-w} \}$, we have
\begin{align*}
T''(\Zcal_{s_n}) &= \E[T''(\Zcal_{s_n})] + \sum_{i=1}^{s_n} T''_1(\Zbf_i)  + \sum_{i_1 < i_2} T''_2(\Zbf_{i_1}, \Zbf_{i_2}) + \cdots + T''_{s_n}(\Zcal_{s_n}),\\
       \Var(T''(\Zcal_{s_n})) &=\sum_{i=1}^{s_n} \binom{s_n}{i} \Var(T''_i\left( \Zbf_1, \ldots, \Zbf_i \right) ), \\
       T''_1(\Zbf_1) &= \E[ T''(\Zcal_{s_n}) | \Zbf_1 ] - \E[T''(\Zcal_{s_n})],
\end{align*}
and thus 
\begin{align*}
    &\Var( \E[T(\Zcal_{s_n})  \1\{ \text{diam}(\Lcal(\xbf, \Zcal_{s_n})) > s_n^{-w} \}  | \Zbf_1 ])\\
    &= \Var(T''_1(\Zbf_1) )\\
    &\leq  \frac{1}{s_n} \sum_{i=1}^{s_n} \binom{s_n}{i} \Var(T''_i\left( \Zbf_1, \ldots, \Zbf_i \right) )\\
    &= \frac{1}{s_n} \Var(T''(\Zcal_{s_n}))\\
    &= \frac{1}{s_n} \Var(T(\Zcal_{s_n})  \1\{ \text{diam}(\Lcal(\xbf, \Zcal_{s_n})) > s_n^{-w} \} ).
\end{align*}
Moreover, with analogous arguments as in the proof of Lemma 12 in~\citet{DRF-paper} it can be shown that,
\begin{align*}
&\Var(T(\Zcal_{s_n})  \1\{ \text{diam}(\Lcal(\xbf, \Zcal_{s_n})) > s_n^{-w} \} )\\
    &\leq \E \left[ \left\| \sum_{i=1}^{s} S_i \xi_i \right \|_{\mathcal{H}}^2 \1\{ \text{diam}(\Lcal(\xbf, \Zcal_{s_n})) > s_n^{-w} \}\right]\\
    &\leq C \sup_{\xbf \in [0,1]^p} \E[\|\xi \|_{\H}^2 | \Xbf=\xbf] \Prob(\text{diam}(\Lcal(\xbf, \Zcal_{s_n})) > s_n^{-w})
\end{align*}
such that
\begin{align*}
&\Var( \E[T(\Zcal_{s_n})  \1\{ \text{diam}(\Lcal(\xbf, \Zcal_{s_n})) > s_n^{-w} \}  | \Zbf_1 ])\\
    &\leq \frac{C \sup_{\xbf \in [0,1]^p} \E[\|\xi \|_{\H}^2 | \Xbf=\xbf]}{s_n} \Prob(\text{diam}(\Lcal(\xbf, \Zcal_{s_n})) > s_n^{-w})\\
    &= \mathcal{O}(s_n^{-(1+w)}),
\end{align*}
where the last step follows from~\eqref{Wagerprobbound} and~\eqref{supbound}. As also $\Var( \E[T(\Zcal_{s_n})| \Zbf_1 ]) \leq \Var(T(\Zcal_{s_n}))/s_n=\O(s_n^{-1})$ and similarly for $T'$, the claim holds by Lemma~\ref{vardifboundlemma}, similar to the proof of~\eqref{truncation_S} above.
\end{claimproof}

Summarizing everything, it follows from~\eqref{BounddiffbetweenTandTd},~\eqref{truncation_result} and~\eqref{truncation_S},
\begin{align*}
    \Var(\E[T(\Zcal_{s_n})| \Zbf_1]) & = \Var( \E[T'(\Zcal_{s_n})| \Zbf_1 ]) + \mathcal{O}(s_n^{-(1+w/2)})\\
   & \geq  \Var(\E[S_1'  |\Xbf_1])\Var(\xi_1|\Xbf=\xbf)  + \mathcal{O}(s_n^{-(1+w)}) + \mathcal{O}(s_n^{-2})  + \mathcal{O}(s_n^{-(1+2w)} )\\
   & =   \Var(\E[S_1  |\Xbf_1])\Var(\xi_1|\Xbf=\xbf)  + \mathcal{O}(s_n^{-(1+w)}) + \mathcal{O}(s_n^{-2})  + \mathcal{O}(s_n^{-(1+2w)} ).
\end{align*}
Because $\Var(\E[S_1  |\Xbf_1]) \succsim  C (s_n \log(s_n))^{-1}$ by Lemma~\ref{lemma4} and $\Var(\xi_1|\Xbf=\xbf) > 0$ by assumption, this implies that
\begin{align}
    \liminf_{s_n} \frac{\Var(\E[T(\Zbf)| \Zbf_1])}{ \Var(\E[S_1  |\Xbf_1])\Var(\xi_1|\Xbf=\xbf)} \geq 1,
\end{align}
or $\Var(\E[T(\Zbf)| \Zbf_1]) \succsim \Var(\E[S_1  |\Xbf_1])\Var(\xi_1|\Xbf=\xbf)$, proving~\eqref{claim}. 

\end{proof}

\asymptoticlinearity*

\begin{proof}

Let $\tilde{\mu}_n(\xbf) $ and $\tilde{T}(\Zcal_{s_n})$ be as in~\eqref{firstorderapproxmu} and~\eqref{firstorderapproxT}, respectively, and observe that we have
\[
\sigma_n^2=\Var(\tilde{\mu}(\xbf))= \frac{s_n^2}{n} \Var(T_1) = \frac{s_n}{n} s_n \Var(T_1) =\frac{s_n}{n}  \Var(\tilde{T}(\Zcal_{s_n}))  \leq \frac{s_n}{n} \Var(T).
\]

We first prove~\eqref{asymptoticlin} for $\hmun -  \E[\hmun]$.

\begin{claim}
\eqref{asymptoticlin} holds for $ \E[\hmun]$ in place of $\mu_n(\xbf) $:
\begin{align}\label{claimwithexpectation}
    \hat{\mu}(\xbf) - \E[\hat{\mu}(\xbf)] = \frac{s_n}{n}\sum_{i=1}^n (\E[T(\Zcal_n)\mid\Zbf_i] - \E[T(\Zcal_n)])+ \o_p(\sigma_n)
\end{align}
\end{claim}

\begin{claimproof}
First, 

\begin{claim}
\begin{align}\label{eq: Convergenceinsquaredmean}
   \frac{1}{\sigma_n^2} \E[\|  \hmun - \tilde{\mu}_n(\xbf) \|_{\H}^2]\precsim  \frac{s_n}{n}   \frac{ \log(s_n)^p}{ C_{f,p}} \to 0 
\end{align}

\end{claim}

\begin{claimproof}
Let $(s_n)_{j}=s_n(s_n-1)\cdots (s_n-(j-1))=s_n!/(s_n-j)!$ and $\Var(\tilde{T})=\Var(\tilde{T}(\Zcal_{s_n}))$. Then, using the decomposition in~\eqref{ANOVAmux} with $\Var(T_j)=\Var(T_{j}(\Zbf_{1}, \Zbf_{2}, \ldots, \Zbf_{j} ))$, $j=1,\ldots, s_n$  that 
\begin{align*}
    &\frac{1}{\sigma_n^2} \E[\|  \hat{\mu}(\xbf) - \tilde{\mu}(\xbf) \|_{\H}^2]\\
    =& \frac{1}{\sigma_n^2}  \Var \Big( \binom{n}{s_n}^{-1}\Big(   \binom{n-2}{s_n-2} \sum_{i_1 < i_2} T_{2}(\Zbf_{i_1}, \Zbf_{i_2})+ \ldots + \sum_{i_1 < i_2 < \ldots < i_{s_n}} T_{s_n}(\Zbf_{i_1}, \ldots, \Zbf_{i_{s_n}}) \Big) \Big) \\
    =& \frac{1}{\sigma_n^2}  \sum_{i=2}^{s_n} \left( \frac{(s_n)_{i}}{(n)_{i}}\right)^2 \binom{n}{i} \Var(T_i)\\
    =& \frac{1}{\sigma_n^2} \sum_{i=2}^{s_n} \left(\frac{(s_n)_{i}}{(n)_{i}}\right) \binom{s_n}{i}\Var(T_i)\\
    \leq&  \frac{1}{\sigma_n^2}  \frac{(s_n)_2}{(n)_2} \sum_{i=2}^{s_n}  \binom{s_n}{i}\Var(T_i)\\
    \leq&   \frac{s_n^2}{n^2} \frac{\Var(T)}{\sigma_{n}^2}\\
    =&\frac{s_n}{n}   \frac{\Var(T)}{\Var(\tilde{T})}\\
     \precsim&   \frac{s_n}{n}   \frac{ \log(s_n)^p}{ C_{f,p}},
\end{align*}
where we used Theorem~\ref{majorsteptheorem} in the last step. Finally, since $s_n=n^{\beta}$ for $\beta < 1$, we infer $(s_n \log(s_n)^p)/n \to 0$.
\end{claimproof}

Since by construction $\E[\hmun  ] = \E[ \tilde{\mu}_n(\xbf)]$, for all $\varepsilon > 0$, we have
\[
\P\left( \left \|  \frac{1}{\sigma_n} (\hat{\mu}(\xbf) - \E[\hat{\mu}(\xbf)]) -   \frac{1}{\sigma_n} (\tilde{\mu}(\xbf) - \E[\tilde{\mu}(\xbf)]) \right \|_{\H}^2 > \varepsilon \right) \leq \frac{1}{\varepsilon^2} \frac{1}{\sigma_n^2} \E[\|  \hat{\mu}(\xbf) - \tilde{\mu}(\xbf) \|_{\H}^2].
\]
Consequently, we have $ \|\frac{1}{\sigma_n} (\hat{\mu}(\xbf) - \E[\hat{\mu}(\xbf)]) -   \frac{1}{\sigma_n} (\tilde{\mu}(\xbf) - \E[\tilde{\mu}(\xbf)]) \|_{\H} \to 0$ in probability, or equivalently
\[
\hat{\mu}(\xbf) - \E[\hat{\mu}(\xbf)] = \tilde{\mu}(\xbf) - \E[\tilde{\mu}(\xbf)] + o_p(\sigma_n).
\]
Since moreover
\begin{align*}
 \tilde{\mu}(\xbf) - \E[\tilde{\mu}(\xbf)]  & = \frac{s_n}{n}\sum_{i=1}^n T_{1}(\Zbf_i)=\frac{s_n}{n}\sum_{i=1}^n (\E[T(\Zcal_n)\mid\Zbf_i] - \E[T(\Zcal_n)]),
\end{align*}
we conclude Claim~\eqref{claimwithexpectation}.
\end{claimproof}

Due to
\[
\frac{1}{\sigma_n}  \left\| \hmun -  \mu(\xbf) \right\|_{\H} \leq  \frac{1}{\sigma_n}\left\| \hmun -  \E[\tilde{\mu}(\xbf)] \right\|_{\H} + \frac{1}{\sigma_n}\left\| \E[\tilde{\mu}(\xbf)] -  \mu(\xbf) \right\|_{\H}, 
\]
the result follows if we can show that the second expression in this upper bound goes to zero 
\[
\Var(\E[T(\Zcal_{s_n})| \Zbf_1]) \succsim \frac{1}{\kappa} \frac{C_{f,p}}{s_n \log(s_n)^p} \Var(\xi|\Xbf=\xbf) > 0,
\]
so that
\begin{align*}
    \sigma_n^2=& \frac{s_n^2}{n} \Var(T_1)\\
    =&\frac{s_n^2}{n} \Var(\E[T(\Zcal_{s_n}) | \Zbf_1]) \\
    \succsim&  \frac{s_n^2}{n} \frac{1}{\kappa} \frac{C_{f,p}}{s_n \log(s_n)^p} \Var(\xi|\Xbf=\xbf)\\
    =& \frac{ s_n}{n \log(s_n)^p} \Var(\xi|\Xbf=\xbf) \frac{1}{\kappa}  C_{f,p}.
\end{align*}
Thus, using that $s_n=n^{\beta}$, we have
\begin{align*}
    \sigma_n = \Omega \left(\frac{ \sqrt{s_n}}{\sqrt{n \log(s_n)^p}} \right)=\Omega \left(\left(\frac{n^{\beta}}{n \beta^p \log(n)^p} \right)^{1/2}\right)=\Omega \left(\left(n^{\beta-1-\varepsilon}\right)^{1/2}\right)
\end{align*}
for some $\varepsilon > 0$ 
On the other hand, due to Theorem~\ref{bias}, we have
\[
\| \E[\hat{\mu}(\xbf)] - \mu(\xbf) \|_{\H} = \mathcal{O}\left( s_n^{-1/2 \frac{\log((1-\alpha)^{-1})}{\log(\alpha^{-1})} \frac{\pi}{p}}\right)= \mathcal{O}\left( s_n^{-1/2 C_{\alpha} \frac{\pi}{p}}\right) = \mathcal{O}\left( n^{-1/2 \beta C_{\alpha} \frac{\pi}{p}}\right),
\]
which implies
\begin{align*}
    \frac{\| \E[\hat{\mu}(\xbf)] - \mu(\xbf) \|_{\H}}{\sigma_n} &=  \mathcal{O}\left( n^{-1/2 (\beta C_{\alpha} \frac{\pi}{p}  + \beta - 1 - \varepsilon )}\right)= \mathcal{O}\left( n^{-1/2 (\beta ( 1 +  C_{\alpha} \frac{\pi}{p} ) - 1 - \varepsilon )}\right).
\end{align*}

This goes to zero  provided that $-(\beta ( 1 +  C_{\alpha} \frac{\pi}{p} ) - 1 - \varepsilon ) < 0$ or $\beta > (1+\varepsilon) \left( 1 +  C_{\alpha} \frac{\pi}{p} \right)^{-1}$, which is satisfied for $\varepsilon > 0$ small enough if
\[
\beta >  \left( 1 +  C_{\alpha} \frac{\pi}{p} \right)^{-1}.
\]

Taking $T_{n}(\Zbf_i)=\E[T(\Zcal_{s_n})\mid\Zbf_i] - \E[T(\Zcal_{s_n})]$ gives the claimed result.
\end{proof}

Before being able to prove Theorem~\ref{thm: asymptoticnormality} in the main text, we need to refine the characterization of the asymptotic behavior of the variance of $T_n(\Zbf_i)$.

\varianceprop*

\begin{proof}

Note that, due to~\eqref{doublesamplerelation}, we can again ``ignore'' the double-sampling and assume to condition on a point $\Zbf_1$ with index in the prediction set $\I$ and use $s_n$ instead of $s_n/2$ elements in the tree predictions.
First, due to $T_{n}(\Zbf_1)=\E[T(\Zcal_{s_n})| \Zbf_1 ] - \E[T(\Zcal_{s_n})]$, we infer
\begin{align*}
    \frac{\Var(\langle T_{n}(\Zbf_1), f \rangle)}{\Var(T_{n}(\Zbf_1))}= \frac{\Var( \E[\langle T(\Zcal_{s_n}), f \rangle \mid \Zbf_1 ])}{\Var( \E[T(\Zcal_{s_n})| \Zbf_1 ])}.
\end{align*}
Combining~\eqref{20} with~\eqref{truncation_result} in Theorem~\ref{majorsteptheorem}, we have 
\begin{align}\label{varexpansion}
    \Var( \E[\langle T'(\Zcal_{s_n}), f \rangle \mid \Zbf_1 ])&= \Var(\E[ \langle T'(\Zcal_{s_n}), f \rangle \mid \Xbf_1 ])+ \Var(\E[S_1'  \mid \Xbf_1])\Var(\langle \xi_1, f \rangle\mid \Xbf=\xbf)  + \mathcal{O}(s_n^{-(1+\epsilon)}), \nonumber \\
\Var( \E[T'(\Zcal_{s_n})| \Zbf_1 ])&= \Var(\E[ T'(\Zcal_{s_n}) |\Xbf_1 ])+ \Var(\E[S_1'  \mid \Xbf_1])\Var(\xi_1 \mid \Xbf=\xbf)  + \mathcal{O}(s_n^{-(1+\epsilon)})
\end{align}
for some $\epsilon > 0$. Let in the following $\1_{w,s_n}=\1\{\text{diam}(\Lcal(\xbf, \Zcal_{s_n}))\leq s_n^{-w}\}$ such that $ S_i'= S_i \1_{w,s_n}$.
We now show that
\begin{claim}
\begin{align}\label{Varboundnew}
    \Var(\E[ \langle T'(\Zcal_{s_n}), f \rangle \mid \Xbf_1 ]) = \mathcal{O}(s_n^{-(1+\epsilon)}) \nonumber\\
\Var( \E[T'(\Zcal_{s_n})| \Xbf_1 ])= \mathcal{O}(s_n^{-(1+\epsilon)}).
\end{align}

\end{claim}

\begin{claimproof}
First, due to honesty, we have
\begin{align*}
   \E[ \langle T'(\Zcal_{s_n}),f \rangle \mid \Xbf_1 ] &=\E[S_1'\mid \Xbf_1]  \E[ \langle \xi_1,f \rangle\mid  \Xbf_1] + \sum_{i=2}^{s_n} \E[S_i' \langle\xi_i,f \rangle\mid \Xbf_1]
\end{align*}
and
\begin{align*}
  \E[ T'(\Zcal_{s_n}) \mid \Xbf_1 ] &=\E[S_1'\mid \Xbf_1]  \E[ \xi_1\mid \Xbf_1] + \sum_{i=2}^{s_n} \E[S_i' \xi_i\mid \Xbf_1].
\end{align*}
Subsequently, we consider the variance of the two terms and their covariance individually. First, we study the variance of the first terms. The variances satisfy
\begin{claim}
\begin{align}\label{variance1term_1}
    \Var(\E[S_1'|\Xbf_1]  \E[ \xi_1|\Xbf_1])&= \Var(\E[S_1'|\Xbf_1]) \|\E[\xi_1|\Xbf=\xbf]\|_{\H}^2 + \mathcal{O}(s_n^{-(1+w)})
\end{align}
and
\begin{align}\label{variance1term_2}
        \Var(\E[S_1'|\Xbf_1]  \E[ \langle \xi_1,f \rangle|\Xbf_1]) &= \Var(\E[S_1'|\Xbf_1]) \E[\langle \xi_1, f \rangle|\Xbf=\xbf]^2 + \mathcal{O}(s_n^{-(1+w)}).
\end{align}
\end{claim}

\begin{claimproof}

We only show~\eqref{variance1term_1} because~\eqref{variance1term_2} follows analogously. We have
\begin{align}
    \Var(\E[S_1' |\Xbf_1]   \E[ \xi_1|\Xbf_1]) &=\Var(\E[S_1' |\Xbf_1] (  \E[ \xi_1|\Xbf_1] - \mu(\xbf)+ \mu(\xbf) )) \nonumber \\
    &=  \Var(\E[S_1' |\Xbf_1] \mu(\xbf)     ) + \Var(\E[S_1' |\Xbf_1]( \E[ \xi_1|\Xbf_1] - \mu(\xbf) )) \nonumber \\
    &\quad+ \Cov\left(\E[S_1' |\Xbf_1] \mu(\xbf), \E[S_1' |\Xbf_1]( \E[ \xi_1|\Xbf_1] - \mu(\xbf) ) \right) . 
\end{align}
Because 
\begin{align*}
    \Var(\E[S_1' |\Xbf_1] \mu(\xbf)     )
    =&\E[ \|(\E[S_1' |\Xbf_1] - \E[S_1']) \mu(\xbf)\|_{\H}^2 ]\\
    =&\E[ (\E[S_1' |\Xbf_1] - \E[S_1'])^2] \|\mu(\xbf)\|_{\H}^2\\
    =&\Var(\E[S_1' |\Xbf_1])\|\mu(\xbf)\|_{\H}^2,
\end{align*}
it follows that 
\begin{align*}
    \Var(\E[S_1' |\Xbf_1]   \E[ \xi_1|\Xbf_1]) &= \Var(\E[S_1' |\Xbf_1])\|\mu(\xbf)\|_{\H}^2+ \Var(\E[S_1' |\Xbf_1]( \E[ \xi_1|\Xbf_1] - \mu(\xbf) )) \nonumber \\
    &\quad+ \Cov\left(\E[S_1' |\Xbf_1] \mu(\xbf), \E[S_1' |\Xbf_1]( \E[ \xi_1|\Xbf_1] - \mu(\xbf) ) \right).
\end{align*}
Because $\E[S_1' |\Xbf_1]$ maps into $\R_{\geq 0}$, we have
\begin{align*}
    \Var(\E[S_1' |\Xbf_1]( \E[ \xi_1|\Xbf_1]- \mu(\xbf))) & \leq \E[ \|\E[S_1' |\Xbf_1]( \E[ \xi_1|\Xbf_1]) - \mu(\xbf))   \|_{\H}^2 ] \\
    &= \E[\E[S_1' |\Xbf_1]^2 \|( \E[ \xi_1|\Xbf_1]- \mu(\xbf)    ) \|_{\H}^2 ]\\
    & \leq \E[\E[S_1'^2 |\Xbf_1] C^2\|\Xbf_1 -x \|_{\R^p}^2 ]\\
    & \leq \E[S_1'^2  ] C^2 s_n^{-2w},
\end{align*}
where the last step followed because $\E[S_1'^2 |\Xbf_1]=0$, for $\|\Xbf_1 -x \|_{\R^p} > s_n^{-w}$ by definition of $S_1'=S_1\1\{ \text{diam}(\Lcal(\xbf, \Zcal_{s_n}))\leq s_n^{-w}\} $. Since $\E[S_1'^2  ]\leq \E[S_1'] \leq \E[S_1]=\O(s_n^{-1})$ from~\eqref{eq:expectS1}, we have 
\begin{align*}
      \Var(\E[S_1' |\Xbf_1]( \E[ \xi_1|\Xbf_1]- \mu(\xbf)))=\O(s_n^{-(1+2w)}).
\end{align*}
Finally, we infer
\begin{align*}
    &\left| \Cov\left(\E[S_1' |\Xbf_1] \mu(\xbf), \E[S_1' |\Xbf_1]( \E[ \xi_1|\Xbf_1] - \mu(\xbf) ) \right)\right| \\
    \leq& \sqrt{\Var(\E[S_1' |\Xbf_1] \mu(\xbf)     )} \sqrt{\Var(\E[S_1' |\Xbf_1]( \E[ \xi_1|\Xbf_1]- \mu(\xbf) )) }\\
    =& \O(s_n^{-(1+w)}), 
\end{align*}
due to
\begin{align}\label{whatweneed3}
    \Var(\E[S_1' |\Xbf_1] \mu(\xbf)     )&= \Var(\E[S_1' |\Xbf_1]) \cdot \O(1)= \O(s_n^{-1}),
\end{align}
again using~\eqref{eq:varDecompSprime2}.
Thus, our Claim~\eqref{variance1term_1} holds.
\end{claimproof}

Before we continue proving the theorem, we note that, due to honesty, we have
\begin{align}\label{honestchange}
 \sum_{i=2}^{s_n} \E[S_i' \xi_i\mid \Xbf_1] &= \sum_{i=2}^{s_n} \E[ \E[S_i' \xi_i \mid  \Xbf_i, \Xbf_1]  \mid \Xbf_1] \nonumber \\
 &=\sum_{i=2}^{s_n} \E[ \E[S_i' \mid  \Xbf_i, \Xbf_1] \E[ \xi_i \mid  \Xbf_i, \Xbf_1]  \mid \Xbf_1] \nonumber \\
  &=\sum_{i=2}^{s_n} \E[ \E[S_i' \E[ \xi_i \mid  \Xbf_i]   \mid  \Xbf_i, \Xbf_1] \mid \Xbf_1]\nonumber \\
  &=\sum_{i=2}^{s_n} \E[ S_i' \E[  \xi_i \mid  \Xbf_i]  \mid \Xbf_1].
\end{align}

Now, we consider the variance of the sum in~\eqref{honestchange}:

\begin{claim}
\begin{align}\label{variance2term_1}
    \Var\left( \sum_{i=2}^{s_n} \E[S_i' \xi_i\mid \Xbf_1] \right)&= \Var(\E[S_1'|\Xbf_1]) \|\E[\xi_1|\Xbf=\xbf]\|_{\H}^2 + \mathcal{O}(s_n^{-(1+w)})
\end{align}
and
\begin{align}\label{variance2term_2}
        \Var\left( \sum_{i=2}^{s_n} \E[S_i' \langle\xi_i,f \rangle\mid \Xbf_1] \right) &= \Var(\E[S_1'|\Xbf_1]) \E[\langle \xi_1, f \rangle|\Xbf=\xbf]^2 + \mathcal{O}(s_n^{-(1+w)}).
\end{align}

\end{claim}

\begin{claimproof}

First we note that, using the definition of $S_i'$, it holds that  
\begin{align*}
\sum_{i=1}^{s_n} \E[S_i'\mid \Xbf_1]&=\E[\sum_{i=1}^{s_n} S_i' \mid \Xbf_1] \nonumber \\
&=\P\left(\text{diam}(\Lcal(\xbf, \Zcal_{s_n}))\leq s_n^{-w} \mid \Xbf_1\right)\\
&=\P\left(\text{diam}(\Lcal(\xbf, \Zcal_{s_n}))\leq s_n^{-w}\right),
\end{align*}
where the last step follows from \forestass{1} and the fact that $1 \in \I$. Thus abbreviating $p_n=\P\left(\text{diam}(\Lcal(\xbf, \Zcal_{s_n}))\leq s_n^{-w}\right)$, it follows that
\begin{align}\label{Sdproperty}
   \sum_{i=2}^{s_n} \E[S_i'\mid \Xbf_1]=p_n-\E[S_1' \mid \Xbf_1]] 
\end{align}

We only show~\eqref{variance2term_1}, because~\eqref{variance2term_2} follows analogously. By~\eqref{Sdproperty} and since $p_n$ is a constant,
\begin{align*}
    \Var(\E[S_1'|\Xbf_1]) \|\mu(\xbf)\|_{\H}^2&= \Var( (p_n-\E[S_1'|\Xbf_1]) \mu(\xbf) )\\
    &=\Var\left(\mu(\xbf) \sum_{i=2}^{s_n} \E[S_i' \mid \Xbf_1]  \right).
\end{align*}
Thus, we need to show that
\begin{align}
     \Var\left( \sum_{i=2}^{s_n} \E[S_i' \xi_i\mid \Xbf_1] \right)=\Var\left(\mu(\xbf) \sum_{i=2}^{s_n} \E[S_i' \mid \Xbf_1]  \right)  + \mathcal{O}(s_n^{-(1+w)}),
\end{align}
which according to Lemma~\ref{vardifboundlemma} is implied by
\begin{align}\label{whatweneed}
    \Var\left( \sum_{i=2}^{s_n} \E[S_i' \xi_i\mid \Xbf_1] -  \mu(\xbf) \sum_{i=2}^{s_n} \E[S_i' \mid \Xbf_1]  \right) =\mathcal{O}(s_n^{-(1+2w)}),
\end{align}
\begin{align}\label{whatweneed2}
    \Var\left( \sum_{i=2}^{s_n} \E[S_i' \xi_i\mid \Xbf_1]  \right) =\mathcal{O}(s_n^{-1}),
\end{align}
and~\eqref{whatweneed3}. Subsequently, we establish~\eqref{whatweneed} and~\eqref{whatweneed2}. 
Now, with~\eqref{honestchange}, we have 
\begin{align}\label{whatweneed22}
      \Var\left( \sum_{i=2}^{s_n} \E[S_i' \xi_i\mid \Xbf_1] -  \mu(\xbf) \sum_{i=2}^{s_n} \E[S_i' \mid \Xbf_1]  \right)&=\Var\left( \sum_{i=2}^{s_n} \E[ S_i' \E[  \xi_i \mid  \Xbf_i]  \mid \Xbf_1] -   \E[S_i' \mu(\xbf) \mid \Xbf_1]  \right) \nonumber \\
      &=\Var\left( \sum_{i=2}^{s_n} \E[ S_i'( \E[  \xi_i \mid  \Xbf_i] - \mu(\xbf) )   \mid \Xbf_1]  \right)\nonumber \\
      &=\Var\left( \sum_{i=2}^{s_n} \E[ S_i'\Delta(\Xbf_i)  \mid \Xbf_1]  \right),
\end{align}
with $\Delta(\Xbf_i)=\E[  \xi_i \mid  \Xbf_i] - \mu(\xbf) $. Next, we note that for each $i$, we have
\begin{align*}
    \E[ S_i'\Delta(\Xbf_i)  \mid \Xbf_1] &=  \E[ S_i'\Delta(\Xbf_i) \1\{\Xbf_1 \in \Lcal(\xbf,\Zcal_{s_n}) \}  \mid \Xbf_1] +   \E[ S_i'\Delta(\Xbf_i)  \1\{\Xbf_1 \notin \Lcal(\xbf,\Zcal_{s_n}) \}  \mid \Xbf_1].
\end{align*}
With $N_j= j + \sum_{i=2}^{s_n} \1\{\Xbf_i \in \Lcal(\xbf,\Zcal_{s_n}) \}$, $j \in \{0,1\}$, we have
\begin{align*}
    &\E[ S_i'\Delta(\Xbf_i) \1\{\Xbf_1 \in \Lcal(\xbf,\Zcal_{s_n}) \}  \mid \Xbf_1]\\ 
    =& \E\left.\left[ \frac{\1\{\Xbf_i \in \Lcal(\xbf,\Zcal_{s_n}) \}\1_{w,s_n}}{N_1}\Delta(\Xbf_i) \1\{\Xbf_1 \in \Lcal(\xbf,\Zcal_{s_n}) \}  \,\right\vert\, \Xbf_1 \right]\\
    =&\E\left.\left[ \frac{\1\{\Xbf_i \in \Lcal(\xbf,\Zcal_{s_n}) \} \1_{w,s_n}}{N_1}\Delta(\Xbf_i)\, \right\vert\, \Xbf_1, \{\Xbf_1 \in \Lcal(\xbf,\Zcal_{s_n}) \}  \right] \Prob(\Xbf_1 \in \Lcal(\xbf,\Zcal_{s_n}) \mid \Xbf_1)\\
      =&\E\left[ \frac{\1\{\Xbf_i \in \Lcal(\xbf,\Zcal_{s_n}) \} \1_{w,s_n}}{N_1}\Delta(\Xbf_i)  \right] \Prob(\Xbf_1 \in \Lcal(\xbf,\Zcal_{s_n}) \mid \Xbf_1),
\end{align*}
where the last step follows due to independence of $\Lcal(\xbf,\Zcal_{s_n})$ and $\Xbf_1$ by~\forestass{1}. Define the element
\begin{align*}
   E_i^1= \E\left[ \frac{\1\{\Xbf_i \in \Lcal(\xbf,\Zcal_{s_n}) \} \1_{w,s_n}}{N_1}\Delta(\Xbf_i) \right].
\end{align*}
Because this is nonrandom element of $\H$ and $\Prob(\Xbf_1 \in \Lcal(\xbf,\Zcal_{s_n}) \mid \Xbf_1)$ does not depend on the index $i\in\I$,  
it follows that 
\begin{align}\label{100}
&\Var\left( \sum_{i=2}^{s_n} \E[ S_i'\Delta(\Xbf_i)   \1\{\Xbf_1 \in \Lcal(\xbf,\Zcal_{s_n}) \}  \mid \Xbf_1]  \right)\nonumber\\
=&\left \|\sum_{i=2}^{s_n} E_i^1 \right \|^2_{\H} \Var(\Prob(\Xbf_1 \in \Lcal(\xbf,\Zcal_{s_n}) \mid \Xbf_1)) \nonumber \\
\leq&  \left( \sum_{i=2}^{s_n}\| E_i^1 \|_{\H} \right )^2  \E[\E[\1\{\Xbf_1 \in \Lcal(\xbf,\Zcal_{s_n})\} \mid \Xbf_1]^2]
\end{align}
Due to Jensen's inequality, 
\begin{align}\label{100d1}
    \E[\E[\1\{\Xbf_1 \in \Lcal(\xbf,\Zcal_{s_n})\} \mid \Xbf_1]^2] &\leq \E[\E[\1\{\Xbf_1 \in \Lcal(\xbf,\Zcal_{s_n})\} \mid \Xbf_1]] \nonumber\\
    &=\Prob(\Xbf_1 \in \Lcal(\xbf,\Zcal_{s_n}) ) \nonumber \\
    &=\mathcal{O}(s_n^{-1}),
\end{align}
where the last step followed because $2\kappa-1 \geq \E[N_{\xbf}]= \sum_{i=1}^{s_n} \E[\1\{\Xbf_i \in \Lcal(\xbf,\Zcal_{s_n})\}]=s_n \Prob(\Xbf_1 \in \Lcal(\xbf,\Zcal_{s_n}) )$ by \forestass{4}.
On the other hand, we have 
\begin{align}\label{100d2}
  \sum_{i=2}^{s_n}\| E_i^1 \|_{\H}  
  & \leq  \sum_{i=2}^{s_n} \E\left[ \frac{\1\{\Xbf_i \in \Lcal(\xbf,\Zcal_{s_n}) \} \1_{w,s_n}}{N_1} \left \| \Delta(\Xbf_i)\right\|_{\H} \right] \nonumber\\
  &\leq \sum_{i=2}^{s_n} \E\left[ \frac{\1\{\Xbf_i \in \Lcal(\xbf,\Zcal_{s_n}) \} \1_{w,s_n}}{N_1} C \left \| \Xbf_i -  \xbf\right\|_{\R^p} \right] \nonumber\\
  &\leq C s_n^{-w} \E\left[\sum_{i=2}^{s_n} \frac{\1\{\Xbf_i \in \Lcal(\xbf,\Zcal_{s_n}) \} \1_{w,s_n}}{N_1}  \right]  \nonumber\\
  &\leq  C s_n^{-w}
\end{align}
as $0 \leq \sum_{i=2}^{s_n} \1\{\Xbf_i \in \Lcal(\xbf,\Zcal_{s_n}) \}/N_1\leq 1$. Combining Equations~\eqref{100d1} and~\eqref{100d2} with~\eqref{100} gives
\begin{align}\label{100_end}
    \Var\left( \sum_{i=2}^{s_n} \E[ S_i'\Delta(\Xbf_i)   \1\{\Xbf_1 \in \Lcal(\xbf,\Zcal_{s_n}) \}  \mid \Xbf_1]  \right)&=\mathcal{O}(s_n^{-(1+2w)}).
\end{align}
Similarly, we have 
\begin{align}\label{101}
\Var\left( \sum_{i=2}^{s_n} \E[ S_i'\Delta(\Xbf_i)   \1\{\Xbf_1 \notin \Lcal(\xbf,\Zcal_{s_n}) \}  \mid \Xbf_1]  \right)&=\left \|\sum_{i=2}^{s_n} E_i^0 \right \|^2_{\H} \Var(\Prob(\Xbf_1 \notin \Lcal(\xbf,\Zcal_{s_n}) \mid \Xbf_1))
\end{align}
with 
\begin{align*}
   E_i^0= \E\left[ \frac{\1\{\Xbf_i \in \Lcal(\xbf,\Zcal_{s_n}) \} \1_{w,s_n}}{N_0}\Delta(\Xbf_i) \right] \in \H.
\end{align*}
With the same arguments as before, it follows that
\begin{align*}
    \left \|\sum_{i=2}^{s_n} E_i^0 \right\|_{\H}^2=\mathcal{O}(s_n^{-2w}).
\end{align*}
Combining this with
\begin{align*}
    \Var(\Prob(\Xbf_1 \notin \Lcal(\xbf,\Zcal_{s_n}) \mid \Xbf_1)) &= \Var(1-\Prob(\Xbf_1 \in \Lcal(\xbf,\Zcal_{s_n}) \mid \Xbf_1))\\
    &=\Var(\Prob(\Xbf_1 \in \Lcal(\xbf,\Zcal_{s_n}) \mid \Xbf_1))\\
    & \leq \E[ \E[ \1\{ \Xbf_1 \in \Lcal(\xbf,\Zcal_{s_n})\} \mid \Xbf_1  ]^2]\\
    & \leq \E[\1\{ \Xbf_1 \in \Lcal(\xbf,\Zcal_{s_n})\}]\\
    &=\mathcal{O}(s_n^{-1})
\end{align*}
results in
\begin{align}\label{101_end}
     \Var\left( \sum_{i=2}^{s_n} \E[ S_i'\Delta(\Xbf_i)   \1\{\Xbf_1 \notin \Lcal(\xbf,\Zcal_{s_n}) \}  \mid \Xbf_1]  \right)&=\mathcal{O}(s_n^{-(1+2w)}).
\end{align}
Consequently, we have 
\begin{align*}
  &\left |  \Cov\left( \sum_{i=2}^{s_n} \E[ S_i'\Delta(\Xbf_i)   \1\{\Xbf_1 \in \Lcal(\xbf,\Zcal_{s_n}) \}  \mid \Xbf_1] , \sum_{i=2}^{s_n} \E[ S_i'\Delta(\Xbf_i)   \1\{\Xbf_1 \notin \Lcal(\xbf,\Zcal_{s_n}) \}  \mid \Xbf_1]\right)  \right |\\
  \leq& \left( \Var\left(\sum_{i=2}^{s_n} \E[ S_i'\Delta(\Xbf_i)   \1\{\Xbf_1 \in \Lcal(\xbf,\Zcal_{s_n}) \}  \mid \Xbf_1]\right) \Var\left(\sum_{i=2}^{s_n} \E[ S_i'\Delta(\Xbf_i)   \1\{\Xbf_1 \notin \Lcal(\xbf,\Zcal_{s_n}) \}  \mid \Xbf_1]\right) \right)^{1/2}\\
  =&\mathcal{O}(s_n^{-(1+2w)}),
\end{align*}
so that~\eqref{whatweneed} holds. Finally, using the reverse triangle inequality as in~\eqref{l2normtrick} in the proof of Lemma~\ref{vardifboundlemma}, we obtain
\begin{align*}
     \Var\left( \sum_{i=2}^{s_n} \E[S_i' \xi_i\mid \Xbf_1]  \right)^{1/2} = \Var(\E[S_1' |\Xbf_1] \mu(\xbf)     )^{1/2} +\O(s_{n}^{-(1/2+w)}) = \O(s_{n}^{-1/2}),
\end{align*}
by~\eqref{whatweneed} and~\eqref{whatweneed3}. This shows~\eqref{whatweneed2} and thus~\eqref{variance2term_1} in the claim holds true.
\end{claimproof}

Finally, we consider the covariance between $\E[S_1' \xi_i\mid \Xbf_1]$ and $\sum_{i=2}^{s_n} \E[S_i' \xi_i\mid \Xbf_1]$.

\begin{claim}
For some $\varepsilon > 0$, we have
\begin{align}\label{variance3term_1}
        \Cov\left( \sum_{i=2}^{s_n} \E[S_i' \xi_i\mid \Xbf_1], \E[S_1'|\Xbf_1]  \E[ \xi_1|\Xbf_1] \right)&= -\Var(\E[S_1'|\Xbf_1]) \|\E[\xi_1|\Xbf=\xbf]\|_{\H}^2 + \mathcal{O}(s_n^{-(1+\varepsilon)})
\end{align}
and
\begin{align}\label{variance3term_2}
    \Cov\left( \sum_{i=2}^{s_n} \E[S_i' \langle\xi_i,f \rangle\mid \Xbf_1], \E[S_1'|\Xbf_1]  \E[ \langle \xi_1, f \rangle|\Xbf_1] \right) &= -\Var(\E[S_1'|\Xbf_1]) \E[\langle \xi_1, f \rangle|\Xbf=\xbf]^2 + \mathcal{O}(s_n^{-(1+\varepsilon)}).
\end{align}
\end{claim}

\begin{claimproof}
Again, we only show~\eqref{variance3term_1}, because~\eqref{variance3term_2} follows analogously. Using~\eqref{honestchange}, we can subtract and add $\mu(\xbf)\sum_{i=2}^{s_n} \E[S_i'|\Xbf_1]$ and $\E[S_1'|\Xbf_1] \mu(\xbf)$ to obtain
\begin{align*}
     &\Cov\left( \sum_{i=2}^{s_n} \E[S_i' \xi_i\mid \Xbf_1], \E[S_1'|\Xbf_1]  \E[ \xi_1|\Xbf_1] \right)\\
     =&  \Cov\left( \sum_{i=2}^{s_n} \E[ S_i' \E[  \xi_i \mid  \Xbf_i]  \mid \Xbf_1], \E[S_1'|\Xbf_1]  \E[ \xi_1|\Xbf_1] \right)\\
    =&  \Cov\left( \sum_{i=2}^{s_n} \E[ S_i' \Delta(\Xbf_i)  \mid \Xbf_1] + \mu(\xbf) \sum_{i=2}^{s_n} \E[S_i' \mid \Xbf_1 ], \E[S_1'  \Delta(\Xbf_1)|\Xbf_1]  + \mu(\xbf) \E[S_1'|\Xbf_1]  \right)\\
    =&\Cov\left( \sum_{i=2}^{s_n} \E[ S_i' \Delta(\Xbf_i)  \mid \Xbf_1] , \E[S_1' \Delta(\Xbf_1)|\Xbf_1]    \right) +\Cov\left( \sum_{i=2}^{s_n} \E[ S_i' \Delta(\Xbf_i)  \mid \Xbf_1] , \mu(\xbf) \E[S_1'|\Xbf_1] \right)\\
    &\quad+\Cov\left( \mu(\xbf) \sum_{i=2}^{s_n} \E[S_i' \mid \Xbf_1 ] , \E[S_1' \Delta(\Xbf_1)|\Xbf_1]    \right) +  \Cov\left( \mu(\xbf) \sum_{i=2}^{s_n} \E[S_i' \mid \Xbf_1 ], \mu(\xbf) \E[S_1'|\Xbf_1]  \right)\\
    =:& (I) + (II) + (III) + (IV),
\end{align*}
where again $\Delta(\Xbf_i)=\E[  \xi_i \mid  \Xbf_i]  - \mu(\xbf)$. Since from~\eqref{Sdproperty},
\begin{align}
    \mu(\xbf)\sum_{i=2}^{s_n} \E[S_i'|\Xbf_1] = \mu(\xbf)(p_n-\E[S_1'|\Xbf_1]),
\end{align}
it holds that
\begin{align*}
    (IV)=\Cov\left( \mu(\xbf) (p_n-\E[S_1'|\Xbf_1]), \mu(\xbf) \E[S_1'|\Xbf_1]  \right)=-\Var(\E[S_1'|\Xbf_1]) \| \mu(\xbf) \|_{\H}^2.
\end{align*}
Subsequently, we show that the remaining terms are negligible.
Due to the Cauchy--Schwarz inequality, we have
\begin{align*}
    |(I)| \leq \left( \Var\left(\sum_{i=2}^{s_n} \E[ S_i' \Delta(\Xbf_i)  \mid \Xbf_1]\right) \Var(\E[S_1' \Delta(\Xbf_1)|\Xbf_1]) \right)^{1/2}.
\end{align*}
As proven above (combining~\eqref{whatweneed} and~\eqref{whatweneed22}), $\Var(\sum_{i=2}^{s_n} \E[ S_i' \Delta(\Xbf_i)  \mid \Xbf_1])=\mathcal{O}(s_n^{-(1+2w)})$, and it can be established that
\[
\Var(\E[S_1' \Delta(\Xbf_1)|\Xbf_1]) \leq \E[  \E[S_1' \|\Delta(\Xbf_1)\|_{\H}|\Xbf_1]^2]=\mathcal{O}(s_n^{-(1+2w)})
\]
holds. Consequently, $(I)=\mathcal{O}(s_n^{-(1+2w)})$. Similarly,
\begin{align*}
    |(II)| \leq \left( \Var\left(\sum_{i=2}^{s_n} \E[ S_i' \Delta(\Xbf_i)  \mid \Xbf_1] \right) \Var(\mu(\xbf) \E[S_1'|\Xbf_1]) \right)^{1/2}=\mathcal{O}(s_n^{-(1+w)}),
\end{align*}
as $\Var(\E[S_1'|\Xbf_1])\leq \E[(S_1')^2]=\mathcal{O}(s_n^{-1}).$
Finally, 
\begin{align*}
    |(III)|&=|\Cov\left( \mu(\xbf) (1-\E[S_1' \Delta(\Xbf_1)|\Xbf_1] ) , \E[S_1' \Delta(\Xbf_1)|\Xbf_1]    \right)| \\
    &=|- \|\mu(\xbf)\|_{\H}^2 \Var(\E[S_1' \Delta(\Xbf_1)|\Xbf_1])|\\
    &=\mathcal{O}(s_n^{-(1+2w)})
\end{align*}
as above.

\end{claimproof}

Combining~\eqref{variance1term_1},~\eqref{variance2term_1}, and~\eqref{variance3term_1}, we obtain
\begin{align*}
        \Var( \E[T'(\Zcal_{s_n})| \Xbf_1 ])&= 2\Var(\E[S_1'|\Xbf_1]) \|\E[\xi_1|\Xbf=\xbf]\|_{\H}^2 - 2\Var(\E[S_1'|\Xbf_1]) \|\E[\xi_1|\Xbf=\xbf]\|_{\H}^2  + \mathcal{O}(s_n^{-(1+\epsilon)})\\
        &=\mathcal{O}(s_n^{-(1+\epsilon)})
\end{align*}
and analogously
\begin{align*}
       \Var( \E[\langle T'(\Zcal_{s_n}), f \rangle \mid \Zbf_1 ])&= \mathcal{O}(s_n^{-(1+\epsilon)}),
\end{align*}
proving~\eqref{Varboundnew}.

\end{claimproof}

We recall that $\Var(\E[S_1'|\Xbf_1]) \sim \Var(\E[S_1|\Xbf_1])=\Var(\E[S_1|\Zbf_1])=\Omega((s_n\log(s_n))^{-1})$, by~\eqref{truncation_S},~\eqref{Honestyconsequence}, and Lemma~\ref{lemma4} respectively. This together with Claim~\eqref{Varboundnew} and the expansion in~\eqref{varexpansion} establishes~\eqref{varianceassumption}.
\end{proof}

This leads us to the proof of Theorem~\ref{thm: asymptoticnormality} in the main text.

\asymptoticnormality*

\begin{proof}

First, by the definition of $\sigma_n$, we have
\[
\xi_n^0 := \sum_{i=1}^n   \frac{s_n}{n \sigma_n} T_{n}(\Zbf_i) =   \sum_{i=1}^{n} \frac{T_{n}(\Zbf_i)}{\sqrt{n \Var(T_n(\Zbf_1))}}.
\]

Define $\sigma_{n}^2(f)= \frac{s_n^2}{n} \Var( \langle T_n(\Zbf_1), f \rangle)$.
Subsequently, we establish univariate convergence for all $f \in \H$:

\begin{claim}
For all $f \in \H$, we have $\sum_{i=1}^n   \frac{s_n}{n \sigma_n(f)} \langle T_{n}(\Zbf_i),f \rangle \stackrel{D}{\to} N(0, \sigma(f)^2)$.
\end{claim}

\begin{claimproof}

Due to linearity, $\langle T_n(\Zbf_1), f \rangle$ is the first order approximation of a tree using the univariate response $f(\Ybf_i)$. Thus, it follows from Assumption~\forestass{1}--\forestass{5} and \dataass{1}--\dataass{7} with the implications~\eqref{eq:Lipschitzcontinuityf}--\eqref{supbound} and the arguments in the proof of Theorem 8 in~\citet{wager2017estimation} that
\begin{align}
     \sum_{i=1}^n   \frac{s_n}{n \sigma_n(f)} \langle T_{n}(\Zbf_i),f \rangle \stackrel{D}{\to} N(0, 1).
\end{align}
From Theorem~\ref{amazingvarianceproposition}, we have
\begin{align*}
   \frac{\sigma_n(f)}{\sigma_n} = \frac{\Var(\langle T_{n}(\Zbf_1), f \rangle)}{\Var(T_{n}(\Zbf_1))} \to \sigma^2(f)  > 0,
\end{align*}
so that due to Slutsky's theorem,
\begin{align}
    \sum_{i=1}^n   \frac{s_n}{n \sigma_n} \langle T_{n}(\Zbf_i),f \rangle = \frac{\sigma_n(f)}{\sigma_n}  \sum_{i=1}^n   \frac{s_n}{n \sigma_n(f)} \langle T_{n}(\Zbf_i),f \rangle \to N(0, \sigma^2(f))
\end{align}
with $\sigma^2(f) > 0$.
\end{claimproof}

Now, we proof uniform tightness:

\begin{claim}
$\left( \xi_n^0   \right)_{n \in \N}$ is uniformly tight.
\end{claim}

\begin{claimproof}

Because $\H$ is separable due to our assumptions on the kernel, there exists a complete orthogonal basis $\left( e_j\right)_{j \in \N}$ of $\H$; see for instance~\citet{hilbertspacebook}. Let $P_k$ be the projection operator onto the linear span of the first $k$ elements of $ \left( e_j\right)_{j \in \N} $, $S_k=\mbox{span}(e_1, \ldots, e_k)$. Because $S_k$ is closed and linear, $P_k$ is well defined. Moreover, for all $f \in \H$, we have $\langle f - P_k(f), P_k(f) \rangle=0$. Furthermore, it can be shown that $P_k(f)=\sum_{j=1}^k \langle f,e_j \rangle e_j$.

We now verify condition (c) of~\citet[Lemma 3.2]{HilbertspaceCLTs}, which is a sufficient condition for tightness:

\begin{claim}
 $\limsup_n \E[\| \xi_n^0 - P_k(\xi_n^0) \|_{\H}^2] \to 0$, as $k \to \infty$.
\end{claim}

\begin{claimproof}
For any $n,k$, we have
\begin{align*}
    \E[\| \xi_n^0 - P_k(\xi_n^0) \|_{\H}^2]=\E[\| \xi_n^0 \|_{\H}^2] +\E[\| P_k(\xi_n^0) \|_{\H}^2] - 2 \E[ \langle \xi_n^0, P_k(\xi_n^0)  \rangle].
\end{align*}
Furthermore, for all $n$, we have
\begin{align*}
    \E[ \| \xi_n^0 \|_{\H}^2]=\Var( \xi_n^0) = \Var\left(  \sum_{i=1}^{n} \frac{T_{n}(\Zbf_i)}{\sqrt{n \Var(T_n(\Zbf_1))}} \right)=\frac{n}{n \Var(T_n(\Zbf_1))} \Var(T_n(\Zbf_1))=1.
\end{align*}
Because $P_k(\xi_n^0)$ is an orthogonal projection, we have
\begin{align*}
    \E[ \langle \xi_n^0, P_k(\xi_n^0)  \rangle] = \E[ \| P_k(\xi_n^0)  \|_{\H}^2].
\end{align*}
Thus,
\begin{align*}
    \E[\| \xi_n^0 - P_k(\xi_n^0) \|_{\H}^2] = 1- \E[ \| P_k(\xi_n^0)  \|_{\H}^2].
\end{align*}
Now for any $k$, we have
\begin{align*}
    \E[ \| P_k(\xi_n^0)  \|_{\H}^2] &= \sum_{j=1}^k \E[\langle \xi_n^0, e_j \rangle^2]\\
    &= \frac{1}{\Var(T_n(\Zbf_1))} \sum_{j=1}^k \E[\langle T_{n}(\Zbf_1), e_j \rangle^2]\\
    &= \sum_{j=1}^k \frac{\Var(\langle T_{n}(\Zbf_1), e_j \rangle)}{\Var(T_n(\Zbf_1))}\\
    & \to  \sum_{j=1}^k  \frac{\Var(\langle k(\Ybf, \cdot) ,e_j  \rangle|\Xbf=\xbf)}{\Var(k(\Ybf, \cdot)|\Xbf=\xbf)},
\end{align*}
as $n \to \infty$ due to~\eqref{varianceassumption} and the fact that the sum over $k$ is finite. 
Additionally, due to~\citet[Chapter 7]{hilbertspacebook}, we have
\begin{align*}
    \Var(k(\Ybf, \cdot)|\Xbf=\xbf) = \sum_{j=1}^{\infty} \Var(\langle k(\Ybf, \cdot) ,e_j  \rangle|\Xbf=\xbf). 
\end{align*}
This means that 
\begin{align*}
    \limsup_n \E[\| \xi_n^0 - P_k(\xi_n^0) \|_{\H}^2] &= 1 - \liminf_{n}  \E[ \| P_k(\xi_n^0)  \|_{\H}^2]\\
    &=1- \sum_{j=1}^k  \frac{\Var(\langle k(\Ybf, \cdot) ,e_j  \rangle|\Xbf=\xbf)}{\Var(k(\Ybf, \cdot)|\Xbf=\xbf)}\\
    & \to 0
\end{align*}
as $k \to \infty$.
\end{claimproof}

Consequently, $\left( \xi_n^0   \right)_{n \in \N}$ is uniformly tight.

\end{claimproof}

Univariate convergence together with tightness imply $\xi_n^0 \stackrel{D}{\to} N(0, \boldsymbol{\Sigma}_{\xbf})$; see for example~\citet[Lemma 3.1/3.2]{HilbertspaceCLTs} or~\citet[Chapter 7]{hilbertspacebook}. Since by Theorem~\ref{thm: asymptoticlinearity} we have 
\[
\frac{1}{\sigma_n} (\hmun - \mu(\xbf)) = \xi_n^0 + o_{p}(1),
\]
the result follows.
\end{proof}

Before being able to prove Theorem~\ref{thm: asymptoticnormalityhalfsampling}, we need a few preliminary results:

Let in the following $\H^*$ be the dual space of $\H$, that is, $$\H^*=\{F\colon\H \to \R \text{ linear, bounded, and continuous}\}.$$ 
Moreover, let 
\begin{align}\label{Fdef}
    \F=\{F \in \H^*, \|F \|_{\mathcal{H}^*} \leq 1 \},
\end{align}
where $\|\cdot \|_{\mathcal{H}^*}$ is the operator norm on $\H^*$. Additionally, let $\ell^{\infty}(\F)$ be the space of all bounded real-valued functions $\F \to \R$.

Due to the Riesz representation theorem,  for each $F \in \H^*$ there exists exactly one $f_F \in \H$ such that $F(h)=\langle f_F, h \rangle$ for all $h \in \H$. Let us define the map $D\colon \H \to \ell^{\infty}(\F)$ by 
\begin{align}\label{Ddefinition}
    D(f)(F)=F(f) \text{ for } F \in \F.
\end{align}
Following the notation of empirical process theory, for $F \in \F$, we let 
\[
\P_{k,\xbf}F = D(\mu(\xbf))(F) = F(\mu(\xbf)) = \E[ F(k(\Ybf, \cdot)) \mid \Xbf=\xbf ]=\E[f_F(\Ybf) \mid \Xbf=\xbf].
\]
Thus, $\P_{k,\xbf}$ is the 
process associated with $k(\Ybf, \cdot) \mid \Xbf=\xbf$ on $\H$. Similarly, let us for $F \in \F$ denote by $\hat{\P}_{k,\xbf} - \P_{k,\xbf}$ the function defined by
\begin{align*}
    (\hat{\P}_{k,\xbf} - \P_{k,\xbf}) F = \langle \hmun-\mu(\xbf), f_F \rangle.
\end{align*}
Moreover, define the Gaussian process $G_{\P_{k,\xbf}}$ on $\ell^{\infty}(\F)$ by 
\[
G_{\P_{k,\xbf}}(F)=D(\xi)(F)=F(\xi)=\langle \xi, f_F \rangle,
\]
where $\xi \sim N(0, \boldsymbol{\Sigma}_{\xbf})$ on $\H$, with $\boldsymbol{\Sigma}_{\xbf}$ as in Theorem~\ref{thm: asymptoticnormality}.

\citet{B3} show that $D$ is linear and continuous and that it has a continuous inverse. With this, it follows that:

\begin{corollarytwo}
For all $n$, $ \frac{1}{\sigma_n}(\hat{\P}_{k,\xbf} - \P_{k,\xbf}) \in \ell^{\infty}(\F)$ and 
\begin{align*}
    \frac{1}{\sigma_n}(\hat{\P}_{k,\xbf} - \P_{k,\xbf}) \stackrel{D}{\to} G_{\P_{k,\xbf}}
\end{align*}
in $\ell^{\infty}(\F)$.
\end{corollarytwo}

\begin{proof}
\citet{B3} show that $D$ in~\eqref{Ddefinition} is a continuous bounded linear operator satisfying
\begin{align*}
    \frac{1}{\sigma_n}  D(\hmun - \mu(\xbf))  = D \left( \frac{1}{\sigma_n} \left( \hmun - \mu(\xbf) \right)  \right) \stackrel{D}{\to} D(\xi)=G_{\P_{k,\xbf}}
\end{align*} 
due to the continuous mapping theorem. Additionally, by the Riesz representation theorem,
\begin{align*}
    D(\hmun - \mu(\xbf))(F)=F(\hmun - \mu(\xbf) ) = \langle f_F, \hmun - \mu(\xbf) \rangle  =(\hat{\P}_{k,\xbf} - \P_{k,\xbf}) F
\end{align*}
for all $F \in \F$, so that $D(\hmun - \mu(\xbf))=\hat{\P}_{k,\xbf} - \P_{k,\xbf}$. 
\end{proof}

This result enables us to use empirical process techniques, 
as we will do in the proof of Theorem~\ref{thm: asymptoticnormalityhalfsampling}.
To prove Theorem~\ref{thm: asymptoticnormalityhalfsampling}, we start with the following important Lemma, in analogy to~\citet[Lemma 3]{B2}:

\begin{lemma}\label{KLemma}
 Let $Y_{ni}$, $i=1,\ldots,m_n, n \geq 1$ be a triangular array of mean zero independent (within rows)  random variables. Let $W_i$, $i=1,\ldots,n$ be i.i.d random variables, independent of $\left(Y_{ni} \right)_{n,i}$, and with $\E[W_i]=0$ and $\Var(W_i)=1$ for all $i$. Additionally, assume that we have 
 \begin{align}\label{eq: LyapconditionI}
     \sum_{i=1}^{m_n}  \Var(Y_{ni}^2) \to  \sigma_0 > 0
 \end{align}
 and
  \begin{align}\label{eq: LyapconditionIextra}
     \sum_{i=1}^{m_n}  Y_{ni}^2 \stackrel{p}{\to}  \sigma_0 > 0
 \end{align}
 and moreover, for some $\delta > 0$,
\begin{align}\label{eq: LyapconditionII}
     \lim_{n \to \infty} \sum_{i=1}^{m_n} \frac{ \E[ \left | Y_{ni}  \right |^{2+\delta}]}{\left(\sum_{j=1}^{m_n} \Var(Y_{nj}) \right)^{1+\delta/2}}=0.
 \end{align}
 Then, for $\mathbb{Y}_n=\{Y_1, \ldots, Y_{m_n}\}$,
 \begin{align}\label{firstassertioneq}
      \Var \left(\left.\sum_{i=1}^{m_n} W_i Y_{ni}\,\right\vert\, \mathbb{Y}_n \right) = \Var \left(\left.\sum_{i=1}^{m_n} W_i Y_{ni}\,\right\vert\, \mathbb{Y}_n \right) \stackrel{p}{\to} \sigma_0
 \end{align}
 and
 \begin{align}\label{secondassertioneq}
     \sum_{i=1}^{m_n} \frac{ \E[ \left | W_i Y_{ni}  \right |^{2+\delta} \mid \mathbb{Y}_n]}{\left(\sum_{j=1}^{m_n} \Var(W_jY_{nj} \mid \mathbb{Y}_n) \right)^{1+\delta/2}} \stackrel{p}{\to} 0
 \end{align}
as $n \to \infty$.
\end{lemma}

\begin{proof}

First, by~\eqref{eq: LyapconditionIextra},
\begin{align*}
     \Var \left(\left.\sum_{i=1}^{m_n} W_i Y_{ni}\,\right\vert\, \mathbb{Y}_n \right) =\sum_{i=1}^{m_n} \Var \left( W_i \mid \mathbb{Y}_n \right)  Y_{ni}^2 = \sum_{i=1}^{m_n} Y_{ni}^2 \stackrel{p}{\to}  \sigma_0,
\end{align*}
which establishes~\eqref{firstassertioneq}. Similarly, we have 
 \begin{align*}
     \sum_{i=1}^{m_n} \frac{ \E[ \left | W_i Y_{ni}  \right |^{2+\delta} \mid \mathbb{Y}_n]}{\left(\sum_{j=1}^{m_n} \Var(W_jY_{nj} \mid \mathbb{Y}_n) \right)^{1+\delta/2}} 
     &=\E[  | W_1  |^{2+\delta} \mid \mathbb{Y}_n] \sum_{i=1}^{m_n} \frac{ | Y_{ni} |^{2+\delta}}{\left(\sum_{j=1}^{m_n} Y_{nj}^2 \right)^{1+\delta/2}}
 \end{align*}
and
\begin{align*}
    \sum_{i=1}^{m_n} \frac{ | Y_{ni} |^{2+\delta}}{\left(\sum_{j=1}^{m_n} Y_{nj}^2 \right)^{1+\delta/2}} 
    &=\sum_{i=1}^{m_n} \frac{ | Y_{ni} |^{2+\delta}}{\left(\sum_{j=1}^{m_n}  \Var(Y_{nj})\right)^{1+\delta/2}} \left(\frac{\sum_{j=1}^{m_n}  \Var(Y_{nj})}{\sum_{j=1}^{m_n} Y_{nj}^2}\right)^{1+\delta/2}.
\end{align*}
By Assumption~\eqref{eq: LyapconditionI} and~\eqref{eq: LyapconditionIextra}, we have
\[
\left(\frac{\sum_{i=1}^{m_n}  \Var(Y_{ni})}{\sum_{j=1}^{m_n} Y_{ni}^2}\right)^{1+\delta/2} \stackrel{p}{\to} 1,
\]
and, due to Markov's inequality and~\eqref{eq: LyapconditionII},
\begin{align*}
    \Prob \left(\sum_{i=1}^{m_n} \frac{ | Y_{ni} |^{2+\delta}}{\left(\sum_{j=1}^{m_n}  \Var(Y_{nj})\right)^{1+\delta/2}} > \varepsilon \right) \leq \frac{1}{\varepsilon} \sum_{i=1}^{m_n} \frac{ \E[ | Y_{ni} |^{2+\delta} ]}{\left(\sum_{j=1}^{m_n}  \Var(Y_{nj})\right)^{1+\delta/2}} \to 0,
\end{align*}
so that
\[
\sum_{i=1}^{m_n} \frac{ | Y_{ni} |^{2+\delta}}{\left(\sum_{j=1}^{m_n}  \Var(Y_{nj})\right)^{1+\delta/2}} 
\cdot
\frac{\left(\sum_{j=1}^{m_n}  \Var(Y_{nj})\right)^{1+\delta/2}}{\left(\sum_{j=1}^{m_n} Y_{nj}^2\right)^{1+\delta/2}} = \o_p(1),
\]
which establishes the result.
\end{proof}

\asymptoticnormalityhalfsampling*
\begin{proof}
For this proof, we recall the definition of $\xi_n$ in~\eqref{xindef}.
For each subsample of size $s_n$ of the data, we have a tree. For a given $\HS$ we consider all such trees that are built using data points from $\HS$. Thus, we consider the same ``base'' random forest built using all the data and select different trees depending on which subsample $\HS$ we consider. 
Since $s_n$ is of smaller order than $n$, $\P(|\HS| \leq s_n) \to 0$, as $n \to \infty$. Thus, by the same arguments as in~\citet[Theorem 5]{athey2019generalized} combined with Theorem~\ref{thm: asymptoticlinearity}, we obtain 
\begin{align*}
  \hmunHS -  \mu(\xbf)  &= \frac{s_n}{|\HS|} \sum_{i \in \HS} T_{n}(\Zbf_i) + o_{p}(\sigma_{n}) \nonumber =\frac{s_n}{n} \sum_{i \in \HS} \frac{n}{|\HS|} T_{n}(\Zbf_i) + o_{p}(\sigma_{n}).
\end{align*}
Due to $\sigma_{n} =  \sqrt{s_n^2/n \cdot \Var( T_n(\Zbf_1))}$, we infer 
\begin{align*} 
  \frac{1}{\sigma_{n}} \left( \hmunHS -  \mu(\xbf) \right) =& \frac{1}{\sqrt{n}} \sum_{i \in \HS}
  \frac{n}{|\HS|} \frac{T_{n}(\Zbf_i)}{\sqrt{\Var(T_n(\Zbf_1))}} + o_{p}(1)\\
  =& \frac{1}{\sqrt{n}} \sum_{i=1}^n
 \frac{n}{|\HS|} W_i\frac{T_{n}(\Zbf_i)}{\sqrt{\Var(T_n(\Zbf_1))}} + o_{p}(1),
\end{align*}
with $(W_i)_{i=1}^n$ independent and $W_i \sim \mathrm{Bernoulli(1/2)}$. Thus,

\begin{align}\label{Sasymptoticlinearity}
    \frac{1}{\sigma_{n}} \left( \hmunHS -  \hmun \right) &=  \frac{1}{\sigma_{n}} \left( \hmunHS -  \mu(\xbf) - (\hmun - \mu(\xbf)) \right) \nonumber \\
    &= \frac{1}{\sqrt{n}} \sum_{i=1}^n
 \left(\frac{n}{|\HS|}W_i -1 \right) \frac{T_{n}(\Zbf_i)}{\sqrt{\Var(T_n(\Zbf_1))}} + o_{p}(1).
\end{align}
Recall our abbreviation
\begin{align*}
   \xi_{n}^{\HS}=\frac{1}{\sigma_n}\left( \hmunHS -  \hmun \right)
\end{align*}
from~\eqref{xinHSdef}. Subsequently, we first prove the result for a simplified version of the sum in~\eqref{Sasymptoticlinearity} consisting of independent summands. Let in the following
\begin{align}
     \tilde{W}_i=2 W_i -1
\end{align}
and
\begin{align}\label{eq: xinWdef}
    \xi_n^{W}= \frac{1}{\sqrt{n}} \sum_{i=1}^n
 \tilde{W}_i\frac{T_{n}(\Zbf_i)}{\sqrt{\Var(T_n(\Zbf_1))}}.
\end{align}

\begin{claim}
It holds that
\begin{equation}\label{eq: condconvergence}
     \xi_n^{W} \xrightarrow[W]{D} N(0, \boldsymbol{\Sigma}_{\xbf}).
\end{equation}
\end{claim}

\begin{claimproof}

The proof combines arguments from~\citet{B2} with arguments made above and the equivalence of $ \hmunHS -  \hmun$ and a certain empirical process as in~\citet{B3}. Note that, since the $W_i$ are i.i.d, $W_i \sim \mathrm{Bernoulli(1/2)}$, $\E[ \tilde{W}_i]=0$ and $\Var( \tilde{W}_i)=1$. First, we prove unconditional convergence:

\begin{claim}
It holds that
\begin{equation}\label{eq: uncondconvergence}
    \xi_n^{W}\stackrel{D}{\to} N(0, \boldsymbol{\Sigma}_{\xbf}).
\end{equation}
\end{claim}

\begin{claimproof}

We start by verifying uniform tightness of the sequence $(\xi_n^{W})_n$:
\begin{claim}
 $\limsup_n \E[\| \xi_n^W - P_k(\xi_n^W) \|_{\H}^2] \to 0$ as $k \to \infty$.
\end{claim}

\begin{claimproof}
For all $n$, we have 
\begin{align*}
    \E[ \| \xi_n^{W} \|_{\H}^2]&= \Var\left(  \sum_{i=1}^{n}  \tilde{W}_i \frac{T_{n}(\Zbf_i)}{\sqrt{n \Var(T_n(\Zbf_1))}} \right)\\
     &=\E \left[\Var\left(\left.  \sum_{i=1}^{n}  \tilde{W}_i \frac{T_{n}(\Zbf_i)}{\sqrt{n \Var(T_n(\Zbf_1))}} \,\right\vert\, \Zcal_n \right) \right ] + \Var\left( \E \left[ \left.\sum_{i=1}^{n}  \tilde{W}_i \frac{T_{n}(\Zbf_i)}{\sqrt{n \Var(T_n(\Zbf_1))}}  \,\right \vert\,\Zcal_n \right] \right).
\end{align*}
For the first term in the above decomposition, we have
\begin{align*}
    \E \left[\Var\left( \left. \sum_{i=1}^{n}  \tilde{W}_i \frac{T_{n}(\Zbf_i)}{\sqrt{n \Var(T_n(\Zbf_1))}} \,\right \vert\, \Zcal_n \right) \right ] = \E \left[ \left \|\frac{T_{n}(\Zbf_i)}{ \Var(T_n(\Zbf_1))} \right \|^2\right ] = 1.
\end{align*}
And for the second term, we have  
\begin{align*}
    \Var\left(\left. \E \left[ \sum_{i=1}^{n}  \tilde{W}_i \frac{T_{n}(\Zbf_i)}{\sqrt{n \Var(T_n(\Zbf_1))}} \,\right \vert\, \Zcal_n \right] \right) =  \Var\left( \E \left[ \left.\sum_{i=1}^{n}  \tilde{W}_i \,\right \vert\, \Zcal_n \right] \frac{T_{n}(\Zbf_i)}{\sqrt{n \Var(T_n(\Zbf_1))}}  \right) = 0.
\end{align*}
Thus, $\E[ \| \xi_n^{W} \|_{\H}^2]=1=\Var(\xi_n)$.
Similarly, 
\begin{align*}
    \E[ \| P_k(\xi_n^W)  \|_{\H}^2] &= \sum_{j=1}^k \E[\langle \xi_n^W, e_j \rangle^2]\\
    &= \sum_{j=1}^k \frac{\Var(\tilde{W}_i\langle T_{n}(\Zbf_1), e_j \rangle)}{\Var(T_n(\Zbf_1))}\\
    &= \sum_{j=1}^k \frac{\Var(\langle T_{n}(\Zbf_1), e_j \rangle)}{\Var(T_n(\Zbf_1))}
\end{align*}
due to the same variance arguments, so that $\E[ \| P_k(\xi_n^W)  \|_{\H}^2]=\E[ \| P_k(\xi_n)  \|_{\H}^2]$ 
Thus, the claim follows by exactly the same argument as in the proof of Theorem~\ref{thm: asymptoticnormality}.
\end{claimproof}

We now verify marginal convergence:

\begin{claim}
For all $f \in \H$, we have $\langle \xi_n^W, f \rangle \stackrel{D}{\to} N(0, \sigma^2(f))$, where $\sigma(f)>0$ is defined in Theorem~\ref{amazingvarianceproposition}.
\end{claim}

\begin{claimproof}
We prove convergence using the Lyapunov central limit theorem similarly to~\citet[Theorem 3.4]{wager2018estimation}.
First, with the arguments in the proof of Theorem 8 in~\citet{wager2017estimation}, it can be shown that, under Assumption~\forestass{1}--\forestass{5} and \dataass{1}--\dataass{7} with the implications~\eqref{eq:Lipschitzcontinuityf}--\eqref{supbound}, that we have 
\begin{align*}
         \lim_{n \to \infty} \sum_{i=1}^{n} \frac{ \E[ \left | \langle T_{n}(\Zbf_i), f \rangle  \right |^{2+\delta}]}{\left(n \Var(\langle T_{n}(\Zbf_1), f \rangle ) \right)^{1+\delta/2}}=0.
\end{align*}
By Theorem~\ref{amazingvarianceproposition}, consequently also
\begin{align*}
    \lim_{n \to \infty} \sum_{i=1}^{n} \frac{ \E[ \left | \langle T_{n}(\Zbf_i), f \rangle  \right |^{2+\delta}]}{\left(n \Var( T_{n}(\Zbf_1)) \right)^{1+\delta/2}}=0,
\end{align*}
that is, the Lyapunov condition holds for $\langle \xi_n, f \rangle$. As $\Var( \tilde{W}_1 T_{n}(\Zbf_1))=1$ and $\E[|\tilde{W}_i|^{2+\delta}]=\E[|\tilde{W}_1|^{2+\delta}] \leq 1$, we have 
\begin{align*}
    \lim_{n \to \infty} \sum_{i=1}^{n} \frac{ \E\left[ \left | \tilde{W}_i \langle T_{n}(\Zbf_i), f \rangle  \right |^{2+\delta}\right]}{\left(n \Var(\tilde{W}_i T_{n}(\Zbf_1)) \right)^{1+\delta/2}} \leq \lim_{n \to \infty} \sum_{i=1}^{n} \frac{ \E[ \left | \langle T_{n}(\Zbf_i), f \rangle  \right |^{2+\delta}]}{\left(n \Var(T_{n}(\Zbf_1)) ) \right)^{1+\delta/2}}=0,
\end{align*}
so that the Lyapunov condition holds for $\langle \xi_n^W, f \rangle$. Finally, by the same arguments,
\begin{align*}
    \Var( \langle \xi_n^W, f \rangle)= \frac{\Var(\langle T_{n}(\Zbf_1), f \rangle)}{\Var(T_{n}(\Zbf_1))} \to \sigma(f),
\end{align*}
which shows the claim.
\end{claimproof}

Uniform tightness and convergence of univariate marginals together imply~\eqref{eq: uncondconvergence}.
\end{claimproof}

Let us consider again the function $D$ defined in~\eqref{Ddefinition} and the set $\F=\{F \in \H^*\colon \|F\|_{\H^*} \leq 1\}$ defined in~\eqref{Fdef}. As mentioned above, $D\colon \H \to \ell^{\infty}(\F)$ is continuous with a continuous inverse, and we consider the non-i.i.d empirical process
\begin{align*}
    D(\xi_n)= \frac{1}{\sqrt{n}} \sum_{i=1}^n
 D \left(\frac{T_{n}(\Zbf_i)}{\sqrt{\Var(T_n(\Zbf_1))}} \right)
\end{align*}
and similarly the multiplier process
\begin{align*}
    D(\xi_n^W)= \frac{1}{\sqrt{n}} \sum_{i=1}^n
 \tilde{W}_i D \left(\frac{T_{n}(\Zbf_i)}{\sqrt{\Var(T_n(\Zbf_1))}} \right).
\end{align*}
Using continuity of $D$, we showed $D(\xi_n^W) \stackrel{D}{\to} D(\xi)$, which in turn is a tight Gaussian element in $\ell^{\infty}(\F)$; see~\citet{B3}. 

Having shown unconditional convergence, we show conditional convergence of finite-dimensional marginals of $D(\xi_n^W)$: 

\begin{claim}
For all $K \in \N$ and $(f_1,\ldots,f_K) \in \F^{K}$,
\begin{align}\label{marginalconvergence}
   \left(D(\xi_n^W)(f_1), \ldots, D(\xi_n^W)(f_K)  \right) \xrightarrow[W]{D} \left(D(\xi)(f_1), \ldots, D(\xi)(f_K)  \right).
\end{align}
\end{claim}

\begin{claimproof}
By the Cramer-Wold device, it suffices to show 
\begin{align}\label{marginalconvergence2}
   \left(D(\xi_n^W)(f_1), \ldots, D(\xi_n^W)(f_K)  \right) \cdot  \wbf \xrightarrow[W]{D} \left(D(\xi)(f_1), \ldots, D(\xi)(f_K)  \right)\cdot  \wbf,
\end{align}
for any $ \wbf \in \R^K$. This in turn is implied if for all $F\colon\ell^{\infty}(\F) \to \R$ linear and continuous, it holds that
\begin{align}\label{marginalconvergence3}
    F(D(\xi_n^W)) \xrightarrow[W]{D} F(D(\xi))
\end{align}
because $F_K\colon \ell^{\infty}(\F) \to \R$, $F_K(D(\xi))=\left(D(\xi)(f_1), \ldots, D(\xi)(f_K)  \right)\cdot  \wbf$ is linear and continuous.
Consider a linear and continuous function $F\colon \mathcal{H} \to \R$.
Because $F \circ D\colon \mathcal{H} \to \R$ is linear and continuous from $\H$ to $\R$, by the Riesz representation theorem, we have 
\begin{align*}
    F(D(\xi_n^W))&=\frac{1}{\sqrt{n}} \sum_{i=1}^n
 \tilde{W}_i F \circ D \left(\frac{T_{n}(\Zbf_i)}{\sqrt{\Var(T_n(\Zbf_1))}}\right)\\
 &=\frac{1}{\sqrt{n}} \sum_{i=1}^n
 \tilde{W}_i \frac{\langle T_{n}(\Zbf_i), f_F \rangle}{\sqrt{\Var(T_n(\Zbf_1))}},
\end{align*}
for a unique $f_F \in \H$. Combining the arguments to prove the Lyapunov conditions in~\citet[Theorem 8]{wager2017estimation} with Theorem~\ref{amazingvarianceproposition}, we see that conditions~\eqref{eq: LyapconditionI} and~\eqref{eq: LyapconditionII} of Lemma~\ref{KLemma} hold for $Y_{ni}=\frac{\langle T_{n}(\Zbf_i), f_F \rangle}{\sqrt{n\Var(T_n(\Zbf_1))}}$. Similarly,~\citet[Lemma 12]{wager2017estimation} implies that~\eqref{eq: LyapconditionIextra} holds as well for $Y_{ni}$. Since $(\tilde{W}_i)_i$ is i.i.d.\ with expectation 0 and variance 1, it follows from Lemma~\ref{KLemma} that the Lyapunov condition for $\tilde{W}_iY_{ni}$ holds in probability, that is,~\eqref{firstassertioneq} and~\eqref{secondassertioneq} hold. Thus, we can find for any subsequence a further subsequence indexed by say $l$ such that Lyapunov condition for $\sum_{i} \tilde{W} Y_{li}$ given $\Zcal_{l}$ hold almost surely.  Arguing pointwise for fixed $\Zcal_{l}$ implies
\begin{align*}
 \sup_{h \in \text{BL}_1(\H) }  \left |  \E[ h( F(D(\xi_{l}^{\HS})) ) \mid \Zcal_{l}] - \E[ h(F(D(\xi))) ] \right | \to 0 \text{ a.s.};
\end{align*}
see~\citet{B2}.
Using an argument by contradiction as in~\citet[Lemma 14]{DRF-paper}, this in turn means
\[
 \sup_{h \in \text{BL}_1(\H) }  \left |  \E[ h( F(D(\xi_{l}^{\HS})) ) \mid \Zcal_{l}] - \E[ h(F(D(\xi))) ] \right | \stackrel{p}{\to} 0,
\]
proving the claim. 
\end{claimproof}

Combining unconditional convergence~\eqref{eq: uncondconvergence} and conditional finite-dimensional convergence~\eqref{marginalconvergence} with the arguments in~\citet[Theorem 2]{B2} then gives
\begin{align}
    D(\xi_n^W) \xrightarrow[W]{D} D(\xi).
\end{align}
Finally, due to continuity of the inverse of $D$, this implies~\eqref{eq: condconvergence}.

\end{claimproof}

Having shown~\eqref{eq: condconvergence}, it holds that

\begin{claim}
\begin{align}\label{eq: Wimin2}
    &\frac{1}{\sqrt{n}} \sum_{i=1}^n
 \left(\frac{n}{|\HS|}W_i -1 \right) \frac{T_{n}(\Zbf_i)}{\sqrt{\Var(T_n(\Zbf_1))}} - \frac{1}{\sqrt{n}} \sum_{i=1}^n
 (2 W_i-1)\frac{T_{n}(\Zbf_i)}{\sqrt{\Var(T_n(\Zbf_1))}} \nonumber \\
 =&  \left(\frac{n}{|\HS|} -2  \right) \frac{1}{\sqrt{n}} \sum_{i=1}^n
  W_i
\frac{T_{n}(\Zbf_i)}{\sqrt{\Var(T_n(\Zbf_1))}}  \\
\stackrel{p}{\to} & 0. 
\end{align}
\end{claim}

\begin{claimproof}

Indeed, $ \left(\frac{n}{|\HS|} -2  \right)=\o_p(1)$, and due to
\begin{align*}
    \Prob \left( \left \| \sum_{i=1}^n
  W_i
\frac{T_{n}(\Zbf_i)}{\sqrt{n\Var(T_n(\Zbf_1))}} \right\| > \varepsilon \right) &\leq \frac{1}{\varepsilon^2} \Var \left(\sum_{i=1}^n
  W_i
\frac{T_{n}(\Zbf_i)}{\sqrt{n\Var(T_n(\Zbf_1))}}  \right)\\
&=\frac{1}{\varepsilon^2} \frac{\Var(W_1 T_{n}(\Zbf_1) )}{\Var(T_n(\Zbf_1))}
\end{align*}
with
\begin{align*}
    \frac{\Var(W_1 T_{n}(\Zbf_1) )}{\Var(T_n(\Zbf_1))} = \frac{1/4 \Var( T_{n}(\Zbf_1) ) + 1/4 \Var( T_{n}(\Zbf_1) )}{\Var(T_n(\Zbf_1))} = \frac{1}{2} < \infty,
\end{align*}
we have
\begin{align*}
    \left\|\frac{1}{\sqrt{n}} \sum_{i=1}^n
  W_i
\frac{T_{n}(\Zbf_i)}{\sqrt{\Var(T_n(\Zbf_1))}}  \right \| = \O_p(1),
\end{align*}
establishing~\eqref{eq: Wimin2}.

\end{claimproof}

Thus, we have~\eqref{eq: condconvergence}, that is,
     $\xi_n^{W} \xrightarrow[W]{D} N(0, \boldsymbol{\Sigma}_{\xbf})$. Moreover, we have 
\begin{align*}
    \frac{1}{\sigma_{n}} \left( \hmunHS -  \hmun \right) =  \xi_n^{W} + o_p(1),
\end{align*}
by combining~\eqref{Sasymptoticlinearity} with~\eqref{eq: Wimin2}. Let as in the main text $\xi_n^{\HS}=\frac{1}{\sigma_{n}} \left( \hmunHS -  \hmun \right)$, and let $\mathcal{D}_n$ be the difference
\begin{align*}
   \mathcal{D}_n= \xi_n^{\HS} -  \xi_n^{W},
\end{align*}
so that $\|  \mathcal{D}_n \|_{\H}=o_p(1)$. With this, we can finally show that
\eqref{conditionalconvergence} holds, that is,
\begin{align*}
     \sup_{h \in \text{BL}_1(\H) }  \left |  \E \left[ h\left( \xi_n^{\HS} \right) \mid \Zcal_{n} \right] - \E[ h(\xi) ] \right | \stackrel{p}{\to} 0.
\end{align*}
Indeed, we have
\begin{align}\label{eq: breakup}
   &\sup_{h \in \text{BL}_1(\H) }  \left |  \E \left[ h\left(  \xi_n^{\HS}\right) \mid \Zcal_{n} \right] - \E[ h(\xi) ] \right |  \nonumber \\
   \leq& \sup_{h \in \text{BL}_1(\H) } |  \E \left[ h\left(  \xi_n^{\HS}\right) \mid \Zcal_{n} \right] - \E[ h(\xi_n^{W}) \mid \Zcal_n ] | + \sup_{h \in \text{BL}_1(\H) } | \E[ h(\xi_n^{W}) \mid \Zcal_n ]- \E[ h(\xi) ]|.
 \end{align}
The second term goes to zero in probability by~\eqref{eq: condconvergence}, and the first term satisfies
\begin{align*}
    \sup_{h \in \text{BL}_1(\H) } |  \E \left[ h\left(  \xi_n^{\HS}\right) \mid \Zcal_{n} \right] - \E[ h(\xi_n^{W}) \mid \Zcal_n ] | &\leq  \sup_{h \in \text{BL}_1(\H) }  \E [ | h(  \xi_n^{\HS}) - h(\xi_n^{W})| \mid \Zcal_n ] \\
    & \leq \E [ \min(\|\mathcal{D}_n\|_{\H}, 2) \mid \Zcal_n ]
\end{align*}
because for all $h \in \text{BL}_1(\H)$, $h$ is Lipschitz with constant bounded by 1, and $|h(f_1) - h(f_2)| \leq 2 \sup_{f \in \H} |h(f)|\leq 2$. Moreover, since $(\min(\|\mathcal{D}_n\|_{\H}, 2))_n$ is a bounded sequence, it is \emph{uniformly integrable}; see~\citet[Chapter 10.3]{dudley}. It follows by an extension of the Dominated Convergence Theorem for convergence in probability~\citep[Theorem 10.3.6]{dudley} that $\min(\|\mathcal{D}_n\|_{\H}, 2)=o_p(1)$, which implies $\E [ \min(\|\mathcal{D}_n\|_{\H}, 2)] \to 0$. Since $\min(\|\mathcal{D}_n\|_{\H}, 2)$ is also nonnegative and
\[
o(1)=\E [ \min(\|\mathcal{D}_n\|_{\H}, 2)] = \E[ \E [ \min(\|\mathcal{D}_n\|_{\H}, 2) \mid \Zcal_n] ],
\]
this implies that $\E [ \min(\|\mathcal{D}_n\|_{\H}, 2) \mid \Zcal_n] \stackrel{p}{\to} 0$. This convergence, together with the above bound, shows that the first part of~\eqref{eq: breakup} also goes to zero in probability.

\end{proof}

Finally, we show that the variance of finite dimensional marginals can be estimated consistently:

\variancestimation*

\begin{proof}

Define 
\begin{align}
   F \circ \boldsymbol{\hat{\Sigma}}_n = \E\left[ \left.\frac{1}{\sigma_n^2}  \left( F( \hmunHS ) -  F(\hmun) \right) \left( F( \hmunHS ) -  F(\hmun) \right)^{\top}   \,\right\vert\, \Zcal_n \right],
\end{align}
and note that $ F \circ \boldsymbol{\hat{\Sigma}}_n= \E\left[  F(\xi_n^{\HS})F(\xi_n^{\HS})^{\top} \mid \Zcal_n \right]$. Similarly, we define
\begin{align}
    F \circ\boldsymbol{\hat{\Sigma}}_n^{o}=\E\left[ \left. F(\xi_n^{W})F(\xi_n^{W})^{\top} \,\right\vert\, \Zcal_n \right],
\end{align}
with $\xi_n^{W}$ defined as in~\eqref{eq: xinWdef}. We will first show in several steps that:
\begin{claim}
For all $ \wbf \in \R^q$, we have 
\begin{align}\label{eq: univariateconvergencehatSigma}
 \wbf^{\top}( F \circ \boldsymbol{\hat{\Sigma}}_n)  \wbf \stackrel{p}{\to}  \wbf^{\top}(F \circ \boldsymbol{\Sigma}_{\xbf}) \wbf.    
\end{align}

\end{claim}

\begin{claimproof}

To prove the claim, we first show:

\begin{claim}
For all $ \wbf \in \R^q$,  we have 
\begin{align}\label{eq: univariateconvergencehatSigmao}
 \wbf^{\top}( F \circ \boldsymbol{\hat{\Sigma}}_n^{o})  \wbf \stackrel{p}{\to}  \wbf^{\top}(F \circ \boldsymbol{\Sigma}_{\xbf}) \wbf.    
\end{align}

\end{claim}

\begin{claimproof}
First, note that we may define $F_{ \wbf} \in \H^*$ by
$F_{ \wbf}(f) =   \wbf^{\top} F(f)$.
Particularly, it is linear, and $\|F_{ \wbf}(f_1) - F_{ \wbf}(f_2)\| \leq \|\wbf \|_{\R^q} \|F(f_1)-F(f_2)\|_{\R^q}$, so that it is also continuous. Then, we have 
\begin{align*}
     \wbf^{\top}( F \circ \boldsymbol{\hat{\Sigma}}_n^{o})  \wbf &=\E \left[ \left.\left \| \frac{1}{\sqrt{n}} \sum_{i=1}^n
 \tilde{W}_i\frac{F_{ \wbf}\circ T_{n}(\Zbf_i)}{\sqrt{\Var( T_n(\Zbf_1))}} \right \|^2 \,\right\vert\, \Zcal_n \right]\\
 &= \frac{1}{n} \sum_{i=1}^n  \frac{ (F_{ \wbf}\circ T_{n}(\Zbf_i))^2}{\Var( T_n(\Zbf_1))},
\end{align*}
because $\E [ 
 \tilde{W}_i^2 ]=1$ and because the cross-terms are of the form
\[
\E [ \tilde{W}_i \tilde{W}_j  ] \frac{F_{ \wbf}\circ T_{n}^2(\Zbf_i) \cdot F_{ \wbf}\circ T_{n}^2(\Zbf_j)}{\Var( T_n(\Zbf_1))}=0.
\]
As already argued in the proof of Theorem~\ref{thm: asymptoticnormalityhalfsampling}, under Assumption~\forestass{1}--\forestass{5} and  \dataass{1}--\dataass{7}, the arguments in the proof of Lemma 12 in~\citet{wager2017estimation} imply that
\begin{align*}
    \frac{1}{n} \sum_{i=1}^n  \frac{ (F_{ \wbf}\circ T_{n}(\Zbf_i))^2}{\Var( T_n(\Zbf_1))} =    \frac{1}{n} \sum_{i=1}^n  \frac{ \langle f_{ \wbf}, T_{n}(\Zbf_i) \rangle^2}{\Var( T_n(\Zbf_1))} \stackrel{p}{\to} \sigma^2(f_{ \wbf})
\end{align*}
for the unique $f_{ \wbf} \in \H$ given by the Riesz representation theorem. Moreover, by consistency arguments, we have $\sigma^2(f_{ \wbf})= \wbf^{\top}( F \circ \boldsymbol{\Sigma}_{\xbf}) \wbf$, proving the claim.
\end{claimproof}

In the proof of Theorem~\ref{thm: asymptoticnormalityhalfsampling}, we showed
$\xi_n^{\HS}=\xi_n^{W} + o_p(1)$. To show that~\eqref{eq: univariateconvergencehatSigma} follows from~\eqref{eq: univariateconvergencehatSigmao}, we now strengthen this to:
\begin{claim}
\begin{align}\label{eq: Convergenceinsquaredmean2}
   \E[\|\xi_n^{\HS}-\xi_n^{W}  \|_{\H}^2]=o(1). 
\end{align}
\end{claim}

\begin{claimproof}
We recall the argument in the beginning of Theorem~\ref{thm: asymptoticnormalityhalfsampling}. By construction, we always consider the same forest and just use different trees or subsamples for each $\HS$, namely such that the subset of size $s_n$ is included in $\HS$. Since $s_n$ is of smaller order than $n$, $\P(|\HS| \leq s_n) \to 0$, as $n \to \infty$. Thus, by the same arguments as in~\citet[Theorem 5]{athey2019generalized} combined with the claim~\eqref{eq: Convergenceinsquaredmean}, we have
\begin{align*}
    \E \left[ \left \| \frac{1}{\sigma_n} \left( (\hmunHS -  \E[ \hmun]) - \frac{s_n}{n} \sum_{i \in \HS} \frac{n}{|\HS|} T_{n}(\Zbf_i) \right)\right \|_{\H}^2 \1\{|\HS| > s_n \} \right] \to 0.
\end{align*}
Moreover, using that we have
\begin{align*}
    \frac{\| \E[\hat{\mu}(\xbf)] - \mu(\xbf) \|_{\H}}{\sigma_n} =o(1),
\end{align*}
as shown in the proof of Theorem~\ref{thm: asymptoticlinearity}, this convergence also holds with $\E[\hat{\mu}(\xbf)]$ replaced by $\mu(\xbf)$, or 
\begin{align*}
    \E \left[ \left \|  \xi_n^{\HS}- \xi_n^{W} \right \|_{\H}^2 \1\{|\HS| > s_n \} \right] \to 0.
\end{align*}

In the case $|\HS| \leq s_n$, we set  
\begin{align*}
    \sum_{i \in \HS} \frac{n}{|\HS|} T_{n}(\Zbf_i) =\hmunHS=0 \in \H
\end{align*}
to zero. 
Then, we have 
\begin{align*}
    &\Big|\E \left[ \left \|  \xi_n^{\HS}- \xi_n^{W} \right \|_{\H}^2  \right] -\E \left[ \left \|  \xi_n^{\HS}- \xi_n^{W} \right \|_{\H}^2 \1\{|\HS| > s_n \} \right]\Big| \\
    =&\E \left[ \left \|  \xi_n^{\HS}- \xi_n^{W} \right \|_{\H}^2 \1\{|\HS| \leq s_n \} \right]\\
    =&\left \|  \mu(\xbf) \right \|_{\H}^2  \frac{\P(|\HS| \leq  s_n)}{\sigma_n^2} .
\end{align*}
From the proof of Theorem~\ref{thm: asymptoticnormality} and the fact that $s_n=n^{\beta}$ with $\beta < 1$, it follows that
\begin{align*}
    \sigma_n^2=\Omega \left(\frac{ s_n}{n \log(s_n)^p}\right)=\Omega(n^{\beta -(1 + \varepsilon)} )
\end{align*}
for $\varepsilon > 0$ arbitrarily small. On the other hand, we can employ a Hoeffding bound on $\P(|\HS| \leq  s_n)$ to obtain
\begin{align*}
    \P(|\HS| \leq  s_n) = \P(|\HS| - n/2 \leq  s_n - n/2) \leq \frac{\Var(|\HS|)}{(s_n-n/2)^2}=\frac{n/4}{s_n^2 + n^2/4 - s_n n},
\end{align*}
so that
\begin{align*}
    \P(|\HS| \leq  s_n)  \leq  \frac{1}{4} \frac{1}{n/4 + n^{2\beta-1}  - n^{\beta}} =\O\left(\frac{1}{n}\right).
\end{align*}
This results in
\begin{align*}
    \frac{\P(|\HS| \leq  s_n)}{\sigma_n^2}=\O \left(n^{1+\varepsilon - \beta - 1}  \right)=\O \left(n^{\varepsilon - \beta}  \right).
\end{align*}
Since $\varepsilon$ can be chosen arbitrarily small, this converges to 0.

\end{claimproof}

Having shown~\eqref{eq: Convergenceinsquaredmean2}, we have that

\begin{align*}
     \wbf^{\top}( F \circ \boldsymbol{\hat{\Sigma}}_n)  \wbf &= \E[  \wbf^{\top} ( F(\xi_{n}^{W} )+ F(D_n) ) ( F(\xi_{n}^{W} )+ F(D_n) )^{\top}  \wbf \mid \Zcal_n]\\
    &=   \wbf^{\top}( F \circ \boldsymbol{\hat{\Sigma}}_n^{o})  \wbf+ \E[  \wbf^{\top} F(D_n) F(D_n)^{\top}  \wbf \mid \Zcal_n] + 2 \E[  \wbf^{\top} F(\xi_{n}^{W}) F(D_n)^{\top}  \wbf \mid \Zcal_n],
\end{align*}
where $D_n= \xi_{n}^{\HS} - \xi_{n}^{W} $. Note that $\E[\|D_n\|^2_{\H}]=\E[ \E[\|D_n\|^2_{\H} \mid \Zcal_n]]=o(1)$ 
implies that $\E[\|D_n\|^2_{\H} \mid \Zcal]=o_p(1)$; see~\citet[Lemma 2.2.2]{durrett}. Moreover, 
\begin{align*}
    \E[  \wbf^{\top} F(D_n) F(D_n)^{\top}  \wbf \mid \Zcal] &=  \E[ | \wbf^{\top} F(D_n)|^2  \mid \Zcal] \\
    & \leq  \|F_{ \wbf}\|_{\H^*}^2 \E[ \|D_n\|_{\H}^2  \mid \Zcal]\\
    & = o_p(1) 
\end{align*}
by~\eqref{eq: Convergenceinsquaredmean2}. Similarly, by H\"older's inequality,
\begin{align*}
     \E[  \wbf^{\top} F(\xi_{n}^{W}) F(D_n)^{\top}  \wbf \mid \Zcal_n] \leq  \E[ | \wbf^{\top} F(\xi_{n}^{W})|^2 \mid \Zcal_n]^{1/2} \cdot \E[ |  \wbf^{\top} F(D_n)|^2 \mid \Zcal_n]^{1/2} \stackrel{p}{\to} 0.
\end{align*}
Thus, $\left|   \wbf^{\top}( F \circ \boldsymbol{\hat{\Sigma}}_n)  \wbf -    \wbf^{\top}( F \circ \boldsymbol{\hat{\Sigma}}_n^o)  \wbf \right| \stackrel{p}{\to} 0$, which shows~\eqref{eq: univariateconvergencehatSigma}.
 
\end{claimproof}

Finally,~\eqref{eq: univariateconvergencehatSigma} implies the result. Indeed, for a matrix $\mathbf{A} \in \R^{q \times q}$, define the operator $\vect(\mathbf{A})\in \R^{q^2}$ that concatenates the rows of $\mathbf{A}$ on top of each other. This operator is continuous and invertible with a continuous inverse. Moreover, for any $ \wbf$, we can consider the element $\tilde{\wbf}=\wbf \otimes \wbf^{\top}\in\R^{q^2}$ satisfying $ \wbf^{\top}\mathbf{A}  \wbf=\mathbf{\tilde{w}}^{\top} \vect(\mathbf{A})$ such that we have
\begin{align*}
    \tilde{ \wbf}^{\top} \vect(F \circ \boldsymbol{\hat{\Sigma}}_n) =  \wbf^{\top} (F \circ \boldsymbol{\hat{\Sigma}}_n)  \wbf \stackrel{p}{\to}  \wbf^{\top} (F \circ \boldsymbol{\Sigma}_{\xbf})  \wbf = \tilde{ \wbf}^{\top} \vect(F \circ \boldsymbol{\Sigma}_{\xbf}). 
\end{align*}
Utilizing the Cramer-Wold device and the fact that convergence in distribution to a constant is equivalent to convergence in probability, this implies that $\vect(F \circ \boldsymbol{\hat{\Sigma}}_n) \stackrel{p}{\to} \vect(F \circ \boldsymbol{\Sigma}_{\xbf})$. By continuity of the inverse of the $\vect$ operator, this implies the result.

\end{proof}

\variancestimationtwo*

\begin{proof}

First, we have 
\begin{align*}
    &\frac{\E[\| \hmunHS  -  \hmun \|_{\H}^2 \mid \Zcal_n]}{\sigma_n^2}\\
    =&\E\left[  \| \xi_{n}^{\HS} -  \xi_n^{W}\|_{\H}^2 \mid \Zcal_n \right] + \E \left[  \|\xi_n^{W}\|^2_{\H} \mid \Zcal_n \right] + 2 \E[ \langle \xi_{n}^{\HS} -  \xi_n^{W}, \xi_n^{W} \rangle \mid \Zcal_n ],
\end{align*}
where we recall
\begin{align*}
\xi_{n}^{\HS}&=\frac{1}{\sigma_n}\left( \hmunHS -  \hmun \right),\\
    \xi_n^{W}&= \frac{1}{\sqrt{n}} \sum_{i=1}^n
 \tilde{W}_i\frac{T_{n}(\Zbf_i)}{\sqrt{\Var(T_n(\Zbf_1))}}.
\end{align*}
As we proved in Corollary~\ref{thm: variancestimation}, as a consequence of~\eqref{eq: Convergenceinsquaredmean2}, we have
\begin{align*}
    \E\left[  \| \xi_{n}^{\HS} -  \xi_n^{W}\|_{\H}^2 \mid \Zcal_n \right]=o_p(1).
\end{align*}
Moreover, using Cauchy--Schwarz inequality and H\"older's inequality, we have
\begin{align*}
\E[ |\langle \xi_{n}^{\HS} -  \xi_n^{W}, \xi_n^{W} \rangle|  ] 
   & \leq \E[ \| \xi_{n}^{\HS} -  \xi_n^{W}\|_{\H} \|\xi_n^{W}\|_{\H}  ] \\
   &\leq \E[ \| \xi_{n}^{\HS} -  \xi_n^{W}\|_{\H}^2]^{1/2} \E[\|\xi_n^{W}\|_{\H}^2 ]^{1/2}.
\end{align*}
Recall that we argued  $\E[\|  \xi_n^{W}\|^2_{\H} ] =1$, and $\E[\| \xi_{n}^{\HS} -  \xi_n^{W}\|^2_{\H} ]=o(1)$ above. This thus implies  $\E[ |\langle \xi_{n}^{\HS} -  \xi_n^{W}, \xi_n^{W} \rangle|  ]=o(1)$, which in turn implies  
\begin{align*}
   | \E[ \langle \xi_{n}^{\HS} -  \xi_n^{W}, \xi_n^{W} \rangle \mid \Zcal_n ]| \leq   \E[| \langle \xi_{n}^{\HS} -  \xi_n^{W}, \xi_n^{W} \rangle| \mid \Zcal_n ] = o_p(1).
\end{align*}
Thus, it remains to show:
\bigskip

\begin{claim}
$ \E \left[  \|\xi_n^{W}\|^2_{\H} \mid \Zcal_n \right] \stackrel{p}{\to} 1$.
\end{claim}

\begin{claimproof}
First, note that
\begin{align*}
   & \E \left[ \left.  \left\langle
 \tilde{W}_i\frac{T_{n}(\Zbf_i)}{\sqrt{\Var(T_n(\Zbf_1))}},  \tilde{W}_j\frac{T_{n}(\Zbf_j)}{\sqrt{\Var(T_n(\Zbf_1))}}  \right\rangle  \,\right\vert\, \Zcal_n \right] \\
 =& \E [ \tilde{W}_i \tilde{W}_j  ]\left\langle
 \frac{T_{n}(\Zbf_i)}{\sqrt{\Var(T_n(\Zbf_1))}},  \frac{T_{n}(\Zbf_j)}{\sqrt{\Var(T_n(\Zbf_1))}} \right \rangle\\
 =&0.
\end{align*}
Consequently,
\begin{align*}
     \E \left[  \|\xi_n^{W}\|^2_{\H} \mid \Zcal_n \right]& =    \E \left[ \left. \left \|\frac{1}{\sqrt{n}} \sum_{i=1}^n
 \tilde{W}_i\frac{T_{n}(\Zbf_i)}{\sqrt{\Var(T_n(\Zbf_1))}} \right \|^2_{\H} \,\right\vert\, \Zcal_n \right]\\
 &=\frac{1}{n} \sum_{i=1}^n  \E [ \tilde{W}_i^2  \mid \Zcal_n ] \left \|
 \frac{T_{n}(\Zbf_i)}{\sqrt{\Var(T_n(\Zbf_1))}} \right \|^2_{\H}\\
  &=\frac{1}{n\Var(T_n(\Zbf_1))} \sum_{i=1}^n  
 \|T_{n}(\Zbf_i)\|^2_{\H} .
\end{align*}
Thus, we need to show that
\begin{align*}
    \frac{\frac{1}{n} \sum_{i=1}^n  
 \|T_{n}(\Zbf_i)\|^2_{\H}}{\Var(T_n(\Zbf_1))} \stackrel{p}{\to} 1.
\end{align*}
But due to assumption~\ref{dataass3}, this can be shown using the same steps as at the end of the proof of Lemma 12 in~\citet{wager2017estimation}, with $\|T_{n}(\Zbf_i)\|^2_{\H}$ in place of their $T_1^2(Z_i)$.

\end{claimproof}

\end{proof}

\normresulttwogroups*

\begin{proof}
Using independence of $\hmuno$ and $\hmunz$ for all $n_0$, $n_1$, 
together with Theorem~\ref{thm: asymptoticnormality}, it follows that
\begin{align}\label{sumasyptoticnormality}
    \frac{1}{\sigma_{n_1,1}}(\hmuno -\mu_1(\xbf))  - \frac{1}{\sigma_{n_0,0}}(\hmunz - \mu_0(\xbf)) \stackrel{D}{\to} \xi_1-\xi_0. 
\end{align}
Similarly, due to independence of $\HS_0$ and $\HS_1$, the arguments in the proof of Theorem~\ref{thm: asymptoticnormalityhalfsampling} can be repeated to obtain
\begin{align*}
 \sup_{h \in \text{BL}_1(\H) }  \left |  \E\left[\left. h  \left(\frac{1}{\sigma_{n_1,1}} (\hmunHSo-\hmuno) - \frac{1}{\sigma_{n_0,0}}(\hmunHSz - \hmunz) \right) \right\vert \Zcal_{n_{01}}\right] - \E[ h(\xi_1-\xi_0) ] \right | \stackrel{p}{\to} 0,
\end{align*}
or in other words
\begin{align}\label{sumasyptoticnormalityHS}
    \frac{1}{\sigma_{n_1,1}} (\hmunHSo-\hmuno) - \frac{1}{\sigma_{n_0,0}}(\hmunHSz - \hmunz) \xrightarrow[W]{D} \xi_1-\xi_0.
\end{align}
Finally, if~\eqref{Techassumption} holds, it follows from the arguments in the proof of Theorem~\ref{amazingvarianceproposition} 
that, ignoring smaller order terms,
\begin{align*}
    \frac{\sigma_{n_0,0}^2}{\sigma_{n_1,1}^2} &= \frac{\Var( \E[\frac{1}{N_{\xbf}^0} \1\{\Xbf_2 \in \mathcal{L}^0(\xbf) \} \mid \Xbf_1] )}{\Var( \E[\frac{1}{N_{\xbf}^1} \1\{\Xbf_2 \in \mathcal{L}^1(\xbf) \} \mid \Xbf_1] )} \frac{\Var(k(\Ybf^0, \cdot) \mid \Xbf=\xbf)}{\Var(k(\Ybf^1, \cdot) \mid \Xbf=\xbf)} \\
    &\to c(\xbf) \frac{\Var(k(\Ybf^0, \cdot) \mid \Xbf=\xbf)}{\Var(k(\Ybf^1, \cdot) \mid \Xbf=\xbf)} = c_2(\xbf)
\end{align*}
where $c(\xbf)$ is as in~\eqref{Techassumption}. 
It thus follows from Slutsky's theorem that
\begin{align}\label{sumasyptoticnormality2}
    &\frac{1}{\sigma_{n_1,1}} (\hmuno-\mu_1(\xbf)) -\frac{1}{\sigma_{n_1,1}} (\hmunz - \mu_0(\xbf)) \nonumber \\
    =&\frac{1}{\sigma_{n_1,1}} (\hmuno-\mu_1(\xbf)) -\frac{\sigma_{n_0,0}}{\sigma_{n_1,1}} \frac{1}{\sigma_{n_0,0}} (\hmunz - \mu_0(\xbf)) \\\stackrel{D}{\to}& \xi_1- c_2(\xbf) \xi_0.
\end{align}
and similarly
\begin{align}\label{sumasyptoticnormality2HS}
    \frac{1}{\sigma_{n_1,1}} (\hmunHSo-\hmuno) -  \frac{1}{\sigma_{n_1,1}} (\hmunHSz - \hmunz) \xrightarrow[W]{D} \xi_1- c_2(\xbf) \xi_0.
\end{align}

Since $f \mapsto \|f \|_{\H}^2$ is continuous,~\eqref{conditionalindep4}--\eqref{conditionalindep2} follow from~\eqref{sumasyptoticnormality}--\eqref{sumasyptoticnormality2HS} combined with the continuous mapping theorem.

\end{proof}

\powerprop*

\begin{proof}
To simplify the proof, we assume 
$n_0=n_1=n$.
Since we assume that $n_0/n_1 \to 1$, this will not impact our asymptotic results. 
Let in the following for $j \in \{0,1\}$ 
\[
\xi_{n,j}^{\HS_j}= \frac{1}{\sigma_{n,1}} \left( \hat{\mu}_{n,j}^{\HS_j}(\mathbf{x}) - \hat{\mu}_{n,j}(\mathbf{x})\right),
\]
where we emphasize the fixed $1$ in $\sigma_{n,1}$.
We first note that by~\eqref{conditionalindep42}, the sequence $\| \xi_{n,1}^{\HS_1} -\xi_{n,0}^{\HS_0} \|_{\H}^2$, $n \in \N$, is uniformly tight, which in turn implies that there exists a large enough number $M_{\alpha} < \infty$ such that we have
\begin{align*}
    \sup_{n} \Prob \left( \| \xi_{n,1}^{\HS_1} -\xi_{n,0}^{\HS_0} \|_{\H}^2 > M_{\alpha}   \right) \leq \alpha.
\end{align*}
Since for each $n$, $c_{n,\alpha}$ is the smallest value such that~\eqref{finitealphabound} holds, we have $c_{n, \alpha} \leq M_{\alpha}<\infty$ for all $n$. In particular, $\sup_{n} c_{n, \alpha} \leq M_{\alpha} < \infty$, and $c_{n, \alpha}$ is a bounded sequence in $\R$. This allows us to find a convergent subsequence below. Second, 
if $\PYgXz=\PYgXo$, 
\[
  \Prob \left(\frac{1}{\sigma_{n,1}^2}\left \| \hat{\mu}_{n,1}(\mathbf{x} )- \hat{\mu}_{n,0}(\mathbf{x}) \right \|_{\H}^2 > z \right) \to \Prob(\|\xi_1 - c_2(\xbf)\xi_0\|_{\H}^2 > z )
\]
for all $z \geq 0$ by~\eqref{conditionalindep1}. We consider an arbitrary subsequence of $n$, $n(\ell)$. By~\eqref{conditionalindep4}, a further subsequence $m=n(\ell(m))$ can be chosen such that
\begin{align}\label{asconvergenceindistribution}
 \sup_{h \in \text{BL}_1(\H) }  \left |  \E[  h( \|\xi_{m,1}^{\HS_1} -\xi_{m,0}^{\HS_0}\|_{\H}^2 )  \mid \Zcal_{2m}] - \E[  h(\|\xi_{1}- c_2(\xbf)\xi_{0}\|_{\H}^2)  ] \right | \to 0
\end{align}
almost surely. Now, we argue pointwise for each realization $(z_{2m})_{m \in \N}$ of $(\Zcal_{2m})_{m \in \N}$  such that~\eqref{asconvergenceindistribution} holds.
As convergence in distribution implies convergence of CDF's at continuity points~\citep[Theorem 9.3.6]{dudley}, this implies that
\begin{align*}
    \Prob \left(\left.\| \xi_{m,1}^{\HS_1} -\xi_{m,0}^{\HS_0} \|_{\H}^2 > z \right\vert \Zcal_{2m} \right) \to \Prob(\|\xi_1 - c_2(\xbf)\xi_{0}\|_{\H}^2 > z )
\end{align*}
almost surely for all $z \geq 0$. 

\begin{claim}
There exists a further subsequence $l=n(\ell(m(l)))$ such that $\lim_{l} c_{l, \alpha} = c_{\alpha}$ exists that satisfies
\begin{align}\label{H0convergence_S}
    \alpha \geq \Prob \left(\left \| \xi_{l,1}^{\HS_1} -\xi_{l,0}^{\HS_0}\right \|_{\H}^2 > c_{l,\alpha} \mid \Zcal_{2l} \right) \to \Prob(\|\xi_1-c_2(\xbf)\xi_0\|_{\H}^2 > c_{\alpha} )
\end{align}
almost surely. Moreover, if $\PYgXz=\PYgXo$, we also have
\begin{align}\label{H0convergence}
    \Prob \left(\frac{1}{\sigma_{l,1}^2}\left \| \hat{\mu}_{l,1}(\mathbf{x} )- \hat{\mu}_{l,0}(\mathbf{x}) \right \|_{\H}^2 > c_{l,\alpha} \right) \to \Prob(\|\xi_1-c_2(\xbf)\xi_0\|_{\H}^2 > c_{\alpha} ).
\end{align}
\end{claim}

\begin{claimproof}
First, since $c_{m,\alpha}$, $m \in \N$, is a bounded sequence as discussed above, we can find a convergent subsequence indexed by $l$ such that $c_{l,\alpha} \to c_{\alpha}$, where $c_{\alpha} \in \R$ might depend on the chosen subsequence. Using Slutsky's theorem, we have that
\begin{align*}
    \frac{1}{\sigma_{l,1}^2}\left \|\hat{\mu}_{l,1}(\mathbf{x} )- \hat{\mu}_{l,0}(\mathbf{x}) \right \|_{\H}^2 - c_{l,\alpha} \stackrel{D}{\to} \|\xi_1-c_2(\xbf)\xi_0\|_{\H}^2-c_{\alpha}.
\end{align*}
Consequently, if $\PYgXz=\PYgXo$, we have
\[
\Prob \left(\frac{1}{\sigma_{l,1}^2}\left \| \hat{\mu}_{l,1}(\mathbf{x} )- \hat{\mu}_{l,0}(\mathbf{x}) \right \|_{\H}^2 - c_{l,\alpha}> 0 \right) \to \Prob(\|\xi_1-c_2(\xbf)\xi_0\|_{\H}^2 - c_{\alpha}> 0 ).
\]
Similarly, arguing again pointwise for a realization $z_{l}$, $l \in \N$, and using Slutsky's theorem, we have
\[
\Prob \left(\left \| \xi_{l,1}^{\HS_1} -\xi_{l,0}^{\HS_0} \right \|_{\H}^2 - c_{l,\alpha} > 0 \mid \Zcal_{2l} \right) \to \Prob \left(\left \| \xi_{1} - c_2(\xbf)\xi_{0} \right \|_{\H}^2 - c_{\alpha} > 0\right)
\]
almost surely.
\end{claimproof}

We note that $c_{\alpha}$, and thus the limit $\Prob(\|\xi_1-c_2(\xbf)\xi_0\|_{\H}^2 > c_{\alpha} )$, might depend on the chosen subsequence. However, the $\alpha$-bound in~\eqref{H0convergence_S} holds by construction. Consequently, it follows from~\eqref{H0convergence} that we have
\begin{align}\label{limitequality}
     \lim_{l}\Prob \left(\frac{1}{\sigma_{l,1}^2}\left \| \hat{\mu}_{l,1}(\mathbf{x} )- \hat{\mu}_{l,0}(\mathbf{x}) \right \|_{\H}^2 > c_{l,\alpha} \right)= \lim_{l} \Prob \left(\left \|\xi_{l,1}^{\HS_1} -\xi_{l,0}^{\HS_0} \right \|_{\H}^2 > c_{l,\alpha} \mid \Zcal_{l} \right) \leq  \alpha
\end{align}
almost surely under $\PYgXz=\PYgXo$.

Thus, we found that for every subsequence, there exists a further subsequence such that~\eqref{limitequality} holds. Now, assume that for the overall sequence
\[
\limsup_{n}\Prob \left(\frac{1}{\sigma_{n,1}^2}\left \| \hat{\mu}_{n,0}(\mathbf{x} )- \hat{\mu}_{n,1}(\mathbf{x}) \right \|_{\H}^2 > c_{n,\alpha} \right) > \alpha.
\]
Then, we can choose a subsequence satisfying
\begin{align*}
   \lim_{m}\Prob \left(\frac{1}{\sigma_{m,1}^2}\left \| \hat{\mu}_{m,0}(\mathbf{x} )- \hat{\mu}_{m,1}(\mathbf{x}) \right \|_{\H}^2 > c_{m,\alpha} \right) > \alpha. 
\end{align*}
But for this sequence, it is not possible to find a further subsequence $l$ such that
\begin{align*}
     \lim_{\ell}\Prob \left(\frac{1}{\sigma_{l,1}^2}\left \| \hat{\mu}_{l,0}(\mathbf{x} )- \hat{\mu}_{l,1}(\mathbf{x}) \right \|_{\H}^2 > c_{l,\alpha} \right) \leq  \alpha,
\end{align*}
a contradiction to~\eqref{limitequality}.
\bigskip

On the other hand, since~\ref{kernelass3} holds, $\PYgXz \neq \PYgXo$ implies $\mu_1(\xbf) \neq \mu_2(\xbf)$. Moreover, we have
\begin{align*}
    &\frac{1}{\sigma_{n,1}} (\hmuno - \hmunz)\\
    =& \frac{1}{\sigma_{n,1}} (\hmuno - \mu_1(\xbf)) +  \frac{1}{\sigma_{n,1}} (\mu_1(\xbf) - \mu_0(\xbf))-  \frac{\sigma_{n,0}}{\sigma_{n,1}} \frac{1}{\sigma_{n,0}} (\hmunz - \mu_0(\xbf) ).
\end{align*}
Define 
\begin{align*}
    \xi_{n}^{01} =  \frac{1}{\sigma_{n,1}} (\hmuno - \mu_1(\xbf)) - \frac{\sigma_{n,0}}{\sigma_{n,1}} \frac{1}{\sigma_{n,0}} (\hmunz - \mu_0(\xbf) )
\end{align*}
Next, we have
\begin{align*}
  &\frac{1}{\sigma_{n,1}^2}\left \| \hat{\mu}_{n,1}(\mathbf{x} )- \hat{\mu}_{n,0}(\mathbf{x}) \right \|_{\H}^2\\
  =& \|\xi_{n}^{01} \|_{\H}^2 + \frac{1}{\sigma_{n,1}^2} \|\mu_1(\xbf) - \mu_0(\xbf) \|_{\H}^2 + 2\left\langle  \xi_{n}^{01},\frac{1}{\sigma_{n,1}} (\mu_1(\xbf) - \mu_0(\xbf))   \right\rangle\\
  \geq&  \|\xi_{n}^{01} \|_{\H}^2 + \frac{1}{\sigma_{n,1}^2} \|\mu_1(\xbf) - \mu_0(\xbf) \|_{\H}^2 - 2\left|\left\langle  \xi_{n}^{01},\frac{1}{\sigma_{n,1}} (\mu_1(\xbf) - \mu_0(\xbf))  \right\rangle\right|\\
    \geq&  \|\xi_{n}^{01} \|_{\H}^2 + \frac{1}{\sigma_{n,1}^2} \|\mu_1(\xbf) - \mu_0(\xbf) \|_{\H}^2 - 2 \| \xi_{n}^{01} \|_{\H} \frac{1}{\sigma_{n,1}} \|\mu_1(\xbf) - \mu_0(\xbf) \|_{\H}.
\end{align*}
Due to $\|\xi_{n}^{01} \|_{\H}^2 \stackrel{D}{\to} \|\xi_0- c_2(\xbf)\xi_1 \|_{\H}^2$, we infer
\[
\|\xi_{n}^{01} \|_{\H}^2=\mathcal{O}_p(1).
\]
For the remaining terms, we have
\begin{align*}
    &\frac{1}{\sigma_{n,1}^2} \|\mu_1(\xbf) - \mu_0(\xbf) \|_{\H}^2 - 2 \| \xi_{n}^{01} \|_{\H} \frac{1}{\sigma_{n,1}} \|\mu_1(\xbf) - \mu_0(\xbf) \|_{\H} \\
    =&\frac{1}{\sigma_{n,1}} \|\mu_1(\xbf) - \mu_0(\xbf) \|_{\H}\left(\frac{1}{\sigma_{n,1}} \|\mu_1(\xbf) - \mu_0(\xbf) \|_{\H} - 2 \| \xi_{n}^{01} \|_{\H}\right) \\
    \to& \infty,
\end{align*}
as $\|\mu_1(\xbf) - \mu_0(\xbf) \|_{\H} >0$ and $\sigma_{n,1} \to 0$. But since $\sup_{n} c_{n,\alpha} \leq M_{\alpha}$, this immediately implies that (ii) must hold.

\end{proof}

\coverprop*

\begin{proof}

Due to~\ref{kernelass2}, the kernel $k$ is bounded. Without loss of generality, we assume that $C=1$ bounds the kernel, so that $\sup_{\ybf} k(\ybf, \ybf) \leq 1$. Moreover, we assume again $n_0=n_1=n$, which does not affect asymptotics.

First, by definition of $ \bandybf$, we have 
\begin{align*}
    &\Prob \left(\forall \ybf \ \  \mu_1(\xbf)(\ybf)-\mu_0(\xbf)(\ybf) \in   \bandybf \right)  \\
    =& \Prob \left(\forall \ybf \ \ |\hat{\mu}_{n,1}(\xbf)(\ybf) - \hat{\mu}_{n,0}(\xbf)(\ybf)-(\mu_1(\xbf)(\ybf)-\mu_0(\xbf)(\ybf))| \leq \sqrt{c_{n,\alpha}} \sigma_{n,1} \right).
\end{align*}
The probability of the complementary event is given by
\begin{align*}
   & \Prob \left(\exists \ybf \ \ |\hat{\mu}_{n,1}(\xbf)(\ybf) - \hat{\mu}_{n,0}(\xbf)(\ybf)-(\mu_1(\xbf)(\ybf)-\mu_0(\xbf)(\ybf))| > \sqrt{c_{n,\alpha}} \sigma_{n,1} \right)\\
    =&\Prob \left(\exists \ybf \ \ |\langle \hat{\mu}_{n,1}(\xbf) - \hat{\mu}_{n,0}(\xbf)-(\mu_1(\xbf)-\mu_0(\xbf)), k(\ybf, \cdot)\rangle| > \sqrt{c_{n,\alpha}} \sigma_{n,1} \right)\\
    \leq&\Prob \left( \frac{1}{\sigma_{n,1}} \|\hat{\mu}_{n,1}(\xbf) - \hat{\mu}_{n,0}(\xbf)-(\mu_1(\xbf)-\mu_0(\xbf))\|_{\H} > \sqrt{c_{n,\alpha}} \right)
 \end{align*}
 due to $\|k(\ybf,\cdot) \|_{\H}^2 = k(\ybf,\ybf) \leq 1$.
By the same arguments as in the proof of Theorem~\ref{powerprop} together with~\eqref{conditionalindep1}, we have 
\begin{align*}
    \limsup_{n}\Prob \left(\frac{1}{\sigma_{n,1}^2}\left \| \hat{\mu}_{n,1}(\mathbf{x} )- \hat{\mu}_{n,0}(\mathbf{x})  - (\mu_1(\xbf) - \mu_0(\xbf)) \right \|_{\H}^2 > c_{n,\alpha} \right) \leq \alpha,
\end{align*}
 implying~\eqref{CIguarantee}. 
 \end{proof}

\section{Additional Discussion on Assumptions}\label{appendix: assumptions}

In this section, we study an example for which Assumption \dataass{2} holds and show that \dataass{4} is trivally satisfied under the Gaussian kernel, if $\PYgX$ has positive variance. We first restate both for convenience:
\begin{enumerate}
        \item[\textbf{(D2)}] The mapping $\xbf \mapsto \mu(\xbf)=\E[ k(\Ybf,\cdot)  \mid \Xbf \myeq \xbf] \in \H $ is Lipschitz.
                \item[\textbf{(D4)}] $\Var(k(\Ybf,\cdot)\mid \Xbf=\xbf) = \E[\| k(\Ybf,\cdot) \|_{\H}^2 \mid \Xbf=\xbf] - \|\E[ k(\Ybf,\cdot)  \mid \Xbf=\xbf]\|_{\H}^2 > 0$.
\end{enumerate}

Despite \dataass{2} being intuitive, it is somewhat difficult to find an example, solely because closed-form expressions of MMD distances are not readily available. We thus concentrate on the simple case $\Ybf \mid \Xbf=\xbf \sim N(m(\xbf), \mathbf{I})$ some function $m:\R^p \to \R^d$, and show that under a Gaussian kernel, \dataass{2} holds. We thereby make use of an intuitive expression of the MMD distance between two Gaussians. The following result is a special case of Theorem 1 in \citet{ourMMDresult}.

\begin{proposition}
    Consider a Gaussian kernel $k$, with parameter $\sigma=1$ and $\Ybf \mid \Xbf=\xbf \sim N(m(\xbf), \mathbf{I})$, for some function $m:\R^p \to \R^d$. Then
    \begin{align}
               \norm{ \mu(\xbf_1) - \mu(\xbf_2)}_{\H}^2 =2\left(\frac{1}{3}\right)^{d/2} \left( 1 - \exp\left( -\frac{\norm{m(\xbf_2) - m(\xbf_1)}^2}{6} \right)\right)
    \end{align}
\end{proposition}

Since, $1+x \leq \exp(x)$ and thus $1-\exp(-x) \leq x$, it holds that
\begin{align*}
    2\left(\frac{1}{3}\right)^{d/2} \left( 1 - \exp\left( -\frac{\norm{m(\xbf_2) - m(\xbf_1)}^2}{6} \right)\right)\leq   2\left(\frac{1}{3}\right)^{d/2} \frac{\norm{m(\xbf_2) - m(\xbf_1)}^2}{6}.
\end{align*}
Thus, \dataass{2} is met, if the function $m$ is Lipschitz.

We now also show that \dataass{4} holds under the Gaussian kernel if $\mathbf{Y}\mid \Xbf=\xbf$ is not almost surely a constant. To see this, we first note that
\begin{align*}
\Var(k(\Ybf, \cdot) \mid \Xbf=\xbf)=\E[\norm{k(\Ybf, \cdot)- \mu(\xbf)}_{\H}^2 \mid \Xbf= \xbf] =0,   
\end{align*}
if and only if, $\norm{k(\Ybf, \cdot)- \mu(\xbf)]}_{\H}^2=0$ on a set $A$ with $\PYgX(A)=1$. This in turn means that $k(\Ybf, \cdot)$ is constant $\PYgX-$a.s. on $\H$, or,
\begin{align}\label{equivalence}
   \Var(k(\Ybf, \cdot) \mid \Xbf=\xbf)=0 \iff k(\ybf, \cdot) = \mu(\xbf)(\cdot), \text{ for } \ybf \in A, \text{ s.t. } \PYgX(A)=1.
\end{align}
However, the map
\[
\R^d \ni \ybf \mapsto k(\ybf, \cdot) = \exp\left( -\frac{\norm{ \ybf- \cdot}}{2\sigma^2} \right) \in \H,
\]
is injective. Indeed for $\ybf_1 \neq \ybf_2$, it holds
\begin{align*}
    \norm{ k(\ybf_1, \cdot) - k(\ybf_2, \cdot) }_{\H}^2&=\norm{k(\ybf_1, \cdot)}_{\H}^2 + \norm{k(\ybf_2, \cdot)}_{\H}^2 - 2 \langle k(\ybf_1, \cdot), k(\ybf_2, \cdot) \rangle\\
    &=2 - 2 \exp\left( -\frac{\norm{\ybf_1- \ybf_2}}{2\sigma^2} \right) > 0,
\end{align*}
such that $k(\ybf_1, \cdot) \neq k(\ybf_2, \cdot)$ in $\H$. Thus, if $\Ybf$ is not $\PYgX-$a.s. constant, then $k(\Ybf, \cdot)$ is not, and by \eqref{equivalence}, \dataass{4} holds.

\section{DRF Pseudocode}\label{appendix: pseudocode}

\begin{algorithm}[htp]
\caption{Pseudocode for Distributional Random Forest in~\citet{DRF-paper}} \label{pseudocode}
\begin{algorithmic}[1]
\Procedure{BuildForest}{set of samples $\mathcal{S} = \{(\bold{x}_i, \bold{y}_i)\}_{i=1}^n$, number of trees $N$}
    \For{$i=1,\ldots, N$}
        \State $\mathcal{S}_\text{subsample} \gets$  Subsample $\mathcal{S}$ as in \forestass{5}
        \State $\mathcal{S}_\text{build}, \mathcal{S}_\text{populate}\gets$ Split $\mathcal{S}_\text{subsample}$ as in \forestass{1} \Comment{$\mathcal{S}_\text{build}$ to determine tree splits, $\mathcal{S}_\text{populate}$ to populate the leaves}
        \State $\mathcal{N}_i \gets$  initialize root node using $\mathcal{S}_\text{build}$ 
        \State $\mathcal{T}_i \gets$ \textsc{BuildTree}($\mathcal{N}_i$) \Comment{Start recursion from the root node}
                \State Populate leaves with $\mathcal{S}_\text{populate}$
    \EndFor
    \State \textbf{return} $\mathcal{F} = \{\mathcal{T}_1, \ldots, \mathcal{T}_N\}$
\EndProcedure

\item[]

\Procedure{BuildTree}{current node $\mathcal{N}$} \Comment{Recursively constructs the trees} 
    \State $\mathcal{S} \gets$ Extract the samples in the node
    \State $\mathcal{I} \gets$ Random set of candidate variables to perform a split on
    \State $\mathcal{C}$ $\gets$  Initialize list
    \Comment{Here we store info  
    about candidate splits}
    \For{idx $\in \mathcal{I}$, level $l$} \Comment{$l$ iterates over all values of variable $X_{\text{idx}}$}
        \State $\mathcal{S}_L, \mathcal{S}_R \gets$ Split samples based on $(\bold{x}_i)_{i \in \text{idx}} \leq l$ 
        \State $v \gets $ Calculate splitting criterion in~\eqref{eq: splitting} using \textsc{($\mathcal{S}_L, \mathcal{S}_R$)} 
        \State Add ($v$, $\mathcal{S}_L, \mathcal{S}_R, \text{idx}, l$) to $\mathcal{C}$ 
    \EndFor
    \State $\mathcal{S}_L, \mathcal{S}_R, \text{idx}, l \gets$ find the best split in $\mathcal{C}$ 
    \State $\mathcal{N}_L \gets$  Create new node with set of samples $\mathcal{S}_L$
    \State $\mathcal{N}_R \gets$ Create new node with set of samples $\mathcal{S}_R$
    \State \textsc{BuildTree}($\mathcal{N}_L$), \textsc{BuildTree}($\mathcal{N}_R$) \Comment{Proceed building recursively}
    \State \textsc{Children}($\mathcal{N}) \gets \mathcal{N}_L, \mathcal{N}_R$ 
    \State \textsc{Split}($\mathcal{N}) \gets \text{idx}, l$ \Comment{Store the split} 
    \State \textbf{return} \textsc{Children}($\mathcal{N}$), \textsc{Split}($\mathcal{N}$) 
\EndProcedure

\item[]

\Procedure{GetWeights}{forest $\mathcal{F}$, test point $\bold{x}$} \Comment{Computes the weighting function}
    \State vector of weights $w$ = \textsc{Zeros}($n$) \Comment{$n$ is the training set size}
    \For{$i = 1,\ldots,|\mathcal{F}|$}
        \State $\mathcal{L} \gets $ Get indices of training samples in same leaf as $\bold{x}$
        \For{$\text{idx} \in \mathcal{L}$}
            \State $w[\text{idx}] = w[\text{idx}]$ + $1/(|\mathcal{L}|\cdot|\mathcal{F}|)$
        \EndFor
    \EndFor
    \State \textbf{return} $w$
\EndProcedure
\end{algorithmic}
\end{algorithm}

\clearpage
{\small
\bibliography{bibfile}
}

\end{document}